\documentclass[11pt]{amsart}

\usepackage[english]{babel}

\usepackage[T1]{fontenc}
\usepackage[utf8]{inputenc}
\usepackage[a4paper]{geometry}
\usepackage{textcomp}
\usepackage{mathtools}
\usepackage{amsmath}
\usepackage{amsthm}
\usepackage{amssymb}
\PassOptionsToPackage{normalem}{ulem}
\usepackage{ulem}

\usepackage[sc]{mathpazo}
\usepackage{newpxmath}

\usepackage[svgnames]{xcolor}

\newcommand\Span{\operatorname{Span}}
\newcommand\Err{\operatorname{Err}}

\makeatletter

\theoremstyle{plain}
\newtheorem{thm}{\protect\theoremname}
  \theoremstyle{definition}
  \newtheorem{defn}[thm]{\protect\definitionname}
  \theoremstyle{remark}
  \newtheorem{rem}[thm]{\protect\remarkname}
  \theoremstyle{plain}
  \newtheorem{prop}[thm]{\protect\propositionname}
  \theoremstyle{plain}
  \newtheorem{cor}[thm]{\protect\corollaryname}
  \theoremstyle{plain}
  \newtheorem{lem}[thm]{\protect\lemmaname}
  \theoremstyle{remark}
  \newtheorem{claim}[thm]{\protect\claimname}

\numberwithin{equation}{subsection}
\numberwithin{thm}{section}

\usepackage{babel}
\providecommand{\claimname}{Claim}
  \providecommand{\corollaryname}{Corollary}
  \providecommand{\lemmaname}{Lemma}
  \providecommand{\propositionname}{Proposition}
  \providecommand{\remarkname}{Remark}
  \providecommand{\theoremname}{Theorem}
\providecommand{\theoremname}{Theorem}

\usepackage{babel}
\providecommand{\claimname}{Claim}
  \providecommand{\corollaryname}{Corollary}
  \providecommand{\lemmaname}{Lemma}
  \providecommand{\propositionname}{Proposition}
  \providecommand{\remarkname}{Remark}
  \providecommand{\theoremname}{Theorem}
\providecommand{\theoremname}{Theorem}

\usepackage{babel}
\providecommand{\claimname}{Claim}
  \providecommand{\corollaryname}{Corollary}
  \providecommand{\lemmaname}{Lemma}
  \providecommand{\propositionname}{Proposition}
  \providecommand{\remarkname}{Remark}
  \providecommand{\theoremname}{Theorem}
\providecommand{\theoremname}{Theorem}

\usepackage{babel}
\providecommand{\claimname}{Claim}
  \providecommand{\corollaryname}{Corollary}
  \providecommand{\lemmaname}{Lemma}
  \providecommand{\propositionname}{Proposition}
  \providecommand{\remarkname}{Remark}
  \providecommand{\theoremname}{Theorem}
\providecommand{\theoremname}{Theorem}

\usepackage{babel}
\providecommand{\claimname}{Claim}
  \providecommand{\corollaryname}{Corollary}
  \providecommand{\lemmaname}{Lemma}
  \providecommand{\propositionname}{Proposition}
  \providecommand{\remarkname}{Remark}
  \providecommand{\theoremname}{Theorem}
\providecommand{\theoremname}{Theorem}

\usepackage{babel}
\providecommand{\claimname}{Claim}
  \providecommand{\corollaryname}{Corollary}
  \providecommand{\lemmaname}{Lemma}
  \providecommand{\propositionname}{Proposition}
  \providecommand{\remarkname}{Remark}
  \providecommand{\theoremname}{Theorem}
\providecommand{\theoremname}{Theorem}

\usepackage{babel}
\providecommand{\claimname}{Claim}
  \providecommand{\corollaryname}{Corollary}
  \providecommand{\lemmaname}{Lemma}
  \providecommand{\propositionname}{Proposition}
  \providecommand{\remarkname}{Remark}
  \providecommand{\theoremname}{Theorem}
\providecommand{\theoremname}{Theorem}

\makeatother

\usepackage{babel}
  \providecommand{\claimname}{Claim}
  \providecommand{\corollaryname}{Corollary}
  \providecommand{\definitionname}{Definition}
  \providecommand{\lemmaname}{Lemma}
  \providecommand{\propositionname}{Proposition}
  \providecommand{\remarkname}{Remark}
\providecommand{\theoremname}{Theorem}

\allowdisplaybreaks

\makeatletter
\@namedef{subjclassname@2020}{%
  \textup{2020} Mathematics Subject Classification}
\makeatother


\begin{document}

\parindent=5mm

\begin{abstract}
We consider the modified Korteweg-de Vries equation \eqref{1mKdV} and prove that given any sum $P$ of solitons and breathers of \eqref{1mKdV}
(with distinct velocities), there exists a solution  $p$ of \eqref{1mKdV} such
that $p(t) - P(t) \to 0$ when $t\rightarrow+\infty$, which we call
multi-breather. In order to do this, we work at the
$H^{2}$ level (even if usually solitons are considered at the $H^{1}$
level). We will show that this convergence takes place in any $H^{s}$
space and that this convergence is exponentially fast in time.

We also show that the constructed multi-breather
is unique in two cases:  in the class of solutions which converge to the profile $P$ faster than the inverse of a polynomial of a large enough degree in time (we will call this a super polynomial convergence), or (without hypothesis on the convergence rate), when all the velocities are positive.
\end{abstract}

\title{On the uniqueness of multi-breathers of the modified Korteweg-de Vries equation}

\author{Alexander Semenov}

\maketitle

\section{Introduction}

\subsection{Setting of the problem}

We consider the modified Korteweg-de Vries equation on $\mathbb{R}$:
\begin{align} \label{1mKdV}
\tag{mKdV}
\begin{cases}
\begin{array}{c}
u_{t}+(u_{xx}+u^{3})_{x}=0\\
u(0)=u_{0}
\end{array} & \begin{array}{c}
(t,x)\in\mathbb{R}^{2}\\
u(t,x)\in\mathbb{R}
\end{array}\end{cases}
\end{align}

	The \eqref{1mKdV} equation appears as a model
of some physical problems as  plasma physics \cite{key-40,key-42}, electrodynamics \cite{key-43}, fluid mechanics \cite{key-48}, ferromagnetic vortices \cite{key-47}, and more.

In \cite{key-19},  Kenig, Ponce and Vega established local well-posedness in $H^{s}$, for $s\geq \frac{1}{4}$, of
the Cauchy problem for \eqref{1mKdV}, by fixed point argument in $L^p_xL^q_t$ type spaces. Moreover,
if $s>\frac{1}{4}$, the Cauchy problem is globally well posed \cite{key-39}. Recently, Harrop-Griffiths, Killip and Visan \cite{HGKV} proved local well-posedness in $H^s$ for  $s > -1/2$. However, in this paper, we will only use the global well-posedness in $H^2$.

 \eqref{1mKdV} is an integrable equation (like the original Korteweg-de Vries equation)
and thus it has an infinity of conservation laws, see \cite{key-27,key-28}.
We will use three of them (the first two of them are called \emph{mass}
and \emph{energy}; the third is sometimes called \emph{second energy}):
\begin{gather}
M[u](t):=\frac{1}{2}\int_{\mathbb{R}}u^{2}(t,x)\,dx,\label{1eq:-1} \\
E[u](t):=\frac{1}{2}\int_{\mathbb{R}}u_{x}^{2}(t,x)\,dx-\frac{1}{4}\int_{\mathbb{R}}u^{4}(t,x)\,dx,\label{1eq:-2} \quad \text{and}  \\
 F[u](t):=\frac{1}{2}\int_{\mathbb{R}}u_{xx}^{2}(t,x)\,dx-\frac{5}{2}\int_{\mathbb{R}}u^{2}(t,x)u_{x}^{2}(t,x)\,dx+\frac{1}{4}\int_{\mathbb{R}}u^{6}(t,x)\,dx.\label{1eq:-3}
\end{gather}

Observe that if $u$ is a solution of \eqref{1mKdV} then $-u$ and, for any $x_0 \in \mathbb R$, $(t,x) \mapsto u(t,x -x_0)$ are solutions
of \eqref{1mKdV} too.

\smallskip

\eqref{1mKdV} is a dispersive nonlinear equation that is a special case of a
more general class of equations: the general Korteweg-de Vries equation
(gKdV), where the nonlinearity $u^3$ is replaced by $f(u)$ for some real valued function $f$. The particularity of \eqref{1mKdV} in comparison to other (gKdV) equation
is that it admits special non linear solutions, namely breather solutions.

The most simple nonlinear solutions of \eqref{1mKdV} are solitons, i.e.
a bump of a constant shape that translates with a constant velocity
without deformation, that is, solutions of the
form $u(t,x)=Q_{c}(x-ct)$, where $c$ is the velocity and $Q_{c}$
is the profile function that depends only on one variable.  $Q_{c} \in H^1(\mathbb R)$ should solve the elliptic equation:
\begin{align}
Q_{c}''=cQ_{c}-Q_{c}^{3}.\label{1eq:-4}
\end{align}
We can show that necessarily $c>0$  and that, if $c>0$, (\ref{1eq:-4})
has a unique solution in $H^1(\mathbb R)$, up to translations and reflexion with respect to the $x$-axis. Actually, one has the explicit formula:
\begin{align}
Q_{c}(x):=\Bigg(\frac{2c}{\cosh^{2}\big(c^{1/2}x\big)}\Bigg)^{\frac{1}{2}}. \label{1eq:gr}
\end{align}
Observe that we chose $Q_c$ so that it is even and positive. 

A \emph{soliton} is a solution
of \eqref{1mKdV}, parameterized by a velocity parameter $c>0$, a sign
parameter $\kappa\in\{-1,1\}$ and a translation parameter $x_{0}\in\mathbb{R}$
(it corresponds to the initial position of the soliton) that has the
following expression: 
\begin{align}
R_{c,\kappa}(t,x;x_{0}):=\kappa Q_{c}(x-x_{0}-ct).\label{1eq:soliton}
\end{align}

When $\kappa=-1$, this object is sometimes called \emph{antisoliton}. Notice that solitons are smooth and decaying. The generalized Korteweg-de Vries equation (gKdV) also admit soliton type solutions, and the focusing nonlinear Schrödinger equation (NLS) as well.
Solitons have been extensively studied
, in particular their stability.
 Cazenave, Lions and Weinstein
in \cite{key-6,key-9,key-10,key-22} were interested in orbital stability
of (gKdV) and (NLS) solitons in $H^1$. A soliton of \eqref{1mKdV} is indeed
orbitally stable, i.e. if a solution is initially close to a soliton in $H^1(\mathbb R)$,
then it stays close to the soliton, up to a space translation defined for any time,  in $H^1(\mathbb R)$. General
results about orbital stability of nonlinear dispersive solitons are
presented by Grillakis, Shatah and Strauss in \cite{key-13}.
The result about orbital stability of a soliton can be improved in
a result of asymptotic stability, as it was done in the works by Martel and Merle \cite{key-20,key-21,key-23}, see also \cite{key-29}.

\smallskip

A \emph{breather} is a solution
of \eqref{1mKdV}, parameterized by $\alpha,\beta>0$, $x_{1},x_{2}\in\mathbb{R}$
that has the following expression: 
\begin{align}
B_{\alpha,\beta}(t,x;x_{1},x_{2}):=2\sqrt{2}\partial_{x} \bigg[ \arctan \bigg( \frac{\beta}{\alpha}\frac{\sin(\alpha y_{1})}{\cosh(\beta y_{2})} \bigg) \bigg], \label{1eq:}
\end{align}
where 
\begin{gather*} 
y_{1}:=x+\delta t+x_{1} \quad \text{and} \quad y_{2}:=x+\gamma t+x_{2}, \\
\text{with} \quad \delta:=\alpha^{2}-3\beta^{2} \quad \text{and} \quad \gamma:=3\alpha^{2}-\beta^{2}.
\end{gather*}

It corresponds to a localized
periodic in time function (with frequency $\alpha$, and exponential localization with decay rate $\beta$) that propagates at a constant velocity $-\gamma$ in
time. Like solitons, breathers are smooth and decaying in space. Unlike solitons,
breather's velocities can be positive, zero or negative. $\alpha,\beta$
are the shape parameters and $x_{1},x_{2}$ are the translation parameters
of a breather. Note that if we replace the parameter $x_{1}$ by $x_{1}+\frac{\pi}{\alpha}$,
we transform $B_{\alpha,\beta}(\cdot,\cdot;x_{1},x_{2})$ in $-B_{\alpha,\beta}(\cdot,\cdot;x_{1},x_{2})$
(therefore, we do not need to talk about ``antibreathers'').

Breathers were first introduced by Wadati in \cite{key-31}, and they
were already used by Kenig, Ponce and Vega in \cite{key-45} to prove
that the flowmap associated to \eqref{1mKdV} equation is \emph{not} uniformly continuous in $H^{s}$
for $s<\frac{1}{4}$ : the point is that two breathers with close velocities
can be very close at $t=0$ and can separate as fast as we want in $H^{s}$ with $s<\frac{1}{4}$, if
$\alpha$ is taken large enough.

\eqref{1mKdV} breathers and their properties, as well as breathers for
other equations, are well studied by Alejo and Muñoz and co-authors in \cite{key-1,key-15,key-14,key-17,key-18}.

Let us singularize a result of $H^{2}$ orbital stability for breathers established in \cite{key-1}, and improved to $H^{1}$ orbital stability
in \cite{key-18}. In this last paper, a partial result of asymptotic stability is also given, for breathers traveling to the right only, with positive velocity $-\gamma >0$; asymptotic stability for breathers in full generality is still an open problem. 

When $\alpha\rightarrow0$, $B_{\alpha,\beta}$ tends to a solution
of \eqref{1mKdV} called \emph{double-pole solution} \cite{key-37}, the
methods employed in this article as well as the proof of orbital stability
made by Alejo and Muñoz seem not to apply for this limit, which is expected to be unstable according to the numerical computations in \cite{key-46}.

\smallskip

An important result regarding the long time dynamics of \eqref{1mKdV} is the soliton-breather resolution \cite{key-25}:  it asserts that any generic solution can be approached by a sum of solitons and breathers when $t\rightarrow+\infty$ (up to a dispersive and a self-similar term). Together with their stability properties, the soliton-breather resolution shows why solitons and breathers are essential objects to study.
This resolution was established for initial conditions in a weighted Sobolev space in \cite{key-25} (see also Schuur \cite{key-34})
by inverse scattering method; see also \cite{key-34} for the soliton resolution for (KdV). Observe that \eqref{1mKdV} breathers do not decouple into simple solitons for large time (it is a \emph{fully bounded state} as it is called in \cite{key-1}); therefore, it
must appear in the resolution. The soliton-breather resolution is one
of the motivations of the study of multi-breathers, which we define below.

There are works in the literature about a more complicated object obtained
from several solitons: a \emph{multi-soliton}. A multi-soliton is
a solution $r(t)$ of \eqref{1mKdV} such that there exists $0<c_{1}<c_{2}<...<c_{N}$, $\kappa_{1},...,\kappa_{N}\in\{-1,1\}$
and $x_{1},...,x_{N}\in\mathbb{R}$, such that 
\begin{align}
\lim_{t\rightarrow+\infty}\Bigg\lVert r(t)-\sum_{j=1}^{N}R_{c_{j},\kappa_{j}}(t,\cdot;x_{j})\Bigg\rVert _{H^{1}(\mathbb{R})}=0.\label{1eq:ms}
\end{align}

This definition is not specific to \eqref{1mKdV} and makes sense for many other nonlinear dispersive PDEs as soon as they admit solitons. This object is introduced by Schuur \cite{key-34} and Lamb \cite{key-35}, see also Miura \cite{key-36}, where explicit formulas are given: these were obtained by inverse scattering method thanks to the integrability of the equation. Multi-solitons were first constructed in a non integrable context by Merle \cite{key-50} for the mass critical (NLS). Martel \cite{key-2} constructed multi-solitons for mass-subcritical  and critical (gKdV) equations and proved that they are unique in $H^1(\mathbb R)$, smooth and converge exponentially fast to their profile in any Sobolev space $H^{s}$. Similar studies were done for other nonlinear dispersive PDEs. Martel and Merle \cite{key-3} have proved the existence of multi-solitons for (NLS) in $H^1$, Côte, Martel and Merle extended this construction to mass supercritical (gKdV) and (NLS) in \cite{key-12}.  Friederich and Côte in \cite{key-61} proved smoothness, and uniqueness in a class of algebraic convergence. Côte and Muñoz constructed in \cite{key-16} multi-solitons for the nonlinear Klein-Gordon equation. Ming, Rousset and Tzvetkov have constructed multi-solitons for the water-waves systems in \cite{key-66}. Valet has proved in \cite{key-60} the existence and uniqueness of multi-solitons  in $H^1$ for the Zakharov-Kuznetsov equation, which generalizes (gKdV) to higher dimension.

\subsection{Main results}

We prove in this article that given any sum of solitons and breathers
with distinct velocities, there exists a solution of \eqref{1mKdV} whose difference with this sum tends to zero when time goes to infinity. This solution will
be called a multi-breather. Let us make the definition more precise.

Let $J\in\mathbb{N}$ and $K,L\in\mathbb{N}$ such that $J=K+L$.
We will consider a set of $L$ solitons and $K$ breathers:
\begin{itemize}
\item the breather parameters are $\alpha_{k}>0$, $\beta_{k}>0$, $x_{1,k}^{0}\in\mathbb{R}$
and $x_{2,k}^{0}\in\mathbb{R}$ for $1\leq  k\leq  K$. 
\item the solitons parameters are $c_{l}>0$, $\kappa_{l}\in\{-1,1\}$ and $x_{0,l}^{0}\in\mathbb{R}$
for $1\leq  l\leq  L$. 
\end{itemize}
We define for $1\leq  k\leq  K$, the $k$th breather:
\begin{align}
B_{k}(t,x):=B_{\alpha_{k},\beta_{k}}(t,x;x_{1,k}^{0},x_{2,k}^{0}); \label{1eq:br}
\end{align}
and for $1\leq  l\leq  L$, the $l$th soliton:
\begin{align}
R_{l}(t,x):=R_{c_{l},\kappa_{l}}(t,x;x_{0,l}^{0}).\label{1eq:sol}
\end{align}

We now define the \emph{velocity} of our objects. Recall that for $1\leq  k\leq  K$, the velocity of $B_{k}$ is 
\begin{align}
v_{k}^{b}:=-\gamma_{k}=\beta_{k}^{2}-3\alpha_{k}^{2}, \label{1eq:vb}
\end{align}
and for $1\leq  l\leq  L$, the velocity of $R_{l}$ is 
\begin{align}
v_{l}^{s}:=c_{l}.\label{1eq:vs}
\end{align}

The most important assumption we make is that all these velocities are distinct:
\begin{align}
\forall k\neq k'\quad v_{k}^{b}\neq v_{k'}^{b},\qquad\forall l\neq l'\quad v_{l}^{s}\neq v_{l'}^{s},\qquad\forall k,l\quad v_{k}^{b}\neq v_{l}^{s}.\label{1eq:diff}
\end{align}
These implies for any two of these objects to be far from each other when
time is large, and this assumption is essential in our analysis.

It will be useful to order our breathers and solitons by increasing velocities. As these are distinct, we can define an increasing function:
\begin{align}
\underline{v}:\{1,...,J\}\longrightarrow\{v_{k}^{b},1\leq  k\leq  K\}\cup\{v_{l}^{s},1\leq  l\leq  L\}.\label{1eq:v}
\end{align}
The set $\{v_{1},...,v_{J}\}$ is thus the (ordered) set of all possible velocities
of our objects. We define $P_{j}$, for $1\leq  j\leq  J$, as the object
(either a soliton $R_{l}$ or a breather $B_{k}$) that corresponds
to the velocity $v_{j}$. Hence, $P_{1},...,P_{J}$ are the considered objects ordered by
increasing velocity. 

We will need both notations: the indexation by $k$ and $l$, and
the indexation by $j$, and we will keep these notations to avoid ambiguity.

We will denote by $x_{j}$ the center of mass of $P_{j}$, that is
\begin{itemize}
\item if $P_{j}=B_k$ is a breather, we set $x_{j}(t):=-x_{2,k}^{0}+v_{j}t$;
\item if $P_{j}=R_l$ is a soliton, we set $x_{j}(t):=x_{0,l}^{0}+v_{j}t$.
\end{itemize}

We denote: 
\begin{align}
R=\sum_{l=1}^{L}R_{l},\quad B=\sum_{k=1}^{K}B_{k},\quad P=R+B=\sum_{j=1}^{J}P_{j}.\label{1eq:sum}
\end{align}

We can now define a multi-breather: as solitons are objects which can be studied naturally in $H^1(\mathbb R)$, it turns out that breathers are best studied in $H^2(\mathbb R)$; therefore, it is in this latter space that we develop our analysis.

\begin{defn}
\label{1def:MB}A \emph{multi-breather} associated to the sum $P$ given in (\ref{1eq:sum})
of solitons and breathers is a solution $p\in \mathcal C([T^{*},+\infty),H^{2}(\mathbb{R}))$,
for a constant $T^{*}>0$, of \eqref{1mKdV} such that
\begin{align}
\lim_{t\rightarrow+\infty}\lVert p(t)-P(t)\rVert _{H^{2}}=0.\label{1eq:mb}
\end{align}
\end{defn}

We will prove two results in this article. The first one is the existence
and the regularity of a multi-breather, the second one is the uniqueness
of a multi-breather. The uniqueness is established in two settings: in the case when all velocities are positive, and without any assumption on the sign of the considered velocities. However, in the last case, the uniqueness is obtained in a narrower class of functions. 

\begin{thm}
\label{1thm:MAIN}Given solitons and breathers (\ref{1eq:br}), (\ref{1eq:sol}) whose velocities
(\ref{1eq:vb}) and (\ref{1eq:vs}) satisfy (\ref{1eq:diff}),
there exists a multi-breather $p$ associated to $P$ given in (\ref{1eq:sum}).
Moreover,
\begin{align*}
p\in\mathcal{C}^{\infty}(\mathbb{R}\times\mathbb{R})\cap\mathcal{C}^{\infty}(\mathbb{R},H^{s}(\mathbb{R}))
\end{align*} 
for any $s\geq 0$ and there exists $\theta>0$ such that for any $s\geq 0$,
there exists $A_{s}\geq 1$ and $T^{*}>0$ such that
\begin{align}
\forall t\geq  T^{*},\quad\lVert p(t)-P(t)\rVert_{H^{s}}\leq  A_{s}e^{-\theta t}.\label{1eq:-12-1-1}
\end{align}
\end{thm}

\begin{rem} \label{1rm:3}
We will also show that $\theta$ does only depend on the shape parameters
of our objects: $\alpha_{k},\beta_{k},c_{l}$. Moreover, if there
exists $D>0$ such that for all $j\geq 2$, $x_{j}(0)\geq  x_{j-1}(0)+D$,
then $A_{s}$ and $T^{*}$ do not depend on $x_{1,k}^{0},x_{2,k}^{0},x_{0,l}^{0}$
but only on $\alpha_{k},\beta_{k},c_{l}$ and $D$. Finally,
if $D>0$ is large enough with respect to the problem data,
then (\ref{1eq:-12-1-1}) is true for $T^{*}=0$. See Section \ref{1sec:2.5} for further details.
\end{rem}

\begin{thm} \label{1thm:uniqueness}
Given the same set of solitons and breathers as in Theorem \ref{1thm:MAIN} whose velocities
satisfy (\ref{1eq:diff}) and $v_{1}>0$ (so that all the velocities are positive), the multi-breather $p$ associated
to $P$ by Theorem \ref{1thm:MAIN}, in the sense of Definition \ref{1def:MB},
is unique.
\end{thm}


\begin{prop}
\label{1lem:polyn}
Given the same set of solitons and breathers as in Theorem \ref{1thm:MAIN} whose velocities
satisfy (\ref{1eq:diff}), there exists $N>0$ large enough such that the multi-breather $p$ associated
to $P$ by Theorem \ref{1thm:MAIN} is the unique solution $u \in \mathcal{C}([T_0,+\infty),H^2(\mathbb{R}))$ of \eqref{1mKdV} such that
\begin{align}
\lVert u(t)-P(t) \rVert_{H^2} = O \bigg( \frac{1}{t^N} \bigg), \qquad \text{as } \ t \rightarrow +\infty.
\end{align}
\end{prop}

In \cite{key-37}, there exists a formula for a multi-breather, obtained
by inverse scattering method, that in some sense already gives the
existence of a multi-breather. 
However, the proof of the theorem \ref{1thm:MAIN} from this formula is rather involved.

In this paper, we give here a different approach to prove the existence
of a multi-breather and we clearly show that we have convergence of
the constructed multi-breather to the corresponding sum of solitons
and breathers in $H^{s}$, that this convergence is exponentially fast in time and that the constructed multi-breather is smooth. To do this, we use the variational structure of solitons and breathers. This is why, we give a proof that is potentially generalizable to non integrable equations, and that uses similar type of techniques as in the proof of the uniqueness (the latter cannot be deduced from the formula).
In any case, uniqueness of multi-breathers is new.

\smallskip

In this paper, we adapt the arguments given by Martel and Merle \cite{key-3}, by Martel \cite{key-2} and by Côte and Friederich \cite{key-61}
to the context of breathers. To do so, one needs to understand the
variational structure of breathers, in the same fashion as Weinstein
did in \cite{key-6} for  (NLS) solitons. Such results were obtained by Alejo and Muñoz in
\cite{key-1}: a breather is a critical point of a Lyapunov functional
at the $H^{2}$ level, whose Hessian is coercive up to several (but finitely many) orthogonal conditions, see Section \ref{1sec:constr} for details.
As we see from \cite{key-1}, the $H^{2}$ regularity level is the most natural
setting to study breathers, and the $H^{1}$ regularity level
is natural for the study of solitons (as we see  in \cite{key-2,key-3}). One important issue we face is therefore to understand soliton variational structure at $H^2$ level, and to adapt the Lyapunov functional in \cite{key-1} to accommodate for a sum of breathers (and solitons). Notice that arguments based on monotonicity may be adapted only if we suppose that all the considered velocities are positive. Because \cite{key-3,key-61} are not based on monotonicity (these are results for (NLS) which is not well suited for monotonicity), we can adapt their arguments to obtain existence and uniqueness results for our case without any condition on the sign of velocities. The uniqueness result obtained in this setting is however weaker than the one that is obtained with monotonicity arguments.

\subsection{Outline of the proof}

The proof of Theorem \ref{1thm:MAIN} (the existence of multi-breathers) is split into two main parts: the construction of an $H^2$ multi-breather and the proof that this multi-breather is smooth.

\subsubsection{An $H^2$ multi-breather}

Let us start with the first part. We consider an increasing sequence $(T_n)$ of $\mathbb{R}_+$ with
$T_{n}\rightarrow+\infty$, and for $n\in\mathbb{N}$, let $p_{n}$
the unique global $H^{2}$ solution of \eqref{1mKdV} such that $p_{n}(T_{n})=P(T_{n})$
(recall that the Cauchy problem for \eqref{1mKdV} is globally well-posed
in $H^{2}$).

We will prove the following \emph{uniform estimate}: 
\begin{prop}
\label{1prop:uest} There exists $T^{*}>0$, $A>0$, $\theta>0$ such
that, for any $n\in\mathbb{N}$ such that $T_{n}\geq  T^{*}$,
\begin{align}
\forall t\in[T^{*},T_{n}],\quad\lVert p_{n}(t)-P(t)\rVert_{H^{2}}\leq  Ae^{-\theta t}.\label{1eq:result}
\end{align}
\end{prop}

With this proposition in hand, we can construct an $H^2$ multi-breather which converges exponentially fast to its profile,  which is the first part of Theorem \ref{1thm:MAIN}, as stated below.

\begin{prop}
\label{1thm:main} There exists $T^{*}\in\mathbb{R}$, $A>0$, $\theta>0$
and a solution $p\in C([T^{*},+\infty),H^{2}(\mathbb{R}))$ of \eqref{1mKdV}
such that
\begin{align}
\forall t\geq  T^{*},\quad\lVert p(t)-P(t)\rVert_{H^{2}}\leq  Ae^{-\theta t}.\label{1eq:-12}
\end{align}
\end{prop}


\begin{proof}[Proof of Proposition \ref{1thm:main} assuming Proposition \ref{1prop:uest}]
We show that the sequence $\big(p_{n}(T^{*})\big)$ is
$L^{2}$-compact, in the following sense: 

\begin{lem}
\label{1lem:uest}For any $\varepsilon>0$, there exists $R>0$ such that 
\begin{align}
\forall n\in\mathbb{N}\quad\int_{\lvert x\rvert>R}p_{n}^{2}(T^{*},x)\,dx<\varepsilon.\label{1eq:compact}
\end{align}
\end{lem}

An analogous lemma has already been proved on p. 1111 of \cite{key-2},
which is the proof of formula (14) (and can also be found in \cite{key-3}).
The same proof works here. We need to use Proposition \ref{1prop:uest}
for $T_{n}$ large enough and then make a time variation to obtain
the result in $T^{*}$. We can first find $R$ that works for $P^2(t_0)$ instead of $p_n^2(T^*)$ for a fixed $t_0>T^*$ large enough. From Proposition \ref{1prop:uest}, we see that if we take $t_0$ large enough, we obtain the desired lemma for $p_n^2(t_0)$ instead of $p_n^2(T^*)$. To finish, with the help of a cut-off function, we control time variations of $\int_{\lvert x\rvert >R}p_n^2(t)\,dx$, where $R$ is taken larger if needed. This is why, we obtain the result at $t=T^*$.

\smallskip

As a consequence of the Proposition \ref{1prop:uest} above,
$(\lVert p_{n}(T^{*})\rVert_{H^{2}})$ is a bounded sequence. Thus,
there exists $p^{*}\in H^{2}(\mathbb{R})$ such that, up to a subsequence, 
\begin{align}
p_{n}(T^{*})\rightharpoonup p^{*}\quad \text{in }  H^{2}.\label{1eq:H2weak}
\end{align}
Thus, from Lemma \ref{1lem:uest}, there holds the strong convergence:
\begin{align}
p_{n}(T^{*})\rightarrow p^{*}\quad \text{in }  L^{2}.\label{1eq:H2L2}
\end{align}
Therefore, we obtain by interpolation: 
\begin{align}
p_{n}(T^{*})\rightarrow p^{*}\quad \text{in }  H^{1}.\label{1eq:interpol}
\end{align}

Now, let us consider the global $H^{1}$ (even $H^{2}$) solution
$p$ of \eqref{1mKdV} such that $p(T^{*})=p^{*}$. As shown in \cite{key-2},
the Cauchy problem for \eqref{1mKdV} has a continuous dependence in $H^{1}$
on compact sets of time. Let $t\geq  T^{*}$. By continuous dependence,
we deduce that $p_{n}(t)\rightarrow p(t)$ in $H^{1}$. $(p_{n}(t)-P(t))$
is a bounded sequence in $H^{2}$, which admits a unique weak limit and so
\begin{align}
 p_{n}(t)-P(t)\rightharpoonup p(t)-P(t) \quad \text{ in }H^{2}.\label{1eq:-13}
\end{align}

By weak convergence and from Proposition \ref{1prop:uest},
we obtain:
\begin{align}
\lVert p(t)-P(t)\rVert_{H^{2}}\leq \liminf_{n \to +\infty} \lVert  p_{n}(t)-P(t) \rVert_{H^2} \leq  Ae^{-\theta t}.\label{1eq:-14}
\end{align}
As this is true for any $t \geq T^*$. This ends the proof of the Proposition \ref{1thm:main}.
\end{proof}

It remains to prove Proposition \ref{1prop:uest}, for which we rest on
a bootstrap argument. More precisely, we will reduce the proof to
the following proposition: 
\begin{prop}
\label{1prop:bootstrap}There exists $T^{*}>0$, $A>0$, $\theta>0$,
such that for any $n\in\mathbb{N}$ such that $T_{n}\geq  T^{*}$,
for any $t^{*}\in[T^*,T_{n}]$, if 
\begin{align}
\forall t\in[t^{*},T_{n}],\quad\lVert p_{n}(t)-P(t)\rVert_{H^{2}}\leq  Ae^{-\theta t},\label{1eq:assumpb}
\end{align}
then 
\begin{align}
\forall t\in[t^{*},T_{n}],\quad\lVert p_{n}(t)-P(t)\rVert_{H^{2}}\leq \frac{A}{2}e^{-\theta t}.\label{1eq:conclb}
\end{align}
\end{prop}

The proof of Proposotion \ref{1prop:uest} then follows from a simple continuity argument.

\begin{proof}[Proof of Proposition \ref{1prop:uest} assuming Proposition \ref{1prop:bootstrap}]
We define $t_{n}^{*}$ in the following way: 
\begin{align}
t_{n}^{*}:=\inf\{t^{*}\in[T^{*},T_{n}),\quad\forall t\in[t^{*},T_{n}],\quad\lVert p_{n}(t)-P(t)\rVert_{H^{2}}\leq  Ae^{-\theta t}\}.\label{1eq:-16}
\end{align}
The map $t\mapsto\lVert p_{n}(t)-P(t)\rVert_{H^{2}}$ is a continuous
function and $\lVert p_{n}(T_{n})-P(T_{n})\rVert_{H^{2}}=0$. This means
that there exists $T^{*}\leq  t^{*}<T_{n}$ such that 
\begin{align}
\forall t\in[t^{*},T_{n}],\quad\lVert p_{n}(t)-P(t)\rVert_{H^{2}}\leq  Ae^{-\theta t}.\label{1eq:-15}
\end{align}
Therefore, we have that
\begin{align}
T^{*}\leq  t_{n}^{*}<T_{n}.\label{1eq:-17}
\end{align}

We would like to prove that $t_{n}^{*}=T^{*}$. Let us argue by contradiction
and assume that $t_{n}^{*}>T^{*}$. The Proposition \ref{1prop:bootstrap} allows us
to deduce that 
\begin{align}
\forall t\in[t_{n}^{*},T_{n}],\quad\lVert p_{n}(t)-P(t)\rVert_{H^{2}}\leq \frac{A}{2}e^{-\theta t}.\label{1eq:-18}
\end{align}
This means that 
\begin{align}
\lVert p_{n}(t_{n}^{*})-P(t_{n}^{*})\rVert_{H^{2}}\leq \frac{A}{2}e^{-\theta t_{n}^{*}},\label{1eq:-19}
\end{align}
which means that $t_{n}^{*}$ could be chosen smaller, by continuity.
This is a contradiction.
\end{proof}

Hence, we are left to prove Proposition \ref{1prop:bootstrap}, which will be done in Section \ref{1sec:constr}.

\subsubsection{The $H^2$ multi-breather is smooth}

We now turn to the second part of Theorem \ref{1thm:MAIN}, which is strongly adapted from \cite{key-2}. The heart of this part is to prove uniform estimates in $H^s$ for $p_n - P$, for any $s \geq 0$:

\begin{prop}
\label{1prop:higher}There exists $T^{*}>0$, $\theta>0$, such that
for any $s\geq 0$, there exists $A_{s}\geq 1$ such that for any $n\in\mathbb{N}$
such that $T_{n}\geq  T^{*}$,
\begin{align}
\forall t\in[T^{*},T_{n}],\quad\lVert p_{n}(t)-P(t)\rVert_{H^{s}}\leq  A_{s}e^{-\theta t}.\label{1eq:higher}
\end{align}
\end{prop}

With this improved version of Proposition \ref{1prop:uest}, one can prove by the same reasonning as in the proof of the Proposition \ref{1thm:main}, that for any $s \geq 0$, $p$ actually belongs to $L^\infty([T^*,+\infty), H^s(\mathbb R))$ and that the convergence of $p(t) - P(t)$ occurs in $H^s$ with an exponential decay rate. More precisely, 

\begin{thm}
For any $s \geq 2$, we have that $p\in \mathcal C([T^{*},+\infty),H^{s}(\mathbb{R}))$, and furthermore,
\begin{align}
\forall t\geq  T^{*},\quad\lVert p(t)-P(t)\rVert_{H^{s}}\leq  A_{s}e^{-\theta t}.\label{1eq:-12-1-1-1}
\end{align}
\end{thm}

It remains to prove Proposition \ref{1prop:higher}, which will be done in Section \ref{1sec:smooth}.

\subsubsection{The uniqueness result}

We denote $p$ the multi-breather constructed in the previous sections,
the existence of which is established. Let $u$ be a solution of \eqref{1mKdV}
such that 
\begin{align}
\lVert u- P\rVert_{H^{2}}\rightarrow_{t\rightarrow+\infty} 0.\label{1eq:-111}
\end{align}
Equivalently, there holds:
\begin{align}
\lVert u-p\rVert_{H^{2}}\rightarrow_{t\rightarrow+\infty} 0.\label{1eq:-112}
\end{align}

We denote
\begin{align}
z:=u-p.\label{1eq:-113}
\end{align}
The goal is to prove that $z=0$. We prove it in two configurations: when all the velocities are positive (Theorem \ref{1thm:uniqueness}), and without any assumption on velocities (Proposition \ref{1lem:polyn}), but in this last case we need to assume a stronger convergence than given in (\ref{1eq:-111}).
 
The proof of Theorem \ref{1thm:uniqueness} will be
carried out in two steps.

\smallskip

We start with Proposition \ref{1lem:polyn}, which is adapted from \cite{key-61}. For this, we do not study $u-P$ anymore,
we deal only with $z=u-p$. $z$ is the difference of two solutions of \eqref{1mKdV}, which is much more precise than $u-P$. Thus, we do not modulate parameters of the solitons, as it is needed in other parts of the proof in order to deal with the soliton part of the linear part of the Lyapunov functional, and we avoid some difficulty. In order to prove our inequalities,
we need again to use coercivity of the same type of quadratic forms.
In order to do this, we replace $z$ by $\widetilde{z}=z+\sum_{j=1}^{J} c_j K_j$,
where $K_{j}, \ j=1,...,J$ is a well chosen basis of  the kernel of the quadratic form,
in order to have $\widetilde{z}$ orthogonal to any $K_{j}$. A important idea is to use slow variations of localized functionals with adapted cut-off functions of the form $\varphi \big( \frac{x- vt}{\delta t} \big)$, which provides an extra $O(1/t)$ decay when derivatives fall on the cut-off, and ultimately explain why algebraic decay comes into play.

\smallskip

In the context of Theorem \ref{1thm:uniqueness}, we actually prove that 
\begin{align}
v:=u-P
\end{align}
converges exponentially fast to 0: this is the purpose of Proposition \ref{1prop:conv_exp},  which uses some ideas of \cite{key-2}. Due to Proposition \ref{1lem:polyn}, we deduce immediately from there that an exponential convergence is trivial, that is $z=0$. 

To prove Proposition \ref{1prop:conv_exp}, we use monotonicity properties combined with coercivity of an energy type functional very similar to that used for the existence result. This is why, we also need to modulate, and the choice of the orthogonality condition is essential: it allows
to bound linear terms in $w$ that appear in the computations. An issue  of the mixed breathers/solitons context is that one cannot build a functional adapted to all the nonlinear objects at once, as it is done in \cite{key-2}. Instead, we carry out an induction
and we argue successively around each object, soliton or breather, separately.


\smallskip

\subsubsection{Organisation of the paper}

Sections \ref{1sec:constr} and \ref{1sec:smooth} are devoted to the proof of the existence of a
multi-breather: Proposition  \ref{1prop:bootstrap} is proved in Section \ref{1sec:constr}, Proposition \ref{1prop:higher} is proved in Section \ref{1sec:smooth}. 
Section \ref{1sec:uniq} gathers the proofs of the uniqueness results: Section \ref{1sec:super-pol} is devoted to the proof of Proposition \ref{1lem:polyn}, and Sections \ref{1sec:uniq_mon} and \ref{1sec:uniq_ccl} are devoted to  the proof of Theorem \ref{1thm:uniqueness}.


\subsection{Acknowledgments}
The author would like to thank his supervisor Raphaël Côte for suggesting the idea of this work, for fruitful discussions and for his useful advice.

\section{Construction of a multi-breather in $H^2(\mathbb{R})$}
\label{1sec:constr}

We set 
\begin{align}
\beta & :=\min\{\beta_{k},1\leq  k\leq  K\}\cup\{\sqrt{c_{l}},1\leq  l\leq  L\},\label{1eq:beta-tau}\\
\tau & :=\min\{v_{j+1}-v_{j},1\leq  j\leq  J-1\}.
\end{align}

Our goal in this section is to prove Proposition \ref{1prop:bootstrap}.

\subsection{Elementary results}

Let us first collect a few basic facts that will be used throughout
the article. One may check an exponential decay result for any of
our objects: 
\begin{prop}
\label{1prop:decay}Let $j=1,...,J$, $n,m\in\mathbb{N}$. Then, there
exists a constant $C>0$ such that for any $t,x\in\mathbb{R}$, 
\begin{align}
\lvert\partial_{x}^{n}\partial_{t}^{m}P_{j}(t,x)\rvert\leq  Ce^{-\beta\lvert x-v_{j}t\rvert}.\label{1eq:expd}
\end{align}
\end{prop}

\begin{cor}
\label{1cor:unifdecay}Let $r>0$. For $t,x$ such that $v_{j}t+r<x<v_{j+1}t-r$,
we have that
\begin{align}
\lvert P(t,x)\rvert\leq  Ce^{-\beta r}.\label{1eq:unifdecay}
\end{align}
The same is true for any space or time derivative of $P$. 
\end{cor}

We will also use the following cross-product result: 
\begin{prop}
\label{1prop:cross-product}Let $i\neq j\in\{1,...,J\}$ and $m,n\in\mathbb{N}$.
There exists a constant $C$ that depends only on $P$, such that
for any $t\in\mathbb{R}$, 
\begin{align}
\bigg\lvert\int\partial_{x}^{m}P_{i}\partial_{x}^{n}P_{j}\bigg\rvert\leq  Ce^{-\beta\tau t/2}.\label{1eq:cross}
\end{align}
\end{prop}

There is also an orthogonality result for breathers that will be useful:
\begin{lem}
\label{1lem:orthob}Let $B:=B_{\alpha,\beta}$ be a breather. We denote
$B_{1}:=\partial_{x_{1}}B$ and $B_{2}:=\partial_{x_{2}}B$. Then,
\begin{align}
\int BB_{1}=\int BB_{2}=0.\label{1eq:orthob}
\end{align}
\end{lem}

\begin{proof}
Note that $\Span(B_{1},B_{2})=\Span(B_{x},B_{t})$. Therefore, it is enough
to prove that 
\begin{align}
\int BB_{x}=\int BB_{t}=0.\label{1eq:orthobs}
\end{align}
Firstly, 
\begin{align}
\int BB_{x}=\frac{1}{2}\int\Big(B^{2}\Big)_{x}=0.\label{1eq:orthobx}
\end{align}
Secondly, 
\begin{align}
\int BB_{t}=\frac{1}{2}\int\Big(B^{2}\Big)_{t}=\frac{1}{2}\frac{d}{dt}\int B^{2}=0,\label{1eq:orthobt}
\end{align}
by mass conservation and because a breather is a solution of \eqref{1mKdV}. 
\end{proof}

\subsection{Almost-conservation of localized conservation laws}
\label{1sec:2.2}
From now on, we will fix $n\in\mathbb{N}$. This is why, for the simplicity
of notations, we can write $T$ for $T_{n}$, and $p$ for $p_{n}$.
The goal will be to find constants $T^{*},A>1,\theta$ that do not
depend on $n$, nor on the translation parameters of the given objects,
and that will be chosen later ($T^{*}$ will depend on $A$ and $\theta$), such
that Proposition \ref{1prop:bootstrap} is verified. We will take 
$t^{*}\in[T^{*},T]$, and we will make the following bootstrap assumption
for the remaining of the article: 
\begin{align}
\forall t\in[t^{*},T],\quad\lVert p(t)-P(t)\rVert_{H^{2}}\leq  Ae^{-\theta t},\label{1eq:-28}
\end{align}
where $p(T)=P(T)$. 
\begin{rem}
\label{1rem:bound}We have the following property for solutions of
\eqref{1mKdV}: there exists $C_{0}>0$ such that for any solution $w$
of \eqref{1mKdV}, $w$ is global and 
\begin{align}
\forall t\in\mathbb{R},\quad\lVert w(t)\rVert_{H^{2}}\leq  C_{0}\lVert w(T)\rVert_{H^{2}}.\label{1eq:global}
\end{align}
Therefore, 
\begin{align}
\forall t\in\mathbb{R},\quad\lVert p(t)\rVert_{H^{2}}\leq  C_{0}\lVert P(T)\rVert_{H^{2}}\leq  C_{0}\sum_{j=1}^{J}\lVert P_{j}(T)\rVert_{H^{2}}\leq  C_{0}C,\label{1eq:gb}
\end{align}
where $C$ is a constant that depends only on the problem data (because
the $H^{s}$-norm of solitons or breathers can be easily bounded). 
\end{rem}

Let $\theta:=\frac{\beta\tau}{32}$. Let $\min(1,\frac{\tau}{4})>\delta>0$
be a constant to be chosen later.

This part of the proof is adapted from \cite{key-3}.

Let $\psi(x)$ be a $C^{3}$ function such that 
\begin{align}
0\leq \psi\leq 1\quad\textrm{on }\mathbb{R}, & \quad\psi'\geq 0\quad\textrm{on }\mathbb{R},\label{1eq:psi-1}\\
\psi(x)=0\quad\textrm{for }x\leq -1, & \quad\psi(x)=1\quad\textrm{for }x\geq 1,
\end{align}
and satisfying, for a constant $C>0$, for any $x\in\mathbb{R}$,
\begin{align}
(\psi'(x))^{4/3}\leq  C\psi(x),\quad(\psi'(x))^{4/3}\leq  C(1-\psi(x)),\quad\lvert\psi''(x)\rvert^{3/2}\leq  C\psi'(x).\label{1eq:psi}
\end{align}
Note that it is enough to take $\psi$ that is equal to $(1+x)^{4}$
on a neighbourhood of $-1$ and  equal to $1-(-1+x)^{4}$ on
a neighbourhood of $1$.

These conditions on $\psi$ will be needed for the proof of Proposition
\ref{1lem:coerc}.

For any $j=2,...,J$, let 
\begin{align}
\sigma_{j}:=\frac{1}{2}(v_{j-1}+v_{j}).\label{1eq:meanvel}
\end{align}

For any $j=2,...,J-1$, let 
\begin{align}
\varphi_{j}(t,x):=\psi\bigg(\frac{x-\sigma_{j}t}{\delta t}\bigg)-\psi\bigg(\frac{x-\sigma_{j+1}t}{\delta t}\bigg),\label{1eq:phij}
\end{align}
\begin{align}
\varphi_{1}(t,x):=1-\psi\Big(\frac{x-\sigma_{2}t}{\delta t}\Big),\qquad\varphi_{J}(t,x):=\psi\Big(\frac{x-\sigma_{J}t}{\delta t}\Big),\label{1eq:phij-1}
\end{align}
so that the function $\varphi_{j}$ corresponds obviously to the object
$P_{j}$. We will also use notations $\varphi_{l}^{s}$ and $\varphi_{k}^{b}$,
which represent the same functions, and where $\varphi_{l}^{s}$ corresponds
to the soliton $R_{l}$ and $\varphi_{k}^{b}$ corresponds to the
breather $B_{k}$. 

We will also denote, for $j=2,...,J-1$,
\begin{align}
\varphi_{1,j}(t,x):=\psi'\bigg(\frac{x-\sigma_{j}t}{\delta t}\bigg)-\psi'\bigg(\frac{x-\sigma_{j+1}t}{\delta t}\bigg),\label{1eq:phid}
\end{align}
\begin{align}
\varphi_{1,1}(t,x):=-\psi'\Big(\frac{x-\sigma_{2}t}{\delta t}\Big),\qquad\varphi_{1,J}(t,x):=\psi'\Big(\frac{x-\sigma_{J}t}{\delta t}\Big).\label{1eq:phid-1}
\end{align}
Of course, notations $\varphi_{1,k}^{b}$, $\varphi_{1,l}^{s}$ or
$\varphi_{2,j}$ will be used, with similar obvious definitions.

We have that, for $j=1,...,J$, 
\begin{align}
\lvert\varphi_{1,j}\rvert\leq  C\varphi_{j}^{3/4}.\label{1eq:phineq}
\end{align}

\begin{rem}
\label{1rem:If-,rem}If $\delta\leq \frac{\tau}{4}$, 
\begin{align}
\int_{-\infty}^{\sigma_{j}t+\delta t}e^{-2\beta\lvert x-v_{j}t\rvert}\,dx & =e^{-2\beta v_{j}t}\int_{-\infty}^{\sigma_{j}t+\delta t}e^{2\beta x}\,dx \\
 & =\frac{1}{2\beta}e^{-2\beta v_{j}t}e^{\beta(v_{j}+v_{j-1})t}e^{2\beta\delta t}\\
 & \leq  Ce^{-\beta\tau t}e^{2\beta\delta t} \leq  Ce^{-\beta\tau t/2},\label{1eq:local}
\end{align}
and
\begin{align}
\int_{\sigma_{j+1}t-\delta t}^{+\infty}e^{-2\beta\lvert x-v_{j}t\rvert}\,dx\leq  Ce^{-\beta\tau t/2},\label{1eq:locali}
\end{align}
for the same reason, and if $i\neq j$, e.g. $j>i$,
\begin{align}
\int_{\sigma_{j}t-\delta t}^{\sigma_{j+1}t+\delta t}e^{-2\beta\lvert x-v_{i}t\rvert}\,dx & =e^{2\beta v_{i}t}\int_{\sigma_{j}t-\delta t}^{\sigma_{j+1}t+\delta t}e^{-2\beta x}\,dx \\
 & \leq \frac{1}{2\beta}e^{2\beta v_{i}t}e^{-\beta(v_{j}+v_{j-1})t}e^{2\beta\delta t} \\
 & \leq  Ce^{-\beta\tau t}e^{2\beta\delta t} \leq  Ce^{-\beta\tau t/2}.\label{1eq:locall}
\end{align}
\end{rem}

And finally, we set for all $j=1,...,J$: 
\begin{align}
M_{j}(t) & :=\int\frac{1}{2}p^{2}(t,x)\varphi_{j}(t,x)\,dx=:M_{j}[p](t),\\
 E_{j}(t) & :=\int\bigg(\frac{1}{2}p_{x}^{2}(t,x)-\frac{1}{4}p^{4}(t,x)\bigg)\varphi_{j}(t,x)\,dx=:E_{j}[p](t).\label{1eq:lme}
\end{align}
Notations $M_{l}^{s},M_{k}^{b},E_{l}^{s},E_{k}^{b}$ will also
be used.

These are local versions of the mass and the energy of the solution
$p$ considered (localized around each breather or soliton). We will
prove the following result for the localized mass and energy: 
\begin{lem}
\label{1lem:lcq}There exists $C>0$ and $T_{1}^{*}:=T_{1}^{*}(A)$
such that, if $T^{*}\geq  T_{1}^{*}$, then for any $j=1,...,J$, for
any $t\in[t^{*},T]$, 
\begin{align}
\lvert M_{j}(T)-M_{j}(t)\rvert+\lvert E_{j}(T)-E_{j}(t)\rvert\leq \frac{C}{\delta^{2}t}A^{2}e^{-2\theta t}.\label{1eq:lcq}
\end{align}
\end{lem}

\begin{proof}
We will use the results of the computations made on the bottom of
page 1115 and on the bottom of page 1116 of \cite{key-2}
to claim the following facts: 
\begin{align}
\frac{d}{dt}\frac{1}{2}\int p^{2}f & =\int\bigg(-\frac{3}{2}p_{x}^{2}+\frac{3}{4}p^{4}\bigg)f'-\int p_{x}pf'',\\
\frac{d}{dt}\int\bigg[\frac{1}{2}p_{x}^{2}-\frac{1}{4}p^{4}\bigg]f & =\int\bigg[-\frac{1}{2}(p_{xx}+p^{3})^{2}-p_{xx}^{2}+3p_{x}^{2}p^{2}\bigg]f'\\
&\quad -\int p_{xx}p_{x}f'',\label{1eq:me}
\end{align}
where $f$ is a $C^{2}$ function that does not depend on time.

$M_{j}(t)$ is a sum of quantities of the form $\frac{1}{2}\int p^{2}\psi(\frac{x-\sigma_{j}t}{\delta t})$. This is why, we compute:
\begin{align}
\MoveEqLeft\frac{d}{dt}\frac{1}{2}\int p^{2}\psi\bigg(\frac{x-\sigma_{j}t}{\delta t}\bigg)  =\frac{1}{\delta t}\int\bigg(-\frac{3}{2}p_{x}^{2}+\frac{3}{4}p^{4}\bigg)\psi'\bigg(\frac{x-\sigma_{j}t}{\delta t}\bigg) \\
 & -\frac{1}{(\delta t)^{2}}\int p_{x}p\psi''\bigg(\frac{x-\sigma_{j}t}{\delta t}\bigg)-\frac{1}{2}\int p^{2}\frac{x}{\delta t^{2}}\psi'\bigg(\frac{x-\sigma_{j}t}{\delta t}\bigg).\label{1eq:mt}
\end{align}
$\psi'\big(\frac{x-\sigma_{j}t}{\delta t}\big)$ is zero outside of $\Omega_{j}(t):=(-\delta t+\sigma_{j}t,\delta t+\sigma_{j}t)$.
Thus, for $x\in\Omega_{j}(t)$, $\big\lvert\frac{x}{t}\big\rvert\leq \lvert\sigma_{j}\rvert+\lvert\delta\rvert\leq \lvert\sigma_{j}\rvert+1$,
this means that $\lvert\frac{x}{t}\rvert$ is bounded by a constant
(that depends only on the given parameters). We can deduce that 
\begin{align}
\bigg\lvert\frac{d}{dt}\frac{1}{2}\int p^{2}\psi\bigg(\frac{x-\sigma_{j}t}{\delta t}\bigg)\bigg\rvert\leq \frac{C}{\delta^{2}t}\Bigg(\int_{\Omega_{j}(t)}p_{x}^{2}+\int_{\Omega_{j}(t)}p^{4}+\int_{\Omega_{j}(t)}p^{2}\Bigg).\label{1eq:mtc}
\end{align}
We bound $\int_{\Omega_{j}(t)}p^{4}$ by the following way: 
\begin{align}
\int_{\Omega_{j}(t)}p^{4} & \leq \lVert p\rVert_{L^{\infty}}^{2}\int_{\Omega_{j}(t)}p^{2} \\
 & \leq  C\lVert p\rVert_{H^{1}}^{2}\int_{\Omega_{j}(t)}p^{2}\qquad\textrm{by Sobolev embedding} \\
 & \leq  C\int_{\Omega_{j}(t)}p^{2}\qquad\textrm{by Remark \ref{1rem:bound}.}\label{1eq:42}
\end{align}
Thus, we have for any $t\in[t^{*},T]$, 
\begin{align}
\bigg\lvert\frac{d}{dt}\frac{1}{2}\int p^{2}\psi\bigg(\frac{x-\sigma_{j}t}{\delta t}\bigg)\bigg\rvert\leq \frac{C}{\delta^{2}t}\Bigg(\int_{\Omega_{j}(t)}p_{x}^{2}+\int_{\Omega_{j}(t)}p^{2}\Bigg).\label{1eq:mtc-1}
\end{align}

$E_{j}(t)$ is a sum of quantities of the form $\int[\frac{1}{2}p_{x}^{2}-\frac{1}{4}p^{4}]\psi(\frac{x-\sigma_{j}t}{\delta t})$. So, we compute:
\begin{align}
\MoveEqLeft\frac{d}{dt}\int\bigg[\frac{1}{2}p_{x}^{2}-\frac{1}{4}p^{4}\bigg]\psi\bigg(\frac{x-\sigma_{j}t}{\delta t}\bigg)\\
&  =\frac{1}{\delta t}\int\bigg[-\frac{1}{2}(p_{xx}+p^{3})^{2}-p_{xx}^{2}+3p_{x}^{2}p\bigg]\psi'\bigg(\frac{x-\sigma_{j}t}{\delta t}\bigg) \\
 & \quad-\frac{1}{(\delta t)^{2}}\int p_{xx}p_{x}\psi''\bigg(\frac{x-\sigma_{j}t}{\delta t}\bigg)-\int\bigg[\frac{1}{2}p_{x}^{2}-\frac{1}{4}p^{4}\bigg]\frac{x}{\delta t^{2}}\psi'\bigg(\frac{x-\sigma_{j}t}{\delta t}\bigg).\label{1eq:et}
\end{align}
We deduce from this, by using similar arguments as for the mass, that
for any $t\in[t^{*},T]$, 
\begin{align}
\bigg\lvert\frac{d}{dt}\int\bigg[\frac{1}{2}p_{x}^{2}-\frac{1}{4}p^{4}\bigg]\psi\bigg(\frac{x-\sigma_{j}t}{\delta t}\bigg)\bigg\rvert\leq \frac{C}{\delta^{2}t}\Bigg(\int_{\Omega_{j}(t)}p^{2}+\int_{\Omega_{j}(t)}p_{x}^{2}+\int_{\Omega_{j}(t)}p_{xx}^{2}\Bigg).\label{1eq:etc}
\end{align}

Now, we write $p(t)=P(t)+(p(t)-P(t))$, and we use the triangular
inequality:
\begin{align}
\int_{\Omega_{j}(t)}\big(p^{2}+p_{x}^{2}+p_{xx}^{2}\big)\leq 2\int_{\Omega_{j}(t)}\big(P^{2}+P_{x}^{2}+P_{xx}^{2}\big)+2\lVert p-P\rVert_{H^{2}}^{2}.\label{1eq:cct}
\end{align}
We have assumed that $\lVert p-P\rVert_{H^{2}}^{2}\leq  A^{2}e^{-2\theta t}$, so
we need to study $P$ on $\Omega_{j}(t)$. The following computations
work also for the derivatives of $P$:
\begin{align}
\int_{\Omega_{j}(t)}P^{2} & =\int_{\Omega_{j}(t)}\Bigg(\sum_{m=1}^{J}P_{m}(t,x)\Bigg)^{2}\,dx=\sum_{1\leq  m,l\leq  J}\int_{\Omega_{j}(t)}P_{m}(t,x)P_{l}(t,x)\,dx \\
 & \leq  C\sum_{1\leq  m,l\leq  J}\int_{\Omega_{j}(t)}e^{-\beta\lvert x-v_{m}t\rvert}e^{-\beta\lvert x-v_{l}t\rvert}\,dx,\label{1eq:P2}
\end{align}
where we use the Proposition \ref{1prop:decay}.

We assume that $m\geq  j$ (we argue similarly if $m\leq  j-1$). Then,
\begin{align}
x\in\Omega_{j}(t) & \Leftrightarrow-\delta t+\sigma_{j}t\leq  x\leq \delta t+\sigma_{j}t \\
 & \Leftrightarrow-\delta t+(\sigma_{j}-v_{m})t\leq  x-v_{m}t\leq \delta t+(\sigma_{j}-v_{m})t.\label{1eq:local?}
\end{align}
We note that $\sigma_{j}-v_{m}\leq -\frac{1}{2}\tau<0$, we can
thus deduce from the condition on $\delta$ that $\sigma_{j}-v_{m}+\delta\leq -\frac{1}{4}\tau<0$.
We deduce that $x-v_{m}t$ is negative for $x\in\Omega_{j}(t)$.
Similarly, if $m\leq  j-1$, $x-v_{m}t$ is positive for $x\in\Omega_{j}(t)$.
We will now make calculations for different cases.
If $m,l\leq  j-1$, 
\begin{align}
\int_{\Omega_{j}(t)}e^{-\beta\lvert x-v_{m}t\rvert}e^{-\beta\lvert x-v_{l}t\rvert}\,dx & \leq \int_{\Omega_{j}(t)}e^{-\beta(x-v_{m}t)}e^{-\beta(x-v_{l}t)}\,dx \\
 & =\frac{1}{2\beta}e^{\beta t(-v_{j}-v_{j-1}+v_{m}+v_{l})}(e^{2\beta\delta t}-e^{-2\beta\delta t}) \\
 & \leq  Ce^{\beta t(-v_{j}-v_{j-1}+v_{m}+v_{l}+2\delta)} \leq  Ce^{-\beta\tau t/2}.\label{1eq:crl}
\end{align}
Similarly, if $m,l\geq  j$, 
\begin{align}
\int_{\Omega_{j}(t)}e^{-\beta\lvert x-v_{m}t\rvert}e^{-\beta\lvert x-v_{l}t\rvert}\,dx\leq  Ce^{-\beta\tau t/2}.\label{1eq:crl-1}
\end{align}
And, if $m\leq j-1,l\geq  j$, 
\begin{align}
\int_{\Omega_{j}(t)}e^{-\beta\lvert x-v_{m}t\rvert}e^{-\beta\lvert x-v_{l}t\rvert}\,dx & \leq \int_{\Omega_{j}(t)}e^{-\beta(x-v_{m}t)}e^{\beta(x-v_{l}t)}\,dx \\
 & \leq 2\delta te^{\beta t(v_{m}-v_{l})} \leq  Ce^{-\frac{\beta\tau t}{2}}.\label{1eq:crl-2}
\end{align}
Thus, 
\begin{align}
\int_{\Omega_{j}(t)}P^{2}\leq  Ce^{-\frac{\beta\tau t}{2}},\label{1eq:crl-3}
\end{align}
and the same is valid for the derivatives of $P$.

Thus, for $t\in[t^{*},T]$, 
\begin{align}
\MoveEqLeft\sum_{j=1}^{J}\bigg\lvert\frac{d}{dt}\frac{1}{2}\int p^{2}\psi\bigg(\frac{x-\sigma_{j}t}{\delta t}\bigg)\bigg\rvert+\bigg\lvert\frac{d}{dt}\int\bigg[\frac{1}{2}p_{x}^{2}-\frac{1}{4}p^{4}\bigg]\psi\bigg(\frac{x-\sigma_{j}t}{\delta t}\bigg)\bigg\rvert \\
 & \leq \frac{C}{\delta^{2}t}A^{2}e^{-2\theta t}+\frac{C}{\delta^{2}t}e^{-\frac{\beta\tau t}{2}} \leq \frac{C}{\delta^{2}t}(A^{2}+e^{-2\theta t})e^{-2\theta t} \leq \frac{C}{\delta^{2}t}A^{2}e^{-2\theta t}.\label{1eq:lcqc}
\end{align}
Thus, for $j=1,...,J$, $t\in[t^{*},T]$, 
\begin{align}
\MoveEqLeft\lvert M_{j}(T)-M_{j}(t)\rvert+\lvert E_{j}(T)-E_{j}(t)\rvert\\
 & \leq \int_{t}^{T}\frac{C}{\delta^{2}s}A^{2}e^{-2\theta s}\,ds \leq \frac{C}{\delta^{2}t}A^{2}\int_{t}^{T}e^{-2\theta s}\,ds \\
 & =\frac{C}{\delta^{2}t}A^{2}\frac{1}{2\theta}(e^{-2\theta t}-e^{-2\theta T}) \leq \frac{C}{\delta^{2}t}A^{2}e^{-2\theta t}.\label{1eq:lcqci}
\end{align}
\end{proof}

\subsection{Modulation} \label{1sec:2.3}
\begin{lem}
\label{1lem:mod}There exists $C>0$, $T_{2}^{*}=T_{2}^{*}(A)$ such
that, if $T^{*}>T_{2}^{*}$, then there exist unique $C^{1}$ functions
$x_{1,k}:[t^{*},T]\rightarrow\mathbb{R}$, $x_{2,k}:[t^{*},T]\rightarrow\mathbb{R}$
for $1\leq  k\leq  K$ and $x_{0,l}:[t^{*},T]\rightarrow\mathbb{R}$,
$c_{0,l}:[t^{*},T]\rightarrow\mathbb{\mathbb{R}}$, such that if we
set 
\begin{align}
\varepsilon(t,x)=p(t,x)-\widetilde{B}(t,x)-\widetilde{R}(t,x)=p(t,x)-\widetilde{P}(t,x),\label{1eq:md}
\end{align}
where, for $1\leq k\leq K$,
\begin{align}
\widetilde{B}(t,x)=\sum_{k=1}^{K}\widetilde{B_{k}}(t),\quad\widetilde{B_{k}}(t,x)=B_{\alpha_{k},\beta_{k}}(t,x;x_{1,k}^{0}+x_{1,k}(t),x_{2,k}^{0}+x_{2,k}(t)),\label{1mb}
\end{align}
for $1\leq l\leq L$,
\begin{align}
\widetilde{R}(t,x):=\sum_{l=1}^{L}\widetilde{R_{l}}(t),\quad\widetilde{R_{l}}(t,x):=\kappa_lQ_{c_{l}+c_{0,l}(t)}(x-x_{0,l}^0+x_{0,l}(t)-c_lt),\label{1eq:mr}
\end{align}
\begin{align}
\widetilde{P}(t):=\widetilde{R}(t)+\widetilde{B}(t),\label{1eq:mp}
\end{align}
and 
\begin{align}
\widetilde{P}(t):=\sum_{j=1}^{J}\widetilde{P_{j}}(t),\label{1eq:mpj}
\end{align}
where there is the usual correspondence between $\widetilde{P_{j}}$
and $\widetilde{B_{k}}$ or $\widetilde{R_{l}}$,

then, $\varepsilon(t)$ satisfies, for any $k=1,...,K$, for any $l=1,...,L$
and for any $t\in[t^{*},T]$, 
\begin{align}
\int\widetilde{R_{l}}(t)\varepsilon(t)\sqrt{\varphi_{l}^{s}(t)}=\int\partial_{x}\widetilde{R_{l}}(t)\varepsilon(t)\sqrt{\varphi_{l}^{s}(t)} & =0,\\
\int\partial_{x_{1}}\widetilde{B_{k}}(t)\varepsilon(t)\sqrt{\varphi_{k}^{b}(t)}=\int\partial_{x_{2}}\widetilde{B_{k}}(t)\varepsilon(t)\sqrt{\varphi_{k}^{b}(t)} & =0.\label{1eq:orthom}
\end{align}

Moreover, for any $t\in[t^{*},T]$, 
\begin{align}
\lVert\varepsilon(t)\rVert_{H^{2}}+\sum_{k=1}^{K}(\lvert x_{1,k}(t)\rvert+\lvert x_{2,k}(t)\rvert)+\sum_{l=1}^{L}(\lvert x_{0,l}(t)\rvert+\lvert c_{0,l}(t)\rvert)\leq  CAe^{-\theta t},\label{1eq:ineg1}
\end{align}
and 
\begin{align}
\sum_{k=1}^{K}(\lvert x_{1,k}'(t)\rvert+\lvert x_{2,k}'(t)\rvert)+\sum_{l=1}^{L}(\lvert x_{0,l}'(t)\rvert+\lvert c_{0,l}'(t)\rvert)\leq  C\lVert\varepsilon(t)\rVert_{L^{2}}+Ce^{-\theta t}.\label{1eq:ineg2}
\end{align}
Finally, $p(T)=P(T)=\widetilde{P}(T)$ and $\varepsilon(T)=x_{0,l}(T)=x_{1,k}(T)=x_{2,k}(T)=c_{0,l}(T)=0$. 
\end{lem}

\begin{proof}[Proof: see for example \cite{key-7} for reference]
Let, for $t\in[t^{*},T]$, 
\begin{align}
F_{t}:L^{2}(\mathbb{R})\times\mathbb{R}^{2K}\times\mathbb{R}^{2L}\rightarrow\mathbb{R}^{2K+2L},\label{1eq:Ft}
\end{align}
such that
\begin{align}
\MoveEqLeft (w,x_{1,k},x_{2,k},x_{0,l},c_{0,l}) \\
 & \longmapsto\Bigg(\int\sqrt{\varphi_{k}^{b}(t,x)}\partial_{x_{1}}B_{\alpha_{k},\beta_{k}}(t,x;x_{1,k}^{0}+x_{1,k},x_{2,k}^{0}+x_{2,k})\epsilon, \\
 & \qquad\int\sqrt{\varphi_{k}^{b}(t,x)}\partial_{x_{2}}B_{\alpha_{k},\beta_{k}}(t,x;x_{1,k}^{0}+x_{1,k},x_{2,k}^{0}+x_{2,k})\epsilon, \\
 & \qquad\int\sqrt{\varphi_{l}^{s}(t,x)}\partial_{x}\kappa_lQ_{c_{l}+c_{0,l}}(x-x_{0,l}^0+x_{0,l}-c_lt)\epsilon, \\
 & \qquad\int\sqrt{\varphi_{l}^{s}(t,x)}\kappa_lQ_{c_{l}+c_{0,l}}(x-x_{0,l}^0+x_{0,l}-c_lt)\epsilon\Bigg),\label{1Ftdef}
\end{align}
where
\begin{align}
\epsilon & :=w-\sum_{m=1}^{K}B_{\alpha_{m},\beta_{m}}(t,x;x_{1,m}^{0}+x_{1,m},x_{2,m}^{0}+x_{2,m})\\
&\quad -\sum_{n=1}^{L}\kappa_nQ_{c_{n}+c_{0,n}}(x-x_{0,n}^0+x_{0,n}-c_nt).
\end{align}

We observe that $F_{t}$ is a $C^{1}$ function and that $F_{t}(P(t),0,0,0,0)=0$.
Now, let us consider the matrix which gives the differential of $F_{t}$
(with respect to $x_{1,k},x_{2,k},x_{0,l},c_{0,l}$) in $(P(t),0,0,0,0)$
(we consider diagonal and extra-diagonal terms for each bloc): 
\begin{align}
DF_t=
\begin{pmatrix}
B_{k,k}^{1} & B_{k,k}^{3} & \times & \times & \times & \times & \times & \times\\
B_{k,k}^{3} & B_{k,k}^{2} & \times & \times & \times & \times & \times & \times\\
\times & \times & B_{k',k'}^{1} & B_{k',k'}^{3} & \times & \times & \times & \times\\
\times & \times & B_{k',k'}^{3} & B_{k',k'}^{2} & \times & \times & \times & \times\\
\times & \times & \times & \times & R_{l,l}^{1} & R_{l,l}^{4} & \times & \times\\
\times & \times & \times & \times & R_{l,l}^{3} & R_{l,l}^{2} & \times & \times\\
\times & \times & \times & \times & \times & \times & R_{l',l'}^{1} & R_{l',l'}^{4}\\
\times & \times & \times & \times & \times & \times & R_{l',l'}^{3} & R_{l',l'}^{2}
\end{pmatrix}
,\label{1matr}
\end{align}
where
\begin{align}
& B_{k,k}^{1}:=-\int\big(\partial_{x_{1}}B_{\alpha_{k},\beta_{k}}\big)^{2}\sqrt{\varphi_{k}^{b}},\quad B_{k,k}^{2}:=-\int\big(\partial_{x_{2}}B_{\alpha_{k},\beta_{k}}\big)^{2}\sqrt{\varphi_{k}^{b}},\\
&  B_{k,k}^{3}:=-\int\partial_{x_{1}}B_{\alpha_{k},\beta_{k}}\partial_{x_{2}}B_{\alpha_{k},\beta_{k}}\sqrt{\varphi_{k}^{b}},\label{1bm}
\end{align}
\begin{align}
& R_{l,l}^{1}:=-\int\big(\partial_{x}Q_{c_l}(y_{0,l}^0)\big)^{2}\sqrt{\varphi_{l}^{s}},\quad R_{l,l}^{3}:=-\int Q_{c_l}(y_{0,l}^0)\partial_{x}Q_{c_l}(y_{0,l}^0)\sqrt{\varphi_{l}^{s}},\\
& R_{l,l}^{2}:=-\frac{1}{2c_l}\int Q_{c_l}(y_{0,l}^0)\Big(Q_{c_l}(y_{0,l}^0)+y_{0,l}^0\partial_xQ_{c_l}(y_{0,l}^0) \Big)\sqrt{\varphi_{l}^{s}},\label{1rm}\\
& R_{l,l}^{4}:=-\frac{1}{2c_l}\int\partial_{x}Q_{c_l}(y_{0,l}^0)\Big(Q_{c_l}(y_{0,l}^0)+y_{0,l}^0\partial_xQ_{c_l}(y_{0,l}^0) \Big)\sqrt{\varphi_{l}^{s}},
\end{align}
denoting $y_{0,l}^0:=x-x_{0,l}^0-c_lt$,
and crosses stand for exponentially decaying terms when $t\rightarrow+\infty$,
and where  we consider variables in the following order: $x_{1,1},x_{2,1}$, $x_{1,2},x_{2,2}$, $x_{1,3},x_{2,3}$, ..., $x_{1,K},x_{2,K}$,
$x_{0,1},c_{0,1}$, ..., $x_{0,L},c_{0,L}$ and we order the coefficients
of the function in the similar way. This is a matrix with dominant
diagonal blocs. 

Note that $B_{k,k}^{1}$ is exponentially close to $-\int\big(\partial_{x_{1}}B_{\alpha_{k},\beta_{k}}\big)^{2}$,
because if $P_{j}=B_{k}$ is a breather,
\begin{align}
\int\Big(\partial_{x_{1}}B_{\alpha_{k},\beta_{k}}\Big)^{2}\Big(1-\sqrt{\varphi_{k}^{b}}\Big) & \leq \int_{-\infty}^{\sigma_{j}t+\delta t}\Big(\partial_{x_{1}}B_{\alpha_{k},\beta_{k}}\Big)^{2}+\int_{\sigma_{j+1}t+\delta t}^{+\infty}\Big(\partial_{x_{1}}B_{\alpha_{k},\beta_{k}}\Big)^{2} \\
 & \leq  C\int_{-\infty}^{\sigma_{j}t+\delta t}e^{-2\beta\lvert x-v_{j}t\rvert}+\int_{\sigma_{j+1}t+\delta t}^{+\infty}e^{-2\beta\lvert x-v_{j}t\rvert} \\
 & \leq  Ce^{-\frac{\beta\tau}{2}t},\label{1bre}
\end{align}
and the same is true for the other dominant diagonal terms of the
matrix (we can get rid of $\varphi$s).

Therefore, the determinant of the matrix is exponentially close to: 
\begin{align}
\MoveEqLeft\det(DF_t)\\
 & =\prod_{k=1}^{K}\Bigg(\int\big(\partial_{x_{1}}B_{\alpha_{k},\beta_{k}}(t,x;x_{1,k}^{0},x_{2,k}^{0})\big)^{2}\int\big(\partial_{x_{2}}B_{\alpha_{k},\beta_{k}}(t,x;x_{1,k}^{0},x_{2,k}^{0})\big)^{2}\\
& \qquad -\bigg(\int\partial_{x_{1}}B_{\alpha_{k},\beta_{k}}\partial_{x_{2}}B_{\alpha_{k},\beta_{k}}\bigg)^{2}\Bigg) \\
 &\ \cdot\prod_{l=1}^{L}\Bigg(\frac{1}{2c_l}\int Q_{c_l}(y_{0,l}^0)\Big(Q_{c_l}(y_{0,l}^0)+y_{0,l}^0\partial_xQ_{c_l}(y_{0,l}^0) \Big)\int\big(\partial_{x}Q_{c_l}(y_{0,l}^0)\big)^{2} \Bigg),\label{1Delta}
\end{align}
because $\int Q_{c_l}(y_{0,l}^0)\partial_x Q_{c_l}(y_{0,l}^0)\,dx=0$.

By Cauchy-Schwarz inequality and the fact that 
\begin{equation*}
\partial_{x_{1}}B_{\alpha_{k},\beta_{k}}(t,x;x_{1,k}^{0},x_{2,k}^{0})\quad\text{and}\quad\partial_{x_{2}}B_{\alpha_{k},\beta_{k}}(t,x;x_{1,k}^{0},x_{2,k}^{0})
\end{equation*}
are linearly independent as functions of the $x$ variable, for any
time $t$ fixed, we see that the first product is positive. Since
each member of the product is periodic in time, then the first product
is bounded below by a positive constant independent from time and translation
parameters.

For the second product, by translation of the variable in the integrations, for any time $t$ fixed, we see that we can replace $y_{0,l}^0$ by $x$. Then, by integration by parts,
\begin{align}
\int xQ_{c_l}(x)\partial_xQ_{c_l}(x)\,dx=-\frac{1}{2}\int Q_{c_l}(x)^2\,dx.
\end{align}
By scaling, if $q$ denotes the soliton with $c=1$, i.e. $q=Q_1$,
\begin{align}
\int Q_{c_l}^2=\sqrt{c_l}\int q^2,\quad \int\partial_xQ_{c_l}^2=c_l^{3/2}\int q_x^2.
\end{align}
Therefore,
\begin{align}
\MoveEqLeft\frac{1}{2c_l}\int Q_{c_l}(y_{0,l}^0)\Big(Q_{c_l}(y_{0,l}^0)+y_{0,l}^0\partial_xQ_{c_l}(y_{0,l}^0) \Big)\int(\partial_{x}Q_{c_l}(y_{0,l}^0))^{2}\\
 & =\frac{1}{4}c_{l}\int q^{2}\int(q_{x})^{2} \\
 & \geq \frac{1}{4}\min\{ c_{n},1\leq  n\leq  L\} \int q^{2}\int q_{x}^{2}.\label{1prod}
\end{align}
This means that the second product is bounded below by a positive constant
independent from time and translation parameters.

This means that if $T_{2}^{*}$ is large enough, the considered matrix 
is invertible.

Now, we may use the implicit function theorem (actually, we use a quantitative version of the implicit function theorem, see \cite[Section 2.2]{CH} for a precise statement). If $w$ is close enough
to $P(t)$, then there exists \begin{align}
\label{1param}
(x_{1,k},x_{2,k},x_{0,l},c_{0,l})
\end{align} 
such that \begin{align}
F_{t}(w,x_{1,k},x_{2,k},x_{0,l},c_{0,l})=0,
\end{align} where \eqref{1param}
depends in a regular $C^{1}$ way on $w$. It is possible to show that
the ``close enough'' in the previous sentence does not depend on $t$;
for this, it is required to use a uniform implicit function theorem.
This means that for $T_{2}^{*}$ large enough (depending on $A$),
$Ae^{-\theta t}$ is small enough for $t\in[t^{*},T]$, thus for $t\in[t^{*},T]$,
$p(t)$ is close enough to $P(t)$ in order to apply the implicit
function theorem. Therefore, we have for $t\in[t^{*},T]$, the existence
of $x_{1,k}(t)$, $x_{2,k}(t)$, $x_{0,l}(t)$ and $c_{0,l}(t)$.
It is possible to show that these functions are $C^{1}$ in time.
Basically, this comes from the fact that they are $C^{1}$ in $p(t)$
and that $p(t)$ has a similar regularity in time (see \cite{key-7}
for more details).

Now, we prove the inequalities (\ref{1eq:ineg1}) and (\ref{1eq:ineg2}).
We can take the differential of the implicit functions with respect to $p(t)$  for $t\in[t^{*},T]$. For this, we differentiate the following equation with respect
to $p(t)$:
\begin{align}
F_{t}\Big(p(t),x_{1,k}\big(p(t)\big),x_{2,k}\big(p(t)\big),x_{0,l}\big(p(t)\big),c_{0,l}\big(p(t)\big)\Big)=0.\label{1impl}
\end{align}
We know that the matrix that gives the differential of $F_{t}$ (with
respect to $x_{1,k}$, $x_{2,k}$, $x_{0,l}$, $c_{0,l}$) in 
\begin{align}
\Big(p(t),x_{1,k}\big(p(t)\big),x_{2,k}\big(p(t)\big),x_{0,l}\big(p(t)\big),c_{0,l}\big(p(t)\big)\Big)\label{1poit}
\end{align}
 is invertible and its inverse is bounded in time (from the formula giving the inverse of a matrix from the comatrix and the determinant).
The differential of $F_{t}$ with respect to the first variable is
also bounded. Thus, by the mean-value theorem: 
\begin{align}
\lvert x_{1,k}\rvert\leq  C\lVert p-P\rVert\leq  CAe^{-\theta t}.\label{1eq:-25}
\end{align}
The same is true for $x_{2,k}$, $x_{0,l}$ and $c_{0,l}$.

By applying the mean-value theorem (inequality) for $Q_{c_l}$ with respect to $x_{0,l}$ and $c_{0,l}$  or for $B_{\alpha_{k},\beta_{k}}$
with respect to $x_{1,k}$ and $x_{2,k}$, we deduce that 
\begin{align}
\lVert P_j(t)-\widetilde{P_j}(t)\rVert_{H^{2}}\leq  C\big(\lvert x_{1,k}(t)\rvert+\lvert x_{2,k}(t)\rvert\big),\label{1eq:-26}
\end{align}
if $P_j=B_k$ is a breather, and
\begin{align}
\lVert P_j(t)-\widetilde{P_j}(t)\rVert_{H^{2}}\leq  C\big(\lvert x_{0,l}(t)\rvert+\lvert c_{0,l}(t)\rvert\big),
\end{align}
if $P_j=R_l$ is a soliton.

Finally, by triangular inequality,
\begin{align}
\lVert\varepsilon(t)\rVert_{H^{2}} & \leq \lVert p(t)-P(t)\rVert_{H^{2}}+\lVert P(t)-\widetilde{P}(t)\rVert_{H^{2}}\leq \lVert p(t)-P(t)\rVert_{H^{2}} \\
 &\quad +C\Bigg(\sum_{k=1}^{K}\big(\lvert x_{1,k}(t)\rvert+\lvert x_{2,k}(t)\rvert\big)+\sum_{l=1}^{L}\big(\lvert x_{0,l}(t)\rvert+\lvert c_{0,l}(t)\rvert\big)\Bigg) \\
 & \leq  C\lVert p(t)-P(t)\rVert_{H^{2}}\leq  CAe^{-\theta t}.\label{1eq:-21}
\end{align}
This completes the proof of (\ref{1eq:ineg1}).

For (\ref{1eq:ineg2}), we will take time derivatives of the equations
\eqref{1eq:orthom}.
From now on, we write $\widetilde{B_{k}}_{1}$ for $\partial_{x_{1}}\widetilde{B_{k}}$
and $\widetilde{B_{k}}_{2}$ for $\partial_{x_{2}}\widetilde{B_{k}}$.
Firstly, we write the PDE verified by $\varepsilon$ (knowing that
$p,B_{1},...,B_{K},R_{1},...,R_{L}$ are solutions of \eqref{1mKdV}): 
\begin{align}
\partial_{t}\varepsilon & =-\varepsilon_{xxx}-\Bigg[\varepsilon\Bigg(\varepsilon^{2}+3\varepsilon\sum_{j=1}^{J}\widetilde{P_{j}}+3\sum_{i,j=1}^{J}\widetilde{P_{i}}\widetilde{P_{j}}\Bigg)\Bigg]_{x}\\
 &\quad -\sum_{k=1}^{K}x_{1,k}'(t)\widetilde{B_{k}}_{1} -\sum_{k=1}^{K}x_{2,k}'(t)\widetilde{B_{k}}_{2}-\sum_{l=1}^{L}x_{0,l}'(t)\widetilde{R_{l}}_{x}\\
 &\quad -\sum_{l=1}^{L}\frac{c_{0,l}'(t)}{2(c_l+c_{0,l}(t))}\Big(\widetilde{R_{l}}+y_{0,l}(t)\widetilde{R_{l}}_x\Big)-\sum_{h\neq i\textrm{ or }i\neq j}\Big(\widetilde{P_{h}}\widetilde{P_{i}}\widetilde{P_{j}}\Big)_{x}  ,\label{1eq:-22}
\end{align}
where $y_{0,l}(t):=x-x_{0,l}^0+x_{0,l}(t)-c_lt$.
Now, we will take the time derivative of the equation $\int\widetilde{B_{k}}_{1}\varepsilon\sqrt{\varphi_{k}^{b}}=0$
(and perform an integration by parts):
\begin{align}
& -\int\Big(\widetilde{B_{k}}^{3}\Big)_{1x}\varepsilon\sqrt{\varphi_{k}^{b}}-\int\widetilde{B_{k}}_{1}\sum_{h\neq i\textrm{ or }g\neq h}\Big(\widetilde{P_{h}}\widetilde{P_{i}}\widetilde{P_{g}}\Big)_{x}\sqrt{\varphi_{k}^{b}}\\
& +x_{2,k}'(t)\int\widetilde{B_{k}}_{12}\varepsilon\sqrt{\varphi_{k}^{b}}+\frac{1}{2\delta t}\int\widetilde{B_{k}}_{1}\varepsilon\Bigg(\varepsilon^{2}+3\varepsilon\sum_{i=1}^{J}\widetilde{P_{i}}+3\sum_{h,i=1}^{J}\widetilde{P_{h}}\widetilde{P_{i}}\Bigg)\frac{\varphi_{1,k}^{b}}{\sqrt{\varphi_{k}^{b}}} \\ &
+\int\widetilde{B_{k}}_{1x}\varepsilon\Bigg(\varepsilon^{2}+3\varepsilon\sum_{i=1}^{J}\widetilde{P_{i}}+3\sum_{h,i=1}^{J}\widetilde{P_{h}}\widetilde{P_{i}}\Bigg)\sqrt{\varphi_{k}^{b}}-\frac{1}{2\delta t^{2}}\int\widetilde{B_{k}}_{1}\varepsilon x\frac{\varphi_{1,k}^{b}}{\sqrt{\varphi_{k}^{b}}} \\ &
+\frac{1}{2\delta t}\int\widetilde{B_{k}}_{1}\varepsilon_{xx}\frac{\varphi_{1,k}^{b}}{\sqrt{\varphi_{k}^{b}}}-\frac{1}{2\delta t}\int\widetilde{B_{k}}_{1x}\varepsilon_{x}\frac{\varphi_{1,k}^{b}}{\sqrt{\varphi_{k}^{b}}}+\frac{1}{2\delta t}\int\widetilde{B_{k}}_{1xx}\varepsilon\frac{\varphi_{1,k}^{b}}{\sqrt{\varphi_{k}^{b}}}
  \\ & +x_{1,k}'(t)\int\widetilde{B_{k}}_{11}\varepsilon\sqrt{\varphi_{k}^{b}} =\sum_{m=1}^{K}x_{1,m}'(t)\int\widetilde{B_{k}}_{1}\widetilde{B_{m}}_{1}\sqrt{\varphi_{k}^{b}}
\\ & +\sum_{m=1}^{K}x_{2,m}'(t)\int\widetilde{B_{k}}_{1}\widetilde{B_{m}}_{2}\sqrt{\varphi_{k}^{b}} +\sum_{n=1}^{L}x_{0,n}'(t)\int\widetilde{B_{k}}_{1}\widetilde{R_{n}}_{x}\sqrt{\varphi_{k}^{b}}
\\
& +\sum_{n=1}^{L}\frac{c_{0,n}'(t)}{2\big(c_{n}+c_{0,n}(t)\big)}\int\widetilde{B_{k}}_{1}\Big(\widetilde{R_{n}}+y_{0,n}(t)\widetilde{R_{n}}_{x}\Big)\sqrt{\varphi_{k}^{b}} .\label{1eq:-27}
\end{align}
Similarly, taking the time derivative of $\int\widetilde{B_{k}}_{2}\varepsilon\sqrt{\varphi_{k}^{b}}=0$:
\begin{align}
& -\int\Big(\widetilde{B_{k}}^{3}\Big)_{2x}\varepsilon\sqrt{\varphi_{k}^{b}}-\int\widetilde{B_{k}}_{2}\sum_{h\neq i\textrm{ or }g\neq h}\Big(\widetilde{P_{h}}\widetilde{P_{i}}\widetilde{P_{g}}\Big)_{x}\sqrt{\varphi_{k}^{b}} \\ &
+x_{2,k}'(t)\int\widetilde{B_{k}}_{22}\varepsilon\sqrt{\varphi_{k}^{b}}+\frac{1}{2\delta t}\int\widetilde{B_{k}}_{2}\varepsilon\Bigg(\varepsilon^{2}+3\varepsilon\sum_{i=1}^{J}\widetilde{P_{i}}+3\sum_{h,i=1}^{J}\widetilde{P_{h}}\widetilde{P_{i}}\Bigg)\frac{\varphi_{1,k}^{b}}{\sqrt{\varphi_{k}^{b}}}\\
& +\int\widetilde{B_{k}}_{2x}\varepsilon\Bigg(\varepsilon^{2}+3\varepsilon\sum_{i=1}^{J}\widetilde{P_{i}}+3\sum_{h,i=1}^{J}\widetilde{P_{h}}\widetilde{P_{i}}\Bigg)\sqrt{\varphi_{k}^{b}}+\frac{1}{2\delta t}\int\widetilde{B_{k}}_{2}\varepsilon_{xx}\frac{\varphi_{1,k}^{b}}{\sqrt{\varphi_{k}^{b}}} \\ &
-\frac{1}{2\delta t}\int\widetilde{B_{k}}_{2x}\varepsilon_{x}\frac{\varphi_{1,k}^{b}}{\sqrt{\varphi_{k}^{b}}}+\frac{1}{2\delta t}\int\widetilde{B_{k}}_{2xx}\varepsilon\frac{\varphi_{1,k}^{b}}{\sqrt{\varphi_{k}^{b}}}
-\frac{1}{2\delta t^{2}}\int\widetilde{B_{k}}_{2}\varepsilon x\frac{\varphi_{1,k}^{b}}{\sqrt{\varphi_{k}^{b}}}  \\ 
& +x_{1,k}'(t)\int\widetilde{B_{k}}_{12}\varepsilon\sqrt{\varphi_{k}^{b}} =\sum_{m=1}^{K}x_{1,m}'(t)\int\widetilde{B_{k}}_{2}\widetilde{B_{m}}_{1}\sqrt{\varphi_{k}^{b}}\\
& 
+\sum_{m=1}^{K}x_{2,m}'(t)\int\widetilde{B_{k}}_{2}\widetilde{B_{m}}_{2}\sqrt{\varphi_{k}^{b}} +\sum_{n=1}^{L}x_{0,n}'(t)\int\widetilde{B_{k}}_{2}\widetilde{R_{n}}_{x}\sqrt{\varphi_{k}^{b}} \\ & +\sum_{n=1}^{L}\frac{c_{0,n}'(t)}{2\big(c_{n}+c_{0,n}(t)\big)}\int\widetilde{B_{k}}_{2}\Big(\widetilde{R_{n}}+y_{0,n}(t)\widetilde{R_{n}}_{x}\Big)\sqrt{\varphi_{k}^{b}}.\label{1eq:-29}
\end{align}
Similarly, taking the time derivative of $\int\widetilde{R_{l}}_{x}(t)\varepsilon(t)\sqrt{\varphi_{l}^{s}}=0$:
\begin{align}
& -\int\Big(\widetilde{R_{l}}^{3}\Big)_{xx}\varepsilon\sqrt{\varphi_{l}^{s}}+\frac{c_{0,l}'(t)}{2\big(c_{l}+c_{0,l}(t)\big)}\int\Big(\widetilde{R_{l}}_{x}+y_{0,l}(t)\widetilde{R_{l}}_{xx}\Big)\varepsilon\sqrt{\varphi_{l}^{s}}\\
& +x_{0,l}'(t)\int\widetilde{R_{l}}_{xx}\varepsilon\sqrt{\varphi_{l}^{s}}  +\frac{1}{2\delta t}\int\widetilde{R_{l}}_{x}\varepsilon\Bigg(\varepsilon^{2}+3\varepsilon\sum_{i=1}^{J}\widetilde{P_{i}}+3\sum_{h,i=1}^{J}\widetilde{P_{h}}\widetilde{P_{i}}\Bigg)\frac{\varphi_{1,l}^{s}}{\sqrt{\varphi_{l}^{s}}}\\
& +\int\widetilde{R_{l}}_{xx}\varepsilon\Bigg(\varepsilon^{2}+3\varepsilon\sum_{i=1}^{J}\widetilde{P_{i}}+3\sum_{h,i=1}^{J}\widetilde{P_{h}}\widetilde{P_{i}}\Bigg)\sqrt{\varphi_{l}^{s}}-\frac{1}{2\delta t^{2}}\int\widetilde{R_{l}}_{x}\varepsilon x\frac{\varphi_{1,l}^{s}}{\sqrt{\varphi_{l}^{s}}}
\\ & +\frac{1}{2\delta t}\int\widetilde{R_{l}}_{x}\varepsilon_{xx}\frac{\varphi_{1,l}^{s}}{\sqrt{\varphi_{l}^{s}}}-\frac{1}{2\delta t}\int\widetilde{R_{l}}_{xx}\varepsilon_{x}\frac{\varphi_{1,l}^{s}}{\sqrt{\varphi_{l}^{s}}}+\frac{1}{2\delta t}\int\widetilde{R_{l}}_{xxx}\varepsilon\frac{\varphi_{1,l}^{s}}{\sqrt{\varphi_{l}^{s}}}\\
& - \int\widetilde{R_{l}}_{x}\sum_{h\neq i\textrm{ or }g\neq h}\Big(\widetilde{P_{h}}\widetilde{P_{i}}\widetilde{P_{g}}\Big)_{x}\sqrt{\varphi_{l}^{s}}   =\sum_{n=1}^{L}x_{0,n}'(t)\int\widetilde{R_{l}}_{x}\widetilde{R_{n}}_{x}\sqrt{\varphi_{l}^{s}}\\
& +\sum_{n=1}^{L}\frac{c_{0,n}'(t)}{2\big(c_{n}+c_{0,n}(t)\big)}\int\widetilde{R_{l}}_{x}\Big(\widetilde{R_{n}}+y_{0,n}(t)\widetilde{R_{n}}_{x}\Big)\sqrt{\varphi_{l}^{s}} \\
& +\sum_{m=1}^{K}x_{1,m}'(t)\int\widetilde{R_{l}}_{x}\widetilde{B_{m}}_{1}\sqrt{\varphi_{l}^{s}}
 +\sum_{m=1}^{K}x_{2,m}'(t)\int\widetilde{R_{l}}_{x}\widetilde{B_{m}}_{2}\sqrt{\varphi_{l}^{s}}.\label{1eq:-30}
\end{align}
Finally, taking the time derivative of $\int\widetilde{R_{l}}\varepsilon\sqrt{\varphi_{l}^{s}}=0$:
\begin{align}
& -\int\Big(\widetilde{R_{l}}^{3}\Big)_{x}\varepsilon\sqrt{\varphi_{l}^{s}}+\frac{c_{0,l}'(t)}{2\big(c_{l}+c_{0,l}(t)\big)}\int\Big(\widetilde{R_{l}}+y_{0,l}(t)\widetilde{R_{l}}_{x}\Big)\varepsilon\sqrt{\varphi_{l}^{s}}\\
& +x_{0,l}'(t)\int\widetilde{R_{l}}_{x}\varepsilon\sqrt{\varphi_{l}^{s}}  +\frac{1}{2\delta t}\int\widetilde{R_{l}}\varepsilon\Bigg(\varepsilon^{2}+3\varepsilon\sum_{i=1}^{J}\widetilde{P_{i}}+3\sum_{h,i=1}^{J}\widetilde{P_{h}}\widetilde{P_{i}}\Bigg)\frac{\varphi_{1,l}^{s}}{\sqrt{\varphi_{l}^{s}}}\\
& +\int\widetilde{R_{l}}_{x}\varepsilon\Bigg(\varepsilon^{2}+3\varepsilon\sum_{i=1}^{J}\widetilde{P_{i}}+3\sum_{h,i=1}^{J}\widetilde{P_{h}}\widetilde{P_{i}}\Bigg)\sqrt{\varphi_{l}^{s}} -\frac{1}{2\delta t^{2}}\int\widetilde{R_{l}}\varepsilon x\frac{\varphi_{1,l}^{s}}{\sqrt{\varphi_{l}^{s}}} \\
 & +\frac{1}{2\delta t}\int\widetilde{R_{l}}\varepsilon_{xx}\frac{\varphi_{1,l}^{s}}{\sqrt{\varphi_{l}^{s}}}-\frac{1}{2\delta t}\int\widetilde{R_{l}}_{x}\varepsilon_{x}\frac{\varphi_{1,l}^{s}}{\sqrt{\varphi_{l}^{s}}} +\frac{1}{2\delta t}\int\widetilde{R_{l}}_{xx}\varepsilon\frac{\varphi_{1,l}^{s}}{\sqrt{\varphi_{l}^{s}}}
 \\ 
& -\int\widetilde{R_{l}}\sum_{h\neq i\textrm{ or }g\neq h}\Big(\widetilde{P_{h}}\widetilde{P_{i}}\widetilde{P_{g}}\Big)_{x}\sqrt{\varphi_{l}^{s}} =\sum_{n=1}^{L}x_{0,n}'(t)\int\widetilde{R_{l}}\widetilde{R_{n}}_{x}\sqrt{\varphi_{l}^{s}}\\
& +\sum_{n=1}^{L}\frac{c_{0,n}'(t)}{2\big(c_{n}+c_{0,n}(t)\big)}\int\widetilde{R_{l}}\Big(\widetilde{R_{n}}+y_{0,n}(t)\widetilde{R_{n}}_{x}\Big)\sqrt{\varphi_{l}^{s}} \\ &
+\sum_{m=1}^{K}x_{1,m}'(t)\int\widetilde{R_{l}}\widetilde{B_{m}}_{1}\sqrt{\varphi_{l}^{s}}+\sum_{m=1}^{K}x_{2,m}'(t)\int\widetilde{R_{l}}\widetilde{B_{m}}_{2}\sqrt{\varphi_{l}^{s}}.\label{1eq:-31}
\end{align}
By the Proposition \ref{1prop:decay-mod} below (that follows from
the first part of the lemma we prove) and its corollary, several terms of the equalities (\ref{1eq:-27}), (\ref{1eq:-29}), (\ref{1eq:-30}) and (\ref{1eq:-31}) are bounded by $Ce^{-\theta t}$;
other terms are $O(\lVert\varepsilon\rVert_{L^{2}})$. We remind that
$O(\lVert\varepsilon\rVert_{L^{2}})\leq  CAe^{-\theta t}$. From the basic
properties of $\varphi_{j}$ (see Section \ref{1sec:2.2}), $\frac{\varphi_{1,j}}{\sqrt{\varphi_{j}}}$
is bounded. Because of the compact support of $\varphi_{j}$, $\frac{x}{t}\frac{\varphi_{1,j}}{\sqrt{\varphi_{j}}}$
is bounded independently on $x$ and $t$. Using these bounds, and
after several linear combinations, we obtain the desired inequalities.
\end{proof}
\begin{rem}
\label{1rem:unif}As a consequence of Lemma \ref{1lem:mod}, there
exists a constant $C>0$ such that
\begin{align}
\forall t\in[t^{*},T]\quad\sum_{k=1}^{K}\big(\lvert x_{1,k}(t)\rvert+\lvert x_{2,k}(t)\rvert\big)+\sum_{l=1}^{L}\big(\lvert x_{0,l}(t)\rvert+\lvert c_{0,l}(t)\rvert\big)\leq  CAe^{-\theta T^{*}}.\label{1eq:-32}
\end{align}
\end{rem}
This means, that if we take $T_{2}^{*}$ eventually larger (which we will assume in the following of the article), we may extend
Proposition \ref{1prop:decay}\textbf{ }to $\widetilde{P_{j}}$ in
the following way, by integration of the bounds given by modulation
(the constant $C$ is a bit larger in a controled way, we write $\frac{\beta}{2}$
because the shape of the solitons is a bit modified in a controled
way):

\begin{prop}
\label{1prop:decay-mod}Let $j=1,...,J$, $n\in\mathbb{N}$. If $T^{*}>T_{2}^{*}$,
then there exists a constant $C>0$ such that for any $t,x\in\mathbb{R}$,
\begin{align}
\lvert\partial_{x}^{n}\widetilde{P_{j}}(t,x)\rvert\leq  Ce^{-\frac{\beta}{2}\lvert x-v_{j}t\rvert}.\label{1eq:-33}
\end{align}
\end{prop}

We will also use that any $\lVert\partial_{x}^{n}\widetilde{P_{j}}\rVert_{H^{2}}$
is bounded by $C$.
\begin{cor}
\label{1cor:decay-mod}
Let $i\neq j\in\{1,...,J\}$ and $m,n\in\mathbb{N}$. If $T^{*}>T_{2}^{*}$,
then there exists a constant $C$ that depends only on $P$, such
that for any $t\in\mathbb{R}$,
\begin{align}
\bigg\lvert\int\partial_{x}^{m}\widetilde{P_{i}}\partial_{x}^{n}\widetilde{P_{j}}\bigg\rvert\leq  Ce^{-\beta\tau t/8}.\label{1eq:-34}
\end{align}
\end{cor}

\subsection{Study of coercivity}

In \cite{key-1}, the Lyapunov functional that was introduced to study
the orbital stability of a breather was the following conserved-in-time
functional: 
\begin{align}
F[p](t)+2\big(\beta^{2}-\alpha^{2}\big)E[p](t)+\big(\alpha^{2}+\beta^{2}\big)^{2}M[p](t).\label{1eq:lf}
\end{align}
The functional that we will consider here is adapted from the latter.
For $t\in[t^{*},T]$, we set
\begin{align}
\mathcal{H}[p](t) & :=F[p](t)+\sum_{k=1}^{K}\Big(2\big(\beta_{k}^{2}-\alpha_{k}^{2}\big)E_{k}^{b}[p](t)+\big(\alpha_{k}^{2}+\beta_{k}^{2}\big)^{2}M_{k}^{b}[p](t)\Big)\\
&\quad +\sum_{l=1}^{L}\Big(2c_{l}E_{l}^{s}[p](t)+c_{l}^{2}M_{l}^{s}[p](t)\Big).\label{1lfg}
\end{align}

For the simplicity of notations, for $j\in\{1,...,J\}$, $a_{j}$ will
denote $\alpha_{k}$ if $P_{j}$ is the breather $B_{k}$ or $0$
if $P_{j}$ is a soliton, and $b_{j}$ will denote $\beta_{k}$ if
$P_{j}$ is the breather $B_{k}$ or $c_{l}^{1/2}$ if $P_{j}$ is
the soliton $R_{l}$. With these notations, we may write: 
\begin{align}
\mathcal{H}[p](t)=F[p](t)+\sum_{j=1}^{J}\Big(2\big(b_{j}^{2}-a_{j}^{2}\big)E_{j}[p](t)+\big(a_{j}^{2}+b_{j}^{2}\big)^{2}M_{j}[p](t)\Big).\label{1lfgs}
\end{align}

We would like to study locally this functional around the considered sum 
of breathers and solitons. The aim of this section will be to prove
two following propositions: 
\begin{prop}[Expansion of $H^{2}$ conserved quantity]
\label{1lem:exp}There exists $T_{4}^{*}>0$ such that if $T^{*}\geq  T_{4}^{*}$,
for all $t\in[t^{*},T]$, we have that
\begin{align}
\mathcal{H}[p](t) & =\sum_{j=1}^{J}\Big(F\big[\widetilde{P_{j}}\big](t)+2\big(b_{j}^{2}-a_{j}^{2}\big)E\big[\widetilde{P_{j}}\big](t)+\big(a_{j}^{2}+b_{j}^{2}\big)^{2}M\big[\widetilde{P_{j}}\big](t)\Big) \\
 &\quad +H_{2}[\varepsilon](t)+O\big(\lVert\varepsilon(t)\rVert_{H^{2}}^{3}\big)+O\big(e^{-2\theta t}\lVert\varepsilon(t)\rVert_{H^{2}}\big)+O\big(e^{-2\theta t}\big),\label{1dl}
\end{align}
where 
\begin{align}
H_{2}[\varepsilon](t) & :=\frac{1}{2}\int\varepsilon_{xx}^{2}-\frac{5}{2}\int\widetilde{P}^{2}\varepsilon_{x}^{2}+\frac{5}{2}\int\widetilde{P}_{x}^{2}\varepsilon^{2}+5\int\widetilde{P}\widetilde{P}_{xx}\varepsilon^{2}+\frac{15}{4}\int\widetilde{P}^{4}\varepsilon^{2} \\
 & +\sum_{j=1}^{J}\big(b_{j}^{2}-a_{j}^{2}\big)\bigg(\int\varepsilon_{x}^{2}\varphi_{j}-3\int\widetilde{P}^{2}\varepsilon^{2}\varphi_{j}\bigg)+\sum_{j=1}^{J}\big(a_{j}^{2}+b_{j}^{2}\big)^{2}\frac{1}{2}\int\varepsilon^{2}\varphi_{j}.\label{1eq:-4-2}
\end{align}
\end{prop}

\begin{prop}[Coercivity of $H_{2}$]
\label{1lem:coerc-1}There exists $\mu>0$, $T_{3}^{*}=T_{3}^{*}(A)$
such that, if $T^{*}\geq  T_{3}^{*}$, we have for any $t\in[t^{*},T]$, 
\begin{align}
H_{2}[\varepsilon](t)\geq \mu\lVert\varepsilon(t)\rVert_{H^{2}}^{2}-\frac{1}{\mu}\sum_{k=1}^{K}\bigg(\int\varepsilon\widetilde{B_{k}}\sqrt{\varphi_{k}^{b}}\bigg)^{2}.\label{1coerc}
\end{align}
\end{prop}

The Propositions \ref{1lem:exp} and \ref{1lem:coerc-1} will be used
in the next concluding subsection to prove the Proposition \ref{1prop:bootstrap}.

Firstly, let us prove the Proposition \ref{1lem:exp}. 
\begin{proof}[Proof of Proposition \ref{1lem:exp}]

We would like to compare $\mathcal{H}[\widetilde{P}+\varepsilon](t)$
and $\mathcal{H}[\widetilde{P}](t)$ (recall that $p=\widetilde{P}+\varepsilon$)
by studying the difference asymptotically when $\varepsilon$ is small.
Firstly, let us see how we could simplify the expression of $\mathcal{H}[\widetilde{P}](t)$.

\emph{Step 1:} 
\begin{claim}
\label{1lem:simpl1}If $T^{*}$ is large enough, for all $t\in[t^{*},T]$,
we have that
\begin{align}
\mathcal{H}[\widetilde{P}](t) & =\sum_{j=1}^{J}\Big(F\big[\widetilde{P_{j}}\big](t)+2\big(b_{j}^{2}-a_{j}^{2}\big)E\big[\widetilde{P_{j}}\big](t)+\big(a_{j}^{2}+b_{j}^{2}\big)^{2}M\big[\widetilde{P_{j}}\big](t)\Big)\\
&\quad +O\big(e^{-2\theta t}\big).\label{1claim}
\end{align}
\end{claim}

\begin{proof}
We prove that, for $t\in[t^{*},T]$,
\begin{align}
\bigg\lvert\mathcal{H}[\widetilde{P}]-\sum_{j=1}^{J}\Big(F\big[\widetilde{P_{j}}\big]+2\big(b_{j}^{2}-a_{j}^{2}\big)E\big[\widetilde{P_{j}}\big]+\big(a_{j}^{2}+b_{j}^{2}\big)^{2}M\big[\widetilde{P_{j}}\big]\Big)\bigg\rvert\leq  Ce^{-2\theta t}.\label{1eq:toprove}
\end{align}

Let us compare $F_{j}[\widetilde{P}]$ and $F[\widetilde{P_{j}}]$:
\begin{align}
F_{j}[\widetilde{P}]=\int\bigg(\frac{1}{2}\widetilde{P}_{xx}^{2}-\frac{5}{2}\widetilde{P}^{2}\widetilde{P}_{x}^{2}+\frac{1}{4}\widetilde{P}^{6}\bigg)\varphi_{j}(t,x)\,dx,\label{1eq:-35}
\end{align}
\begin{align}
F[\widetilde{P_{j}}]=\int\bigg(\frac{1}{2}\widetilde{P_{j}}_{xx}^{2}-\frac{5}{2}\widetilde{P_{j}}^{2}\widetilde{P_{j}}_{x}^{2}+\frac{1}{4}\widetilde{P_{j}}^{6}\bigg)\,dx.\label{1eq:-36}
\end{align}
We compare the corresponding terms of these equalities. Let us start by
the first one: 
\begin{align}
\MoveEqLeft\bigg\lvert\int\Big(\widetilde{P}_{xx}^{2}\varphi_{j}(t,x)-\widetilde{P_{j}}_{xx}^{2}\Big)\bigg\rvert \\
& \leq \int\widetilde{P_{j}}_{xx}^{2}\big\lvert 1-\varphi_{j}(t,x)\big\rvert +\sum_{(r,s)\neq(j,j)}\int\Big\lvert\widetilde{P_{r}}_{xx}\widetilde{P_{s}}_{xx}\Big\rvert\varphi_{j}(t,x) \\
 & \leq  C\int e^{-\frac{\beta}{2}\lvert x-v_{j}t\rvert}e^{\frac{\beta\tau}{32}t}\big\lvert 1-\varphi_{j}(t,x)\big\rvert\,dx\\
 & \quad +C\sum_{i\neq j}\int e^{-\frac{\beta}{2}\lvert x-v_{i}t\rvert}e^{\frac{\beta\tau}{32}t}\varphi_{j}(t,x)\,dx \\
 & \leq  Ce^{\frac{\beta\tau}{32}t}\Bigg[\Bigg(\int_{-\infty}^{\sigma_{j}t+\delta t}+\int_{\sigma_{j+1}t-\delta t}^{+\infty}\Bigg)e^{-\frac{\beta}{2}\lvert x-v_{j}t\rvert}\,dx\\
 & \quad+\sum_{i\neq j}\int_{\sigma_{j}t-\delta t}^{\sigma_{j+1}t+\delta t}e^{-\frac{\beta}{2}\lvert x-v_{i}t\rvert}\,dx\Bigg] 
  \leq  Ce^{-\beta\tau t/16},\label{1eq:-37}
\end{align}
by Proposition \ref{1prop:decay-mod} and Remark \ref{1rem:If-,rem}.
For the other terms of the difference to be bounded, we reason in a similar
way. This completes the proof of the claim. 
\end{proof}
\emph{Step 2:}

Therefore, when we manage to compare $\mathcal{H}[p](t)$ and $\mathcal{H}[\widetilde{P}](t)$,
we are also able to compare $\mathcal{H}[p](t)$ and
\begin{align}
\sum_{j=1}^{J}\Big(F\big[\widetilde{P_{j}}\big](t)+2\big(b_{j}^{2}-a_{j}^{2}\big)E\big[\widetilde{P_{j}}\big](t)+\big(a_{j}^{2}+b_{j}^{2}\big)^{2}M\big[\widetilde{P_{j}}\big](t)\Big).
\end{align}
 
 We compute the Taylor expansion of $\mathcal{H}[p]=\mathcal{H}[\widetilde{P}+\varepsilon]$:
\begin{align}
\mathcal{H}\big[\widetilde{P}+\varepsilon\big] & =\frac{1}{2}\int\big(\widetilde{P}+\varepsilon\big)_{xx}^{2}-\frac{5}{2}\int\big(\widetilde{P}+\varepsilon\big)^{2}\big(\widetilde{P}+\varepsilon\big)_{x}^{2}+\frac{1}{4}\int\big(\widetilde{P}+\varepsilon\big)^{6} \\
 & \quad+\sum_{j=1}^{J}\bigg[\big(b_{j}^{2}-a_{j}^{2}\big)\bigg(\int\big(\widetilde{P}+\varepsilon\big)_{x}^{2}\varphi_{j}-\frac{1}{2}\int\big(\widetilde{P}+\varepsilon\big)^{4}\varphi_{j}\bigg)\bigg]\\
&\quad +\sum_{j=1}^{J}\bigg[\big(a_{j}^{2}+b_{j}^{2}\big)^{2}\frac{1}{2}\int\big(\widetilde{P}+\varepsilon\big)^{2}\varphi_{j}\bigg] \\
&  =\frac{1}{2}\int\widetilde{P}_{xx}^{2}-\frac{5}{2}\int\widetilde{P}^{2}\widetilde{P}_{x}^{2}
  +\frac{1}{4}\int\widetilde{P}^{6}+\int\widetilde{P}_{(4x)}\varepsilon+5\int\widetilde{P}\widetilde{P}_{x}^{2}\varepsilon\\
  & \quad+5\int\widetilde{P}^{2}\widetilde{P}_{xx}\varepsilon+\frac{3}{2}\int\widetilde{P}^{5}\varepsilon 
  +\frac{1}{2}\int\varepsilon_{xx}^{2}-\frac{5}{2}\int\widetilde{P}^{2}\varepsilon_{x}^{2}\\
  &\quad +\frac{5}{2}\int\widetilde{P}_{x}^{2}\varepsilon^{2}+5\int\widetilde{P}\widetilde{P}_{xx}\varepsilon^{2}+\frac{15}{4}\int\widetilde{P}^{4}\varepsilon^{2}+O\big(\lVert\varepsilon(t)\rVert_{H^{2}}^{3}\big) \\
 & \quad+\sum_{j=1}^{J}\big(b_{j}^{2}-a_{j}^{2}\big)\bigg(\int\widetilde{P}_{x}^{2}\varphi_{j}-\frac{1}{2}\int\widetilde{P}^{4}\varphi_{j}-2\int\widetilde{P}_{xx}\varepsilon\varphi_{j} \\
 & \qquad-2\int\widetilde{P}_{x}\varepsilon\varphi_{j,x}-2\int\widetilde{P}^{3}\varepsilon\varphi_{j}+\int\varepsilon_{x}^{2}\varphi_{j}-3\int\widetilde{P}^{2}\varepsilon^{2}\varphi_{j}\bigg) \\
 & \quad+\sum_{j=1}^{J}\big(a_{j}^{2}+b_{j}^{2}\big)^{2}\frac{1}{2}\bigg(\int\widetilde{P}^{2}\varphi_{j}+2\int\widetilde{P}\varepsilon\varphi_{j}+\int\varepsilon^{2}\varphi_{j}\bigg).\label{1eq:-38}
\end{align}

We can observe that the sum (\ref{1eq:-38}) is composed of 0-order terms in $\varepsilon$,
of $1^{st}$-order terms in $\varepsilon$, of $2^{nd}$-order terms
in $\varepsilon$; $3^{rd}$ and larger-order terms in $\varepsilon$
are contained in $O(\lVert\varepsilon(t)\rVert_{H^{2}}^{3})$. The sum
of the 0-order terms is actually $\mathcal{H}[\widetilde{P}]$. The
sum of $2^{nd}$-order terms in $\varepsilon$ is $H_{2}[\varepsilon](t)$.

Let us study more closely the $1^{st}$-order terms: 
\begin{align}
H_{1} & =\int\widetilde{P}_{(4x)}\varepsilon+5\int\widetilde{P}\widetilde{P}_{x}^{2}\varepsilon+5\int\widetilde{P}^{2}\widetilde{P}_{xx}\varepsilon+\frac{3}{2}\int\widetilde{P}^{5}\varepsilon \\
 &\quad +\sum_{j=1}^{J}\big(b_{j}^{2}-a_{j}^{2}\big)\bigg(2\int\widetilde{P}_{x}\varepsilon_{x}\varphi_{j}-2\int\widetilde{P}^{3}\varepsilon\varphi_{j}\bigg)+\sum_{j=1}^{J}\big(a_{j}^{2}+b_{j}^{2}\big)^{2}\int\widetilde{P}\varepsilon\varphi_{j}.\label{1eq:-5}
\end{align}

From \cite{key-1}, we know that a breather $A=A_{\alpha,\beta}$
satisfies for any fixed $t\in\mathbb{R}$, the following nonlinear
equation: 
\begin{align}
A_{(4x)}-2\big(\beta^{2}-\alpha^{2}\big)\big(A_{xx}+A^{3}\big)+\big(\alpha^{2}+\beta^{2}\big)^{2}A+5AA_{x}^{2}+5A^{2}A_{xx}+\frac{3}{2}A^{5}=0.\label{1eq:elipt}
\end{align}
This equation is also satisfied for $A=\widetilde{B}_{k}$ with $\alpha=\alpha_{k}$
and $\beta=\beta_{k}$ for any $k=1,...,K$ (the shape parameters of
a breather are not changed by modulation).

For a soliton $Q=R_{c,\kappa}$, we know from $Q_{xx}=cQ-Q^{3}$ that
$Q$ satisfies for any fixed $t\in\mathbb{R}$, the following nonlinear
equation (see Section \ref{1sec:51} (Appendix)):
\begin{align}
Q_{(4x)}-2c\big(Q_{xx}+Q^{3}\big)+c^{2}Q+5QQ_{x}^{2}+5Q^{2}Q_{xx}+\frac{3}{2}Q^{5}=0.\label{1eq:elipts}
\end{align}
This equation is not exactly satisfied for $Q=\widetilde{R_{l}}$
for any $l=1,...,L$ (the shape parameters of a soliton are changed
by modulation). The exact equation satisfied by $Q=\widetilde{R_{l}}$
is:
\begin{align}
\MoveEqLeft Q_{(4x)}-2c_{l}\big(Q_{xx}+Q^{3}\big)+c_{l}^{2}Q+5QQ_{x}^{2}+5Q^{2}Q_{xx}+\frac{3}{2}Q^{5}\\
& =2c_{0,l}(t)\big(Q_{xx}+Q^{3}\big)-2c_{l}c_{0,l}(t)Q-c_{0,l}(t)^{2}Q.\label{1eq:eliptsm}
\end{align}

We will compare $H_{1}$ and 
\begin{align}
H_{1}' & :=\int\widetilde{P}_{(4x)}\varepsilon+5\sum_{j=1}^{J}\int\widetilde{P_{j}}\widetilde{P_{j}}_{x}^{2}\varepsilon+5\sum_{j=1}^{J}\int\widetilde{P_{j}}_{x}^{2}\widetilde{P_{j}}_{xx}\varepsilon+\frac{3}{2}\sum_{j=1}^{J}\int\widetilde{P_{j}}^{5}\varepsilon \\
 & -2\sum_{j=1}^{J}\big(b_{j}^{2}-a_{j}^{2}\big)\bigg(\int\widetilde{P_{j}}_{xx}\varepsilon+\int\widetilde{P_{j}}^{3}\varepsilon\bigg)+\sum_{j=1}^{J}\big(a_{j}^{2}+b_{j}^{2}\big)^{2}\int\widetilde{P_{j}}\varepsilon.\label{1eq:-6}
\end{align}

Firstly, let us compare $\int\widetilde{P}\widetilde{P}_{x}^{2}\varepsilon$
and $\sum_{j=1}^{J}\int\widetilde{P_{j}}\widetilde{P_{j}}_{x}^{2}\varepsilon$:
\begin{align}
\int\widetilde{P}\widetilde{P}_{x}^{2}\varepsilon & =\int\bigg(\sum_{j=1}^{J}\widetilde{P_{j}}\bigg)\bigg(\sum_{j=1}^{J}\widetilde{P_{j}}_{x}\bigg)^{2}\varepsilon \\
 & =\sum_{j=1}^{J}\int\widetilde{P_{j}}\widetilde{P_{j}}_{x}^{2}\varepsilon+\sum_{h\neq i\textrm{ or }i\neq j}\int\widetilde{P_{h}}\widetilde{P_{i}}_{x}\widetilde{P_{j}}_{x}\varepsilon.\label{1eq:-39}
\end{align}
To succeed, we need to find a bound for a term of the type $\int\widetilde{P_{h}}\widetilde{P_{i}}_{x}\widetilde{P_{j}}_{x}\varepsilon$
where $h\neq i$ or $i\neq j$. We can perform the following upper bounding
(where without loss of generality, we suppose that $i\neq j$): 
\begin{align}
\bigg\lvert \int\widetilde{P_{h}}\widetilde{P_{i}}_{x}\widetilde{P_{j}}_{x}\varepsilon\bigg\rvert & \leq  Ce^{\frac{\beta\tau}{16}t}\int e^{-\frac{\beta}{2}\lvert x-v_{i}t\rvert}e^{-\frac{\beta}{2}\lvert x-v_{j}t\rvert}\lvert\varepsilon\rvert \\
 & \leq  C\lVert\varepsilon\rVert_{L^{\infty}}e^{\frac{\beta\tau}{16}t}\int e^{-\frac{\beta}{2}\lvert x-v_{i}t\rvert}e^{-\frac{\beta}{2}\lvert x-v_{j}t\rvert} \\
 & \leq  C\lVert\varepsilon\rVert_{H^{2}}e^{-\beta\tau t/8},\label{1eq:-40}
\end{align}
by Sobolev embeddings and Proposition \ref{1prop:cross-product}.

The bounding is quite similar for $\int\widetilde{P}^{2}\widetilde{P}_{xx}\varepsilon$
and $\int\widetilde{P}^{5}\varepsilon$. We observe that $-\int\widetilde{P_{j}}_{xx}\varepsilon=\int\widetilde{P_{j}}_{x}\varepsilon_{x}$.
To compare $\int\widetilde{P}_{x}\varepsilon_{x}\varphi_{j}$ and
$\int_{\mathbb{R}}\widetilde{P_{j}}_{x}\varepsilon_{x}$, and for similar
terms, we can use computations that we have already
performed at the beginning of this proof. Therefore, 
\begin{align}
\bigg\lvert\int\widetilde{P}_{x}\varepsilon_{x}\varphi_{j}-\int_{\mathbb{R}}\widetilde{P_{j}}_{x}\varepsilon_{x}\bigg\rvert\leq  C\lVert\varepsilon\rVert_{H^{2}}e^{-\frac{\beta\tau t}{16}}.\label{1eq:-41}
\end{align}
This enables us to bound the difference between $H_{1}$ and $H_{1}'$:
\begin{align}
\big\lvert H_{1}-H_{1}'\big\rvert\leq  C\lVert\varepsilon(t)\rVert_{H^{2}}e^{-\frac{\beta\tau t}{16}}.\label{1eq:-42}
\end{align}
Now, because our objects are not only breathers, $H_{1}'$ is not
equal to $0$. Actually, we have that
\begin{align}
\MoveEqLeft H_{1}'=2\sum_{l=1}^{L}c_{0,l}(t)\bigg(\int\widetilde{R_{l}}_{xx}\varepsilon+\int\widetilde{R_{l}}^{3}\varepsilon\bigg)\\
& -2\sum_{l=1}^{L}c_{l}c_{0,l}(t)\int\widetilde{R_{l}}\varepsilon-\sum_{l=1}^{L}c_{0,l}(t)^{2}\int\widetilde{R_{l}}\varepsilon.\label{1eq:-43}
\end{align}

Now, we introduce:
\begin{align}
\MoveEqLeft H_{1}''=2\sum_{l=1}^{L}c_{0,l}(t)\bigg(\int\widetilde{R_{l}}_{xx}\varepsilon\sqrt{\varphi_{l}^{s}}+\int\widetilde{R_{l}}^{3}\varepsilon\sqrt{\varphi_{l}^{s}}\bigg)\\
& -2\sum_{l=1}^{L}c_{l}c_{0,l}(t)\int\widetilde{R_{l}}\varepsilon\sqrt{\varphi_{l}^{s}}-\sum_{l=1}^{L}c_{0,l}(t)^{2}\int\widetilde{R_{l}}\varepsilon\sqrt{\varphi_{l}^{s}}.\label{1eq:-44}
\end{align}
By reasonning the same way as for $H_1$ and $H_{1}'$, we see that
\begin{align}
\big\lvert H_{1}'-H_{1}''\big\rvert\leq  C\lVert\varepsilon(t)\rVert_{H^{2}}e^{-2\theta t}.\label{1eq:-45}
\end{align}
Because of (\ref{1eq:orthom}) and
because of the elliptic equation satisfied by a soliton, we have that
\begin{align}
H_{1}''=0.\label{1eq:-46}
\end{align}

Thus,
\begin{align}
\lvert H_{1}\rvert=\lvert H_{1}-H_{1}'\rvert+\lvert H_{1}'-H_{1}''\rvert+\lvert H_{1}''\rvert\leq  C\lVert\varepsilon(t)\rVert_{H^{2}}e^{-2\theta t}.\label{1eq:-47}
\end{align}
The proof of Proposition \ref{1lem:exp} is now completed. 
\end{proof}
Now, we would like to study the quadratic terms in $\varepsilon$
of the development of $\mathcal{H}[\widetilde{P}+\varepsilon]$. They
are contained in $H_{2}[\text{\ensuremath{\varepsilon}}](t)$.

Let $A=B_{\alpha,\beta}$ be a breather (we note $A_{1}:=\partial_{x_{1}}A$
and $A_{2}:=\partial_{x_{2}}A$). We define a quadratic form associated
to this breather: 
\begin{align}
\mathcal{Q}_{\alpha,\beta}^{b}[\epsilon] & :=\frac{1}{2}\int\epsilon_{xx}^{2}-\frac{5}{2}\int A^{2}\epsilon_{x}^{2}+\frac{5}{2}\int A_{x}^{2}\epsilon^{2}+5\int AA_{xx}\epsilon^{2}+\frac{15}{4}\int A^{4}\epsilon^{2} \\
 &\quad+\big(\beta^{2}-\alpha^{2}\big)\bigg(\int\epsilon_{x}^{2}-3\int A^{2}\epsilon^{2}\bigg)+\big(\alpha^{2}+\beta^{2}\big)^{2}\frac{1}{2}\int\epsilon^{2}=:\mathcal{Q}_{\alpha,\beta}[\epsilon].\label{1eq:-7}
\end{align}
From \cite{key-1}, we know that the kernel of this quadratic form
is of dimension 2 and is spanned by $\partial_{x_{1}}B_{\alpha,\beta}$
and $\partial_{x_{2}}B_{\alpha,\beta}$, and that this quadratic form
has only one negative eigenvalue that is of multiplicity 1: 
\begin{prop}[Proposition 4.11, \cite{key-3}]
\label{1fact:coerc}There exists $\mu_{\alpha,\beta}^{b}>0$ that
depends only on $\alpha$ and $\beta$ (and does not depend on time), such
that if $\epsilon\in H^{2}(\mathbb{R})$ is such that
\begin{align}
\int A_{1}\epsilon=\int A_{2}\epsilon=0,\label{1eq:-48}
\end{align}
then 
\begin{align}
\mathcal{Q}_{\alpha,\beta}^{b}[\epsilon]\geq \mu_{\alpha,\beta}^{b}\lVert\epsilon\rVert_{H^{2}}^{2}-\frac{1}{\mu_{\alpha,\beta}^{b}}\bigg(\int\epsilon A\bigg)^{2}.\label{1eq:-49}
\end{align}
\end{prop}

\begin{rem}
$\mu_{\alpha,\beta}^{b}$ is continuous in $\alpha,\beta$. Note that
translation parameters are implicit in $\mathcal{Q}_{\alpha,\beta}^{b}$.
\end{rem}
Let $Q=R_{c,\kappa}$ be a soliton. We define a quadratic form associated
to this soliton: 
\begin{align}
\mathcal{Q}_{c}^{s}[\epsilon] & :=\frac{1}{2}\int\epsilon_{xx}^{2}-\frac{5}{2}\int Q^{2}\epsilon_{x}^{2}+\frac{5}{2}\int Q_{x}^{2}\epsilon^{2}+5\int QQ_{xx}\epsilon^{2}+\frac{15}{4}\int Q^{4}\epsilon^{2} \\
 &\quad +c\bigg(\int\epsilon_{x}^{2}-3\int Q^{2}\epsilon^{2}\bigg)+c^{2}\frac{1}{2}\int\epsilon^{2}=:\mathcal{Q}_{0,\sqrt{c}}[\epsilon].\label{1eq:-50}
\end{align}
By the same techniques, such as those presented in \cite{key-1},
adapted to the quadratic form of a soliton, we may establish that the
kernel of this quadratic form is of dimension 2, and is spanned
by $\partial_{x}Q$ and $\partial_{c}Q$, and that this quadratic
form does not have any negative eigenvalue (see Section \ref{1coer_sol} (Appendix)).
After that, from Section \ref{1sec:53} (Appendix), we deduce that the coercivity
still works when $\epsilon$ is orthogonal to $Q$ and $\partial_{x}Q$. More
precisely:
\begin{prop}
There exists $\mu_{c}^{s}>0$ that depends only on $c$ (and does not
depend on time), such that if $\epsilon\in H^{2}(\mathbb{R})$ is such that
\begin{align}
\int Q\epsilon=\int Q_{x}\epsilon=0,\label{1eq:-51}
\end{align}
then 
\begin{align}
\mathcal{Q}_{c}^{s}[\epsilon]\geq \mu_{c}^{s}\lVert\epsilon\rVert_{H^{2}}^{2}.\label{1eq:-52}
\end{align}
\end{prop}

\begin{rem}
$\mu_{c}^{s}$ is continuous in $c$. Note that translation and sign
parameters are implicit in the notation $\mathcal{Q}_{c}^{s}$.
\end{rem}

We would like to find a similar minoration for $H_{2}$ (which is
a generalization of $\mathcal{Q}$).

For $j=1,...,J$, let us define for $\epsilon\in H^{2}$, 
\begin{align}
\mathcal{Q}_{j}[\epsilon] & :=\frac{1}{2}\int\epsilon_{xx}^{2}\varphi_{j}-\frac{5}{2}\int\widetilde{P_{j}}^{2}\epsilon_{x}^{2}\varphi_{j}+\frac{5}{2}\int\widetilde{P_{j}}_{x}^{2}\epsilon^{2}\varphi_{j}\\
&\quad +5\int\widetilde{P_{j}}\widetilde{P_{j}}_{xx}\epsilon^{2}\varphi_{j}+\frac{15}{4}\int\widetilde{P_{j}}^{4}\epsilon^{2}\varphi_{j} \\
 &\quad +\big(b_{j}^{2}-a_{j}^{2}\big)\bigg(\int\epsilon_{x}^{2}\varphi_{j}-3\int\widetilde{P_{j}}^{2}\epsilon^{2}\varphi_{j}\bigg)+\big(a_{j}^{2}+b_{j}^{2}\big)^{2}\frac{1}{2}\int\epsilon^{2}\varphi_{j},\label{1eq:-9}
\end{align}
and
\begin{align}
\mathcal{Q}_{j}'[\epsilon] & :=\frac{1}{2}\int\epsilon_{xx}^{2}\varphi_{j}-\frac{5}{2}\int\widetilde{P}^{2}\epsilon_{x}^{2}\varphi_{j}+\frac{5}{2}\int\widetilde{P}_{x}^{2}\epsilon^{2}\varphi_{j}\\
&\quad +5\int\widetilde{P}\widetilde{P}_{xx}\epsilon^{2}\varphi_{j}+\frac{15}{4}\int\widetilde{P}^{4}\epsilon^{2}\varphi_{j} \\
 &\quad +\big(b_{j}^{2}-a_{j}^{2}\big)\bigg(\int\epsilon_{x}^{2}\varphi_{j}-3\int\widetilde{P}^{2}\epsilon^{2}\varphi_{j}\bigg)+\big(a_{j}^{2}+b_{j}^{2}\big)^{2}\frac{1}{2}\int\epsilon^{2}\varphi_{j}.\label{1eq:-53}
\end{align}
We have that
\begin{align}
H_{2}[\varepsilon(t)]=\sum_{j=1}^{J}\mathcal{Q}_{j}'[\varepsilon(t)].\label{1eq:-54}
\end{align}
Notations $\mathcal{Q}_{k}^{b}$, $(\mathcal{Q}_{k}^{b})'$,
$\mathcal{Q}_{l}^{s}$ and $(\mathcal{Q}_{l}^{s})'$ will
also be used.

We note that the support of $\varphi_{j}$ increases with time, so that $\mathcal{Q}_{j}$
is near a $\mathcal{Q}_{\alpha_{k},\beta_{k}}^{b}$ or a $\mathcal{Q}_{c_{l}}^{s}$
when time is large (note that $\mathcal{Q}_{\alpha_{k},\beta_{k}}^{b}$
is the canonical quadratic form associated to the breather $\widetilde{B_{k}}$,
but the canonical quadratic form associated to the soliton $\widetilde{R_{c}}$
is $\mathcal{Q}_{c_{l}+c_{0,l}(t)}^{s}$). However, firstly, let us study
the difference between $\mathcal{Q}{}_{j}$ and $\mathcal{Q}'_{j}$.
Using the computations carried out at the beginning of this part
(those done for the linear part) and Sobolev inequalities, we obtain: 
\begin{align}
\big\lvert\mathcal{Q}_{j}[\epsilon]-\mathcal{Q}_{j}'[\epsilon]\big\rvert\leq  Ce^{-2\theta t}\lVert\epsilon\rVert_{H^{2}(\mathbb{R})}^{2}.\label{1eq:quad-approx}
\end{align}

\begin{lem}
\label{1lem:coerc}There exists $\mu>0$ such that for $\rho>0$, there
exists $T_{3}^{*}$ such that, if $T^{*}\geq  T_{3}^{*}$, for any
$\epsilon\in H^{2}(\mathbb{R})$, for any $t\in[t^{*},T]$,

if 
\begin{align}
\int\widetilde{B_{k}}_{1}(t)\epsilon\sqrt{\varphi_{k}^{b}(t)}=\int\widetilde{B_{k}}_{2}(t)\epsilon\sqrt{\varphi_{k}^{b}(t)}=0,\label{1eq:-55}
\end{align}

then 
\begin{align}
\mathcal{Q}_{k}^{b}[\epsilon]\geq \mu\int\big(\epsilon^{2}+\epsilon_{x}^{2}+\epsilon_{xx}^{2}\big)\varphi_{k}^{b}(t)-\frac{1}{\mu}\bigg(\int\epsilon\widetilde{B_{k}}(t)\sqrt{\varphi_{k}^{b}(t)}\bigg)^{2}-\rho\lVert\epsilon\rVert_{H^{2}}^{2}.\label{1eq:-23}
\end{align}
\end{lem}

\begin{proof}[Proof of Lemma \ref{1lem:coerc}]
The idea is to write $\mathcal{Q}_{k}^{b}[\epsilon]$ as $\mathcal{Q}_{\alpha_{k},\beta_{k}}[\epsilon\sqrt{\varphi_{k}^{b}}]$
plus several error terms. Let $j$ such that $\widetilde{P_{j}}=\widetilde{B_{k}}$.
We will denote $\varphi_{1,j}:=\psi'(\frac{x-\sigma_{j}t}{\delta t})-\psi'(\frac{x-\sigma_{j+1}t}{\delta t})$
and $\varphi_{2,j}:=\psi''(\frac{x-\sigma_{j}t}{\delta t})-\psi''(\frac{x-\sigma_{j+1}t}{\delta t})$,
as defined by \eqref{1eq:phid} and \eqref{1eq:phid-1},
which will be useful to write the derivatives of $\varphi_{j}$. We
recall that they have the same support and bounding properties as
$\varphi_{j}$. We have that
\begin{align}
\int\big(\epsilon\sqrt{\varphi_{j}}\big)_{xx}^{2} & =\int\epsilon_{xx}^{2}\varphi_{j}+\int\frac{\epsilon_{x}^{2}}{(\delta t)^{2}}\frac{\varphi_{1,j}^{2}}{\varphi_{j}}+\frac{1}{4}\int\frac{\epsilon^{2}}{(\delta t)^{4}}\frac{\varphi_{2,j}^{2}}{\varphi_{j}}+\frac{1}{16}\int\frac{\epsilon^{2}}{(\delta t)^{4}}\frac{\varphi_{1,j}^{4}}{\varphi_{j}^{3}}\\
&\quad -\frac{1}{4}\int\frac{\epsilon^{2}}{(\delta t)^{4}}\frac{\varphi_{2,j}\varphi_{1,j}^{2}}{\varphi_{j}^{2}} 
  +2\int\frac{\epsilon_{xx}\epsilon_{x}}{\delta t}\varphi_{1,j}+\int\frac{\epsilon_{xx}\epsilon}{(\delta t)^{2}}\varphi_{2,j}\\
  &\quad -\frac{1}{2}\int\frac{\epsilon_{xx}\epsilon}{(\delta t)^{2}}\frac{\varphi_{1,j}^{2}}{\varphi_{j}}+\int\frac{\epsilon_{x}\epsilon}{(\delta t)^{3}}\frac{\varphi_{1,j}\varphi_{2,j}}{\varphi_{j}}-\frac{1}{2}\int\frac{\epsilon_{x}\epsilon}{(\delta t)^{3}}\frac{\varphi_{1,j}^{3}}{\varphi_{j}^{2}}.\label{1eq:-56}
\end{align}
We observe that, for $T_{3}^{*}$ large enough, and by using the inequalities
that define $\psi$, the error terms can be bounded by $\frac{C}{\delta t}\lVert\epsilon\rVert_{H^{2}}^{2}\leq \frac{\rho}{100}\lVert\epsilon\rVert_{H^{2}}^{2}$.
The computation for the other terms is similar and the same bound
can be used for the error terms.

Because $\epsilon\sqrt{\varphi_{k}^{b}}$ satisfies the orthogonality
conditions,
we can apply Proposition \ref{1fact:coerc}, and obtain that
\begin{align}
\mathcal{Q}_{\alpha_{k},\beta_{k}}\Big[\epsilon\sqrt{\varphi_{k}^{b}}\Big]\geq \mu_{k}^{b}\Big\lVert\epsilon\sqrt{\varphi_{k}^{b}}\Big\rVert_{H^{2}}^{2}-\frac{1}{\mu_{k}^{b}}\bigg(\int\epsilon\sqrt{\varphi_{k}^{b}}\widetilde{B_{k}}\bigg)^{2}.\label{1eq:-57}
\end{align}

To finish, $\lVert\epsilon\sqrt{\varphi_{k}^{b}}\rVert_{H^{2}}^{2}$
is $\int(\epsilon^{2}+\epsilon_{x}^{2}+\epsilon_{xx}^{2})\varphi_{k}^{b}(t)$
plus several error terms as in (\ref{1eq:-56}).
\end{proof}
\begin{lem}
\label{1lem:ver_sol}
There exists $\mu>0$ such that for $\rho>0$, there exists $T_{3}^{*}$
such that if $T^{*}\geq  T_{3}^{*}$, then for any $\epsilon\in H^{2}(\mathbb{R})$,
for any $t\in[t^{*},T]$, we have that

if
\begin{align}
\int\widetilde{R_{l}}(t)\epsilon\sqrt{\varphi_{l}^{s}(t)}=\int\widetilde{R_{l}}_{x}(t)\epsilon\sqrt{\varphi_{l}^{s}(t)}=0,\label{1eq:-58}
\end{align}
then
\begin{align}
\mathcal{Q}_{l}^{s}[\epsilon]\geq \mu\int\Big(\epsilon^{2}+\epsilon_{x}^{2}+\epsilon_{xx}^{2}\Big)\varphi_{l}^{s}(t)-\rho\lVert\epsilon\rVert_{H^{2}}^{2}.\label{1eq:-59}
\end{align}
\end{lem}

\begin{proof}
As in the previous proof, we write $\mathcal{Q}_{l}^{s}[\epsilon]$
as $\mathcal{\mathcal{Q}}_{c_{l}}^{s}[\epsilon\sqrt{\varphi_{l}^{s}}]$
(with $Q=\widetilde{R_{l}}$) plus several error terms, that are all
bounded by $\rho\lVert\epsilon\rVert_{H^{2}}^{2}$ if $T_{3}^{*}$ is
chosen large enough. However, $\mathcal{\mathcal{Q}}_{c_{l}}^{s}[\epsilon\sqrt{\varphi_{l}^{s}}]$
is not appropriate in order to have coercivity, the appropriate quadratic form is $\mathcal{\mathcal{Q}}_{c_{l}+c_{0,l}(t)}^{s}[\epsilon\sqrt{\varphi_{l}^{s}}]$. This is why, we need to bound the difference
between $\mathcal{\mathcal{Q}}_{c_{l}}^{s}[\epsilon\sqrt{\varphi_{l}^{s}}]$
and $\mathcal{\mathcal{Q}}_{c_{l}+c_{0,l}(t)}^{s}[\epsilon\sqrt{\varphi_{l}^{s}}]$.
This difference is
\begin{align}
c_{0,l}(t)\bigg(\int\epsilon_{x}^{2}\varphi_{j}-3\int\widetilde{R_{l}}^{2}\epsilon^{2}\varphi_{j}\bigg)+c_{l}c_{0,l}(t)\int\epsilon^{2}\varphi_{j}+c_{0,l}(t)^{2}\frac{1}{2}\int\epsilon^{2}\varphi_{j},\label{1eq:-60}
\end{align}
which can, because of the bound for $c_{0,l}(t)$, for $T_{3}^{*}$
large enough (depending on $A$), be bounded by $\rho\lVert\epsilon\rVert_{H^{2}}^{2}$.

Now, $\epsilon\sqrt{\varphi_{l}^{s}}$ satisfies the orthogonality
conditions we need, and as in the previous proof we may apply coercivity.
\end{proof}
\begin{proof}[Proof of Proposition \ref{1lem:coerc-1}]
We will now use the Lemma \ref{1lem:coerc} and its version for solitons (Lemma \ref{1lem:ver_sol})
for $\epsilon=\varepsilon(t)$. From this, we deduce that for $\rho>0$
small enough, we have that
\begin{align}
\sum_{j=1}^{J}\mathcal{Q}_{j}[\varepsilon(t)]\geq \mu\lVert\varepsilon(t)\rVert_{H^{2}}^{2}-\frac{1}{\mu}\sum_{k=1}^{K}\bigg(\int\varepsilon(t)\widetilde{B_{k}}\sqrt{\varphi_{k}^{b}}\bigg)^{2},\label{1eq:-61}
\end{align}
for a suitable constant $\mu>0$. This means that for $T_{3}^{*}$
large enough, by taking, if needed, a smaller constant $\mu$, 
\begin{align}
H_{2}[\varepsilon(t)]\geq \mu\lVert\varepsilon\rVert_{H^{2}}^{2}-\frac{1}{\mu}\sum_{k=1}^{K}\bigg(\int\varepsilon\widetilde{B_{k}}\sqrt{\varphi_{k}^{b}}\bigg)^{2}.\label{1eq:-62}
\end{align}
The proof of Proposition \ref{1lem:coerc-1} is now completed. 
\end{proof}

\subsection{Proof of Proposition \ref{1prop:bootstrap} (Bootstrap)}
\label{1sec:existence}

We recall that $p_{n}$ from Proposition \ref{1prop:bootstrap} is
denoted by $p$ and $T_{n}$ is denoted by $T$ in what follows, in
order to simplify the notations. We do the proof that follows under
the assumption (\ref{1eq:-28}), so that the Propositions proved above are true for $t\in[t^{*},T]$.

The aim of this subsection is to complete the proof of Proposition \ref{1prop:bootstrap}
by using the Propositions \ref{1lem:exp} and \ref{1lem:coerc-1}.

We note that by Lemma \ref{1lem:lcq}, the conservation of $F[p](t)$
and the definition of $\mathcal{H}[p]$, we have for any $t\in[t^{*},T]$, that
\begin{align}
\big\lvert\mathcal{H}[p](T)-\mathcal{H}[p](t)\big\rvert\leq \frac{CA^{2}}{\delta^{2}t}e^{-2\theta t}.\label{1eq:-63}
\end{align}
Thus, for any $t\in[t^{*},T]$, 
\begin{align}
\mathcal{H}[p](t)\leq \mathcal{H}[p](T)+\frac{CA^{2}}{\delta^{2}t}e^{-2\theta t}.\label{1eq:vart}
\end{align}

From Proposition \ref{1lem:exp},
\begin{align}
\MoveEqLeft\bigg\lvert\mathcal{H}\big[\widetilde{P}+\varepsilon\big](t)-H_{2}[\varepsilon](t)\\
&\quad -\sum_{j=1}^{J}\Big(F\big[\widetilde{P_{j}}\big](t)+2\big(b_{j}^{2}-a_{j}^{2}\big)E\big[\widetilde{P_{j}}\big](t)+\big(a_{j}^{2}+b_{j}^{2}\big)^{2}M\big[\widetilde{P_{j}}\big](t)\Big)\bigg\rvert \\
 & \leq  Ce^{-2\theta t}+C\lVert\varepsilon\rVert_{H^{2}}e^{-2\theta t}+C\lVert\varepsilon\rVert_{H^{2}}^{3} 
  \leq  Ce^{-2\theta t}+\frac{\mu}{100}\lVert\varepsilon\rVert_{H^{2}}^{2}.\label{1eq:varl}
\end{align}
In order to obtain the last line, we use the fact that $\lVert\varepsilon(t)\rVert_{H^{2}}\leq  CAe^{-\theta t}$,
and we take $T^{*}\geq  T_{5}^{*}$ for $T_{5}^{*}$ large enough
(depending on $A$) so that $\lVert\varepsilon\rVert_{H^{2}}\leq  C$ and $C\lVert\varepsilon(t)\rVert_{H^{2}}\leq \frac{\mu}{100}$,
and thus $C\lVert\varepsilon(t)\rVert_{H^{2}}^{3}\leq \frac{\mu}{100}\lVert\varepsilon(t)\rVert_{H^{2}}^{2}$.

We remark that if $P_{j}=B_k$ is a breather, then $F[\widetilde{P_{j}}]$,
$E[\widetilde{P_{j}}]$ and $M[\widetilde{P_{j}}]$ are all constants
in time. If $P_{j}=R_{l}$ is a soliton and we denote $q$ the basic
ground state (i.e. the ground state for $c=1$), we have the following:
\begin{align}
M[\widetilde{R_{l}}](t)=\big(c_{l}+c_{0,l}(t)\big)^{1/2}M[q],\label{1eq:msol}
\end{align}
\begin{align}
E[\widetilde{R_{l}}](t)=\big(c_{l}+c_{0,l}(t)\big)^{3/2}E[q],\label{1eq:esol}
\end{align}
\begin{align}
F[\widetilde{R_{l}}](t)=\big(c_{l}+c_{0,l}(t)\big)^{5/2}F[q].\label{1eq:fsol}
\end{align}
Using that, we can simplify $\mathcal{R}_{l}(t):=F[\widetilde{R_{l}}](t)+2c_{l}E[\widetilde{R_{l}}](t)+c_{l}^{2}M[\widetilde{R_{l}}](t)$
as follows:
\begin{align}
\mathcal{R}_{l}(t) & =\big(c_{l}+c_{0,l}(t)\big)^{5/2}F[q]+2c_{l}\big(c_{l}+c_{0,l}(t)\big)^{3/2}E[q]\\
&\quad +c_{l}^{2}\big(c_{l}+c_{0,l}(t)\big)^{1/2}M[q] \\
 & =c_{l}^{5/2}\bigg(1+\frac{c_{0,l}(t)}{c_{l}}\bigg)^{5/2}F[q]+2c_{l}^{5/2}\bigg(1+\frac{c_{0,l}(t)}{c_{l}}\bigg)^{3/2}E[q]\\
 &\quad +c_{l}^{5/2}\bigg(1+\frac{c_{0,l}(t)}{c_{l}}\bigg)^{1/2}M[q].\label{1eq:var}
\end{align}
Note that from Lemma \ref{1lem:mod}, $\lvert c_{0,l}(t)\rvert^{3}\leq  CA^{3}e^{-\theta t}e^{-2\theta t}$.
That is why, if we take $T_{5}^{*}$ eventually larger, $\lvert c_{0,l}(t)\rvert^{3}\leq  Ce^{-2\theta t}$.
For this reason, we will do Taylor expansions of order $2$ of (\ref{1eq:var}):
\begin{align}
\bigg(1+\frac{c_{0,l}(t)}{c_{l}}\bigg)^{5/2}=1+\frac{5}{2}\frac{c_{0,l}(t)}{c_{l}}+\frac{15}{8}\frac{c_{0,l}(t)^{2}}{c_{l}^{2}}+O\big(e^{-2\theta t}\big),\label{1eq:-64}
\end{align}
\begin{align}
\bigg(1+\frac{c_{0,l}(t)}{c_{l}}\bigg)^{3/2}=1+\frac{3}{2}\frac{c_{0,l}(t)}{c_{l}}+\frac{3}{8}\frac{c_{0,l}(t)^{2}}{c_{l}^{2}}+O\big(e^{-2\theta t}\big),\label{1eq:-65}
\end{align}
\begin{align}
\bigg(1+\frac{c_{0,l}(t)}{c_{l}}\bigg)^{1/2}=1+\frac{1}{2}\frac{c_{0,l}(t)}{c_{l}}-\frac{1}{8}\frac{c_{0,l}(t)^{2}}{c_{l}^{2}}+O\big(e^{-2\theta t}\big).\label{1eq:-66}
\end{align}
This allows us to write:
\begin{align}
\mathcal{R}_{l}(t) & =c_{l}^{5/2}\big(F[q]+2E[q]+M[q]\big)+c_{l}^{3/2}c_{0,l}(t)\bigg(\frac{5}{2}F[q]+3E[q]+\frac{1}{2}M[q]\bigg) \\
&\quad +c_{l}^{1/2}c_{0,l}(t)^{2}\bigg(\frac{15}{8}F[q]+\frac{3}{4}E[q]-\frac{1}{8}M[q]\bigg)+O\big(e^{-2\theta t}\big).\label{1eq:dlvar}
\end{align}
Now, $c_{l}^{5/2}(F[q]+2E[q]+M[q])$ is constant in time.
For both other terms, we use that $M[q]=2$, $E[q]=-\frac{2}{3}$
and $F[q]=\frac{2}{5}$, and we see that $\frac{5}{2}F[q]+3E[q]+\frac{1}{2}M[q]=0$
and $\frac{15}{8}F[q]+\frac{3}{4}E[q]-\frac{1}{8}M[q]=0$. This allows
us to write:
\begin{align}
\mathcal{R}_{l}(t)=\frac{16}{15}c_{l}^{5/2}+O\big(e^{-2\theta t}\big).\label{1eq:dlvarsimpl}
\end{align}
From this, we deduce that
\begin{align}
\mathcal{R}_{l}(t)-\mathcal{R}_{l}(T)=O\big(e^{-2\theta t}\big).\label{1eq:solit}
\end{align}

By using that $\mathcal{H}[p](T)=\mathcal{H}[P](T)=\mathcal{H}[\widetilde{P}](T)$,
the equations (\ref{1eq:varl}) and (\ref{1eq:vart}), Claim \ref{1lem:simpl1},
and the fact that for $t\geq  T_{4}^{*}$, $O(\lVert\varepsilon(t)\rVert_{H^{2}}^{3})\leq \frac{\mu}{100}\lVert\varepsilon\rVert_{H^{2}}^{2}$,
we have that
\begin{align}
\MoveEqLeft H_{2}[\varepsilon](t)\\ & \leq \mathcal{H}[p](t)+Ce^{-2\theta t}+\frac{\mu}{100}\lVert\varepsilon(t)\rVert_{H^{2}}^{2}\\
&\quad -\sum_{j=1}^{J}\Big(F\big[\widetilde{P_{j}}\big](t)+2\big(b_{j}^{2}-a_{j}^{2}\big)E\big[\widetilde{P_{j}}\big](t)+\big(a_{j}^{2}+b_{j}^{2}\big)^{2}M\big[\widetilde{P_{j}}\big](t)\Big) \\
 & \leq \mathcal{H}[\widetilde{P}](T)+C\bigg(\frac{A^{2}}{\delta^{2}t}+1\bigg)e^{-2\theta t}+\frac{\mu}{100}\lVert\varepsilon(t)\rVert_{H^{2}}^{2}  \\
 & \quad-\sum_{j=1}^{J}\Big(F\big[\widetilde{P_{j}}\big](t)+2\big(b_{j}^{2}-a_{j}^{2}\big)E\big[\widetilde{P_{j}}\big](t)+\big(a_{j}^{2}+b_{j}^{2}\big)^{2}M\big[\widetilde{P_{j}}\big](t)\Big)\\
 & \leq \mathcal{H}\big[\widetilde{P}\big](T)+C\bigg(\frac{A^{2}}{\delta^{2}t}+1\bigg)e^{-2\theta t}+\frac{\mu}{100}\lVert\varepsilon(t)\rVert_{H^{2}}^{2}+\sum_{l=1}^{L}\big(\mathcal{R}_{l}(T)-\mathcal{R}_{l}(t)\big) \\
 & \quad -\sum_{j=1}^{J}\Big(F\big[\widetilde{P_{j}}\big](T)+2\big(b_{j}^{2}-a_{j}^{2}\big)E\big[\widetilde{P_{j}}\big](T)+\big(a_{j}^{2}+b_{j}^{2}\big)^{2}M\big[\widetilde{P_{j}}\big](T)\Big)\\
 & \leq  C\bigg(\frac{A^{2}}{\delta^{2}t}+1\bigg)e^{-2\theta t}+\frac{\mu}{100}\lVert\varepsilon(t)\rVert_{H^{2}}^{2}.\label{1eq:quad-maj}
\end{align}

From Proposition \ref{1lem:coerc-1}, we deduce (by taking a smaller
constant $\mu$) that
\begin{align}
\mu\lVert\varepsilon\rVert_{H^{2}}^{2}\leq  C\bigg(\frac{A^{2}}{\delta^{2}t}+1\bigg)e^{-2\theta t}+\frac{1}{\mu}\sum_{k=1}^{K}\bigg(\int\varepsilon\widetilde{B_{k}}\sqrt{\varphi_{k}^{b}}\bigg)^{2}.\label{1eq:maj}
\end{align}

We will now need to establish a result close to Lemma \ref{1lem:lcq}.
We set for any $j=1,...,J$: 
\begin{align}
m_{j}(t):=\int\frac{1}{2}p^{2}(t,x)\sqrt{\varphi_{j}(t,x)}\,dx:=m_{j}[p](t).\label{1eq:-67}
\end{align}

\begin{lem}
\label{1lem:lcqsqrt}There exists $C>0$, $T_{6}^{*}=T_{6}^{*}(A)$
such that, if $T^{*}\geq  T_{6}^{*}$, for any $j=1,...,J$, for any
$t\in[t^{*},T]$, 
\begin{align}
\big\lvert m_{j}(T)-m_{j}(t)\big\rvert\leq \frac{C}{\delta^{2}t}A^{2}e^{-2\theta t}.\label{1eq:-68}
\end{align}
\end{lem}

\begin{proof}
We compute:
\begin{align}
\MoveEqLeft\frac{d}{dt}\int\frac{1}{2}p^{2}(t,x)\sqrt{\varphi_{j}(t,x)}\,dx\\ & =\frac{1}{2\delta t}\int\bigg(-\frac{3}{2}p_{x}^{2}+\frac{3}{4}p^{4}\bigg)\frac{\varphi_{1,j}}{\sqrt{\varphi_{j}}}-\frac{1}{2(\delta t)^{2}}\int p_{x}p\frac{\varphi_{2,j}}{\sqrt{\varphi_{j}}} \\
 &\quad +\frac{1}{4(\delta t)^{2}}\int p_{x}p\frac{\varphi_{1,j}^{2}}{\varphi_{j}^{3/2}}-\frac{1}{4}\int p^{2}\frac{x}{\delta t^{2}}\frac{\varphi_{1,j}}{\sqrt{\varphi_{j}}}.\label{1eq:-69}
\end{align}

From the inequalities that define $\psi$, we find that
\begin{align}
\bigg\lvert\frac{d}{dt}\int\frac{1}{2}p^{2}(t,x)\sqrt{\varphi_{j}(t,x)}\,dx\bigg\rvert\leq \frac{C}{\delta^{2}t}\int_{\Omega_{j}(t)\cup\Omega_{j+1}(t)}\big(p_{x}^{2}+p^{2}+p^{4}\big).\label{1eq:-70}
\end{align}

From now on, we can follow the proof of Lemma \ref{1lem:lcq}. 
\end{proof}
Now, we observe the following: 
\begin{align}
\int\big(\widetilde{P}+\varepsilon\big)^{2}\sqrt{\varphi_{k}^{b}}=\int\widetilde{B_{k}}^{2}+2\int\widetilde{B_{k}}\varepsilon\sqrt{\varphi_{k}^{b}}+\int\varepsilon^{2}\sqrt{\varphi_{k}^{b}}+\Err,\label{1eq:final}
\end{align}
where $\Err$ stands for the other terms of the sum, which we consider
as error terms, and we will show that they are bounded by
$Ce^{-\theta t}$.

For $i\neq j$ and any $h$ (if $P_{j}=B_{k}$ is a breather), 
\begin{align}
\bigg\lvert\int\widetilde{P_{i}}\widetilde{P_{h}}\sqrt{\varphi_{j}}\bigg\rvert & \leq  C\int_{-\delta t+\sigma_{j}t}^{\delta t+\sigma_{j+1}t}e^{-\frac{\beta}{2}\lvert x-v_{i}t\rvert}\,dx \leq  Ce^{-\theta t},\label{1eq:-71}
\end{align}
and 
\begin{align}
\bigg\lvert\int\widetilde{P_{i}}\varepsilon\sqrt{\varphi_{j}}\bigg\rvert & \leq \sqrt{\bigg(\int\widetilde{P_{i}}^{2}\varphi_{j}\bigg)\bigg(\int\varepsilon^{2}\bigg)} \\
 & \leq  Ce^{-\frac{\theta}{2}t}\lVert\varepsilon\rVert_{H^{2}} \leq  CAe^{-\theta t}e^{-\frac{\theta}{2}t}\leq  Ce^{-\theta t},\label{1eq:-72}
\end{align}
where $T^{*}\geq  T_{7}^{*}$ with $T_{7}^{*}$ being large enough depending
on $A$.
If we use the calculations we have made in the proof of Claim \ref{1lem:simpl1}, we see that
\begin{align}
\bigg\lvert\int\widetilde{P_{j}}^{2}-\int\widetilde{P_{j}}^{2}\sqrt{\varphi_{j}}\bigg\rvert\leq  Ce^{-\theta t}.\label{1eq:-73}
\end{align}
This proves the bound for the error terms. 

Now, we study the variations
of (\ref{1eq:final}).
We know that $\int\widetilde{P_{j}}^{2}=\int\widetilde{B_{k}}^{2}$
has no variations.
We can apply Lemma \ref{1lem:lcqsqrt} for $\int(\widetilde{P}+\varepsilon)^{2}\sqrt{\varphi_{j}}$.
By writing the difference of the equation (\ref{1eq:final}) between
$t$ and $T$, and using that $\varepsilon(T)=0$, we deduce, for
$T^{*}\geq \max(T_{6}^{*},T_{7}^{*})$, that
\begin{align}
\bigg\lvert\int\widetilde{P_{j}}\varepsilon\sqrt{\varphi_{j}}(t)\bigg\rvert & \leq  C\bigg(\frac{A^{2}}{\delta^{2}t}+1\bigg)e^{-\theta t}+\lVert\varepsilon\rVert_{H^{2}}^{2} \\
 & \leq  C\bigg(\frac{A^{2}}{\delta^{2}t}+1\bigg)e^{-\theta t}+\frac{\mu}{100}\lVert\varepsilon(t)\rVert_{H^{2}}.\label{1eq:-74}
\end{align}

Thus, 
\begin{align}
\mu\lVert\varepsilon\rVert_{H^{2}}^{2} & \leq  C\bigg(\frac{A^{2}}{\delta^{2}t}+1\bigg)e^{-2\theta t}+\frac{1}{\mu}\sum_{j=1}^{J}\bigg(\int\varepsilon\widetilde{P_{j}}\sqrt{\varphi_{j}}\bigg)^{2} \\
 & \leq  C\bigg(\frac{A^{4}}{\delta^{4}t}+1\bigg)e^{-2\theta t}+\frac{\mu}{100}\lVert\varepsilon(t)\rVert_{H^{2}}^{2}.\label{1eq:-75}
\end{align}
Therefore, 
\begin{align}
\lVert\varepsilon(t)\rVert_{H^{2}}^{2}\leq  C\bigg(\frac{A^{4}}{\delta^{4}t}+1\bigg)e^{-2\theta t}.\label{1eq:err}
\end{align}

By using (\ref{1eq:err}), the mean-value theorem and Lemma \ref{1lem:mod}, we deduce that
for $t\in[t^{*},T]$, 
\begin{align}
\MoveEqLeft\big\lVert p(t)-P(t)\big\rVert_{H^{2}}  \leq \big\lVert\varepsilon(t)\big\rVert_{H^{2}}+\big\lVert\widetilde{P}(t)-P(t)\big\rVert_{H^{2}} \\
 & \leq  C\Bigg(\sqrt{\frac{A^{4}}{\delta^{4}t}+1}\Bigg)e^{-\theta t}\\
 &\quad+C\Bigg(\sum_{k=1}^{K}\big(\lvert x_{1,k}(t)\rvert+\lvert x_{2,k}(t)\rvert\big)+\sum_{l=1}^{L}\big(\lvert x_{0,l}(t)\rvert+\lvert c_{0,l}(t)\rvert\big)\Bigg) \\
 & \leq  C\Bigg(\sqrt{\frac{A^{4}}{\delta^{4}t}+1}\Bigg)e^{-\theta t}+C\sum_{k=1}^{K}\bigg(\bigg\lvert\int_{t}^{T}x_{1,k}'(s)\,ds\bigg\rvert +\bigg\lvert\int_{t}^{T}x_{2,k}'(s)\,ds\bigg\rvert\bigg) \\
 & +C\sum_{l=1}^{L}\bigg(\bigg\lvert\int_{t}^{T}x_{0,l}'(s)\,ds\bigg\rvert +\bigg\lvert\int_{t}^{T}c_{0,l}'(s)\,ds\bigg\rvert\bigg) \\
 & \leq  C\bigg(\frac{A^{4}}{\delta^{4}t}+1\bigg)e^{-\theta t}+C\bigg(\int_{t}^{T}\big\lVert\varepsilon(s)\big\rVert_{H^{2}}\,ds+\int_{t}^{T}e^{-\theta s}\,ds\bigg) \\
 & \leq  C\bigg(\frac{A^{4}}{\delta^{4}t}+1\bigg)e^{-\theta t}.\label{1eq:exisccl}
\end{align}
We take $A=4C$ (where $C$ is a constant that can be used anywhere
in the proof above) and 
\begin{align}
T^{*}:=\max\big(T_{1}^{*},T_{2}^{*},T_{3}^{*},T_{4}^{*},T_{5}^{*},T_{6}^{*},T_{7}^{*},T_{8}^{*}\big)
\end{align}
(depending on $A$), where $T_{8}^{*}:=T_{8}^{*}(A)$ is such that
for $t\geq  T_{8}^{*}$, $\frac{A^{4}}{\delta^{4}t}\leq 1$. And thus,
for any $t\in[t^{*},T]$, 
\begin{align}
C\bigg(\frac{A^{4}}{\delta^{4}t}+1\bigg)\leq 2C=\frac{A}{2},\label{1eq:-24}
\end{align}
which is exactly what we wanted to prove.

\section{$p$ is a smooth multi-breather}
\label{1sec:smooth}

Our goal here is to prove Proposition \ref{1prop:higher}.

\subsection{Estimates in higher order Sobolev norms}

Firstly, we notice that the proposition is already established for $s=2$.
We note also that if this proposition is proved for an $s\geq 2$ with
a corresponding constant $A_{s}$, then this proposition is also valid
for any $s'\leq  s$ with the same constant $A_{s}$. This means that
$A_{s}$ can possibly increase with $s$ and that this proposition is
already established for $0\leq  s\leq 2$. From now on, we will denote
(as before) $p_{n}$ by $p$, $T_{n}$ by $T$ and $p_{n}-P$ by $v$,
and make sure that the constant $A_{s}$ that we will obtain in the
proof does not depend on $n$ (although it will depend on $s$). For
the constant $\theta$, we will take the usual value: $\theta:=\frac{\beta\tau}{32}$.
For the constant $T^{*}$, we will also take the value
that works for Proposition \ref{1prop:uest}.

We will prove the proposition by induction on $s$ (it is sufficient
to prove the proposition for any integer $s$). Let $s\geq 3$. We will prove the proposition for $s$, assuming that 
the proposition is true for any $0\leq  s'\leq  s-1$.

Let us deduce from the \eqref{1mKdV} equation the equation satisfied by
$v$: 
\begin{align}
v_{t} & =p_{t}-\sum_{j=1}^{J}P_{jt} \\
 & =-\Bigg(p_{xx}+p^{3}-\sum_{j=1}^{J}P_{jxx}-\sum_{j=1}^{J}P_{j}^{3}\Bigg)_{x} \\
 & =-\Bigg(v_{xx}+(v+P)^{3}-\sum_{j=1}^{J}P_{j}^{3}\Bigg)_{x} \\
 & =-\Bigg(v_{xx}+v^{3}+3v^{2}P+3vP^{2}+P^{3}-\sum_{j=1}^{J}P_{j}^{3}\Bigg)_{x}.\label{1eq:-80}
\end{align}

Firstly, we compute $\frac{d}{dt}\int(\partial_{x}^{s}v)^{2}$ by
integration by parts: 
\begin{align}
\frac{d}{dt}\int\big(\partial_{x}^{s}v\big)^{2} & =2\int\big(\partial_{x}^{s}v_{t}\big)\big(\partial_{x}^{s}v\big) \\
 & =-2\int\partial_{x}^{s+1}\bigg(v_{xx}+v^{3}+3v^{2}P+3vP^{2}+P^{3}-\sum_{j=1}^{J}P_{j}^{3}\bigg)\big(\partial_{x}^{s}v\big) \\
 & =2(-1)^{s+1}\int\partial_{x}^{2s+1}\bigg(P^{3}-\sum_{j=1}^{J}P_{j}^{3}\bigg)v-2\int\partial_{x}^{s+1}\big(v^{3}\big)\big(\partial_{x}^{s}v\big) \\
 & -6\int\partial_{x}^{s+1}\big(v^{2}P\big)\big(\partial_{x}^{s}v\big)-6\int\partial_{x}^{s+1}\big(vP^{2}\big)\big(\partial_{x}^{s}v\big),\label{1eq:-81}
\end{align}
because $\int(\partial_{x}^{s+3}v)(\partial_{x}^{s}v)=-\int(\partial_{x}^{s+2}v)(\partial_{x}^{s+1}v)=0$.

We will now bound above each of the terms of the obtained sum. By Sobolev embedding, Proposition \ref{1prop:cross-product} and Proposition \ref{1prop:uest},
\begin{align}
\Bigg\lvert\int\partial_{x}^{2s+1}\Bigg(P^{3}-\sum_{j=1}^{J}P_{j}^{3}\Bigg)v\Bigg\rvert & \leq \lVert v\rVert_{L^{\infty}}\int\bigg\lvert\partial_{x}^{2s+1}\bigg(P^{3}-\sum_{j=1}^{J}P_{j}^{3}\bigg)\bigg\rvert \\
 & \leq  C\lVert v\rVert_{H^{1}}e^{-\beta\tau t/2} \\
 & \leq  CAe^{-\theta t}e^{-\beta\tau t/2} \\
 & \leq  CAe^{-2\theta t}\leq  CA_{s-1}^{2}e^{-2\theta t},\label{1eq:-83}
\end{align}
where $C\geq 0$ is a constant that depends only on $s$.

We observe that 
\begin{align}
\partial_{x}^{s+1}\big(v^{3}\big) & =3\big(\partial_{x}^{s+1}v\big)v^{2}+6(s+1)\big(\partial_{x}^{s}v\big)v_{x}v+Z_{1}\big(v,v_{x},...,\partial_{x}^{s-1}v\big), \\
\partial_{x}^{s+1}\big(v^{2}P\big) & =2\big(\partial_{x}^{s+1}v\big)vP+2(s+1)\big(\partial_{x}^{s}v\big)(vP)_{x}\\
&\quad +Z_{2}\big(v,v_{x},...,\partial_{x}^{s-1}v,P,P_{x},...,\partial_{x}^{s+1}P\big),\label{1eq:-82}
\end{align}
where $Z_{1}$ and $Z_{2}$ are homogeneous polynomials of degree 3 with
constant coefficients.

Now, let us look for a bound for $\int\partial_{x}^{s+1}(v^{3})(\partial_{x}^{s}v)$.
Firstly, by integration by parts, 
\begin{align}
\int\partial_{x}^{s+1}\big(v^{3}\big)\big(\partial_{x}^{s}v\big) & =\frac{3}{2}\int\Big(\big(\partial_{x}^{s}v\big)^{2}\Big)_{x}v^{2}+3(s+1)\int\big(\partial_{x}^{s}v\big)^{2}\big(v^{2}\big)_{x}+\int\big(\partial_{x}^{s}v\big)Z_{1} \\
 & =\frac{6(s+1)-3}{2}\int\big(\partial_{x}^{s}v\big)^{2}\big(v^{2}\big)_{x}+\int\big(\partial_{x}^{s}v\big)Z_{1}.\label{1eq:-84}
\end{align}
Then, we bound above each of the terms of the obtained sum: 
\begin{align}
\bigg\lvert\int(\partial_{x}^{s}v)^{2}(v^{2})_{x}\bigg\rvert & \leq  C\big\lVert v\big\rVert_{L^{\infty}}\big\lVert v_{x}\big\rVert_{L^{\infty}}\int\big(\partial_{x}^{s}v\big)^{2} \\
 & \leq  C\big\lVert v\big\rVert_{H^{2}}^{2}\int\big(\partial_{x}^{s}v\big)^{2} \\
 & \leq  C\big(\lVert p\rVert_{H^{2}}+\lVert P\rVert_{H^{2}}\big)Ae^{-\theta t}\int\big(\partial_{x}^{s}v\big)^{2} \\
 & \leq  CC_{0}Ae^{-\theta t}\int\big(\partial_{x}^{s}v\big)^{2}\leq  CA_{s-1}e^{-\theta t}\int\big(\partial_{x}^{s}v\big)^{2}.\label{1eq:-85}
\end{align}
We have actually shown in the computation above that $\lVert v\rVert_{H^{2}}^{2}$
can be bounded above by $\lVert v\rVert_{H^{2}}$ (with a constant that depends only on problem data), and therefore the degree of
$\lVert v\rVert_{H^{2}}$ can be lowered without harm in the upper bound.
We will use this fact again for the rest of the proof. In fact,
all what it means is that, for several terms, what we have is more than
what we need.

By the Cauchy-Schwarz and Gagliardo-Nirenberg-Sobolev inequalities,
\begin{align}
\bigg\lvert\int\big(\partial_{x}^{s}v\big)Z_{1}\bigg\rvert & \leq  C\int\big\lvert\partial_{x}^{s}v\big\rvert\bigg(\sum_{s'=0}^{s-1}\big\lvert\partial_{x}^{s'}v\big\rvert^{3}\bigg) \\
 & \leq  C\bigg(\int\big\lvert\partial_{x}^{s}v\big\rvert^{2}\bigg)^{1/2}\sum_{s'=0}^{s-1}\bigg(\int\big\lvert\partial_{x}^{s'}v\big\rvert^{6}\bigg)^{1/2} \\
 & \leq  C\bigg(\int\big\lvert\partial_{x}^{s}v\big\rvert^{2}\bigg)^{1/2}\sum_{s'=0}^{s-1}\bigg(\int\big\lvert\partial_{x}^{s'}v\big\rvert^{2}\bigg)\bigg(\int\big\lvert\partial_{x}^{s'+1}v\big\rvert^{2}\bigg)^{1/2} \\
 & \leq  C\sum_{s'=0}^{s-1}\bigg(\int\big\lvert\partial_{x}^{s'}v\big\rvert^{2}\bigg)\bigg(\int\big\lvert\partial_{x}^{s}v\big\rvert^{2}+\int\big\lvert\partial_{x}^{s'+1}v\big\rvert^{2}\bigg) \\
 & \leq  CA_{s-1}^{2}e^{-2\theta t}+CA_{s-1}e^{-\theta t}\int\big\lvert\partial_{x}^{s}v\big\rvert^{2}.\label{1eq:-86}
\end{align}
Similarly, we bound $\int\partial_{x}^{s+1}(v^{2}P)(\partial_{x}^{s}v)$.
By integration by parts, 
\begin{align}
\int\partial_{x}^{s+1}\big(v^{2}P\big)\big(\partial_{x}^{s}v\big) & =\int\Big(\big(\partial_{x}^{s}v\big)^{2}\Big)_{x}vP+2(s+1)\int\big(\partial_{x}^{s}v\big)^{2}\big(vP\big)_{x}+\int\big(\partial_{x}^{s}v\big)Z_{2} \\
 & =(2s+1)\int\big(\partial_{x}^{s}v\big)^{2}\big(vP\big)_{x}+\int\big(\partial_{x}^{s}v\big)Z_{2}.\label{1eq:-87}
\end{align}

We bound above each of the terms of the obtained sum, starting by
\begin{align}
\bigg\lvert\int\big(\partial_{x}^{s}v\big)^{2}\big(vP\big)_{x}\bigg\rvert & \leq  C(\lVert v\rVert_{L^{\infty}}+\lVert v_{x}\rVert_{L^{\infty}})\int\big(\partial_{x}^{s}v\big)^{2} \\
 & \leq  CAe^{-\theta t}\int\big(\partial_{x}^{s}v\big)^{2}.\label{1eq:-88}
\end{align}
The upper bound of $\big\lvert\int(\partial_{x}^{s}v)Z_{2}\big\rvert$ is
similar to (\ref{1eq:-86}) above: 
\begin{align}
\bigg\lvert\int\big(\partial_{x}^{s}v\big)Z_{2}\bigg\rvert\leq  CA_{s-1}^{2}e^{-2\theta t}+CA_{s-1}e^{-\theta t}\int\big\lvert\partial_{x}^{s}v\big\rvert^{2}.\label{1eq:-89}
\end{align}

$\int\partial_{x}^{s+1}(vP^{2})(\partial_{x}^{s}v)$ remains to be bounded.
By integration by parts, 
\begin{align}
\MoveEqLeft\int\partial_{x}^{s+1}\big(vP^{2}\big)\big(\partial_{x}^{s}v\big)  =-\int\partial_{x}^{s+2}\big(vP^{2}\big)\big(\partial_{x}^{s-1}v\big) \\
 & =-\int\big(\partial_{x}^{s+2}v\big)\big(\partial_{x}^{s-1}v\big)P^{2}-(s+2)\int\big(\partial_{x}^{s+1}v\big)\big(\partial_{x}^{s-1}v\big)\big(P^{2}\big)_{x} \\
 & -\frac{(s+2)(s+1)}{2}\int\big(\partial_{x}^{s}v\big)\big(\partial_{x}^{s-1}v\big)\big(P^{2}\big)_{xx}+\int\big(\partial_{x}^{s-1}v\big)Z_{3}^{0}(v,v_{x},...,\partial_{x}^{s-1}v) \\
 & =\frac{1}{2}\int\Big(\big(\partial_{x}^{s}v\big)^{2}\Big)_{x}P^{2}+(s+1)\int\big(\partial_{x}^{s}v\big)^{2}\big(P^{2}\big)_{x} \\
 & -\frac{s(s+1)}{4}\int\Big(\big(\partial_{x}^{s-1}v\big)^{2}\Big)_{x}\big(P^{2}\big)_{xx}+\int\big(\partial_{x}^{s-1}v\big)Z_{3}^{0}(v,v_{x},...,\partial_{x}^{s-1}v) \\
 & =\frac{2s+1}{2}\int\big(\partial_{x}^{s}v\big)^{2}\big(P^{2}\big)_{x}+\int\big(\partial_{x}^{s-1}v\big)Z_{3}(v,v_{x},...,\partial_{x}^{s-1}v),\label{1eq:-90}
\end{align}
where $Z_{3}^{0}$ and $Z_{3}$ are homogeneous polynomials of degree
1 whose coefficients are polynomials in $P$ and its space derivatives.
We have that $\lvert Z_{3}\rvert\leq  C(\sum_{s'=0}^{s-1}\lvert\partial_{x}^{s'}v\rvert)$.
Therefore, 
\begin{align}
\bigg\lvert\int\big(\partial_{x}^{s-1}v\big)Z_{3}\bigg\rvert\leq  CA_{s-1}^{2}e^{-2\theta t}.\label{1eq:-91}
\end{align}

Thus, by taking the sum of all those inequalities, we obtain: 
\begin{align}
\bigg\lvert\frac{d}{dt}\int\big(\partial_{x}^{s}v\big)^{2}+3(2s+1)\int\big(\partial_{x}^{s}v\big)^{2}\big(P^{2}\big)_{x}\bigg\rvert\leq  CA_{s-1}^{2}e^{-2\theta t}+CA_{s-1}e^{-\theta t}\int\big\lvert\partial_{x}^{s}v\big\rvert^{2}.\label{1eq:-92}
\end{align}

Next, we perform similar computations for $\frac{d}{dt}\int(\partial_{x}^{s-1}v)^{2}P^{2}$:
\begin{align}
\frac{d}{dt}\int(\partial_{x}^{s-1}v)^{2}P^{2} & =2\int\big(\partial_{x}^{s-1}v_{t}\big)\big(\partial_{x}^{s-1}v\big)P^{2}+2\int\big(\partial_{x}^{s-1}v\big)^{2}P_{t}P \\
 & =-2\int\partial_{x}^{s}\bigg(v_{xx}+v^{3}+3v^{2}P+3vP^{2}+P^{3}-\sum_{j=1}^{J}P_{j}^{3}\bigg)\big(\partial_{x}^{s-1}v\big)P^{2} \\
 & -2\int\big(\partial_{x}^{s-1}v\big)^{2}\bigg(P_{xx}+\sum_{j=1}^{J}P_{j}^{3}\bigg)_{x}P.\label{1eq:-93}
\end{align}
Let us study each of the obtained terms.

Firstly, 
\begin{align}
-2\int\big(\partial_{x}^{s+2}v\big)\big(\partial_{x}^{s-1}v\big)P^{2} & =2\int\big(\partial_{x}^{s+1}v\big)\big(\partial_{x}^{s}v\big)P^{2}+2\int\big(\partial_{x}^{s+1}v\big)(\partial_{x}^{s-1}v\big)\big(P^{2}\big)_{x} \\
 & =-3\int\big(\partial_{x}^{s}v\big)^{2}\big(P^{2}\big)_{x}-2\int\big(\partial_{x}^{s}v\big)\big(\partial_{x}^{s-1}v\big)\big(P^{2}\big)_{xx} \\
 & =-3\int\big(\partial_{x}^{s}v\big)^{2}\big(P^{2}\big)_{x}+\int\big(\partial_{x}^{s-1}v\big)^{2}\big(P^{2}\big)_{xxx}.\label{1eq:-94}
\end{align}
Indeed, 
\begin{align}
\bigg\lvert\int\big(\partial_{x}^{s-1}v\big)^{2}\big(P^{2}\big)_{xxx}\bigg\rvert\leq  CA^{2}e^{-2\theta t}.\label{1eq:-95}
\end{align}

Secondly, 
\begin{align}
\bigg\lvert\int\partial_{x}^{s}\bigg(P^{3}-\sum_{j=1}^{J}P_{j}^{3}\bigg)\big(\partial_{x}^{s-1}v\big)P^{2}\bigg\rvert\leq  CA_{s-1}^{2}e^{-2\theta t}\label{1eq:-96}
\end{align}
can be obtained similarly to the first part of the proof (starting
by an integration by parts to have $\partial_{x}^{s-2}v$ at the place
of $\partial_{x}^{s-1}v$).

Thirdly, 
\begin{align}
\int\partial_{x}^{s}\big(v^{3}\big)\big(\partial_{x}^{s-1}v\big)P^{2} & =3\int\big(\partial_{x}^{s}v\big)\big(\partial_{x}^{s-1}v\big)v^{2}P^{2}+\int Z_{4}(v,v_{x},...,\partial_{x}^{s-1}v)P^{2} \\
 & =-\frac{3}{2}\int\big(\partial_{x}^{s-1}v\big)^{2}\big(v^{2}P^{2}\big)_{x}+\int Z_{4}P^{2},\label{1eq:-97}
\end{align}
where $Z_{4}$ is a homogeneous polynomial of degree 4 with constant
coefficients. Both terms are easily bounded by $CA_{s-1}^{2}e^{-2\theta t}$.

Fourthly, for $\int\partial_{x}^{s}(v^{2}P)(\partial_{x}^{s-1}v)P^{2}$
and $\int\partial_{x}^{s}(vP^{2})(\partial_{x}^{s-1}v)P^{2}$,
we reason similarly.

Fifthly, 
\begin{align}
\bigg\lvert\int\big(\partial_{x}^{s-1}v\big)^{2}\bigg(P_{xx}+\sum_{j=1}^{J}P_{j}^{3}\bigg)_{x}P\bigg\rvert\leq  CA_{s-1}^{2}e^{-2\theta t}\label{1eq:-98}
\end{align}
is clear.

Therefore, 
\begin{align}
\bigg\lvert\frac{d}{dt}\int\big(\partial_{x}^{s-1}v\big)^{2}P^{2}+3\int\big(\partial_{x}^{s}v\big)^{2}\big(P^{2}\big)_{x}\bigg\rvert\leq  CA_{s-1}^{2}e^{-2\theta t}.\label{1eq:-99}
\end{align}

We set 
\begin{align}
F(t):=\int\big(\partial_{x}^{s}v\big)^{2}-(2s+1)\int\big(\partial_{x}^{s-1}v\big)^{2}P^{2}.\label{1eq:-100}
\end{align}
By putting the both parts of the proof together: 
\begin{align}
\bigg\lvert\frac{d}{dt}F(t)\bigg\rvert\leq  CA_{s-1}^{2}e^{-2\theta t}+CA_{s-1}e^{-\theta t}\int\big\lvert\partial_{x}^{s}v\big\rvert^{2}.\label{1eq:-101}
\end{align}
Because $\big\lvert\int(\partial_{x}^{s-1}v)^{2}P^{2}\big\rvert\leq  CA^{2}e^{-2\theta t}$,
we can write the following upper bound: 
\begin{align}
\int\big(\partial_{x}^{s}v\big)^{2}\leq \lvert F(t)\rvert+CA_{s-1}^{2}e^{-2\theta t}.\label{1eq:-102}
\end{align}
Therefore, we have, for a suitable constant $C>0$ that depends only on $s$,
\begin{align}
\bigg\lvert\frac{d}{dt}F(t)\bigg\rvert\leq  CA_{s-1}^{2}e^{-2\theta t}+CA_{s-1}e^{-\theta t}\lvert F(t)\rvert.\label{1eq:-103}
\end{align}

For $t\in[T^{*},T]$, by integration between $t$ and $T$ (we recall
that $F(T)=0$), 
\begin{align}
\lvert F(t)\rvert & =\lvert F(T)-F(t)\rvert=\bigg\lvert\int_{t}^{T}\frac{d}{dt}F(\sigma)\,d\sigma\bigg\rvert \leq \int_{t}^{T}\bigg\lvert\frac{d}{dt}F(\sigma)\bigg\rvert \,d\sigma \\
 & \leq  CA_{s-1}^{2}\int_{t}^{T}e^{-2\theta\sigma}\,d\sigma+CA_{s-1}\int_{t}^{T}e^{-\theta\sigma}\lvert F(\sigma)\rvert \,d\sigma \\
 & \leq  CA_{s-1}^{2}e^{-2\theta t}+CA_{s-1}\int_{t}^{T}e^{-\theta\sigma}\lvert F(\sigma)\rvert \,d\sigma.\label{1eq:-104}
\end{align}

By Gronwall lemma, for all $t\in[T^{*},T]$, 
\begin{align}
\MoveEqLeft\lvert F(t)\rvert  \leq  CA_{s-1}^{2}e^{-2\theta t}\\
&\quad+CA_{s-1}\int_{t}^{T}e^{-\theta\sigma}CA_{s-1}^{2}e^{-2\theta\sigma}\exp\bigg(\int_{t}^{\sigma}CA_{s-1}e^{-\theta u}du\bigg)\,d\sigma \\
 & \leq  CA_{s-1}^{2}e^{-2\theta t}\\
 &\quad+CA_{s-1}^{3}\exp\bigg(\frac{CA_{s-1}}{\theta}e^{-\theta t}\bigg)\int_{t}^{T}e^{-3\theta\sigma}\exp\bigg(-\frac{CA_{s-1}}{\theta}e^{-\theta\sigma}\bigg)\,d\sigma \\
 & \leq  CA_{s-1}^{2}e^{-2\theta t}+CA_{s-1}^{3}\exp\bigg(\frac{CA_{s-1}}{\theta}\bigg)\int_{t}^{T}e^{-3\theta\sigma}\,d\sigma \\
 & \leq  CA_{s-1}^{2}e^{-2\theta t}+CA_{s-1}^{3}\exp\bigg(\frac{CA_{s-1}}{\theta}\bigg)e^{-3\theta t} \\
 & \leq  CA_{s-1}^{3}\exp\bigg(\frac{CA_{s-1}}{\theta}\bigg)e^{-2\theta t}.\label{1eq:-105}
\end{align}

Therefore, 
\begin{align}
\int\big(\partial_{x}^{s}v\big)^{2}\leq  A_{s}e^{-2\theta t},\label{1eq:-106}
\end{align}
where $A_{s}:=CA_{s-1}^{3}\exp\big(\frac{CA_{s-1}}{\theta}\big)$
and $C$ is a constant large enough that depends only on $s$. 
This conclude the proof of Proposition \ref{1prop:higher}, and so of Theorem \ref{1thm:MAIN}.%

\subsection{Uniformity of constants} \label{1sec:2.5}

We conclude this section with an explanation regarding Remark \ref{1rm:3}.

In the proof above, the constants that we obtain $A,T^{*},\theta$
do depend on $P_{j}(0)$ ($1\leq  j\leq  J$). Actually, we may characterize
this dependence. In fact, they do not depend on the initial positions
of our objects in the case when our objects are initially ordered in
the right order and sufficiently far from each other.
\begin{thm}
Given parameters (\ref{1eq:br}), (\ref{1eq:sol}), (\ref{1eq:vb}) and
(\ref{1eq:vs}) which satisfy (\ref{1eq:diff}), there exists $D>0$
large enough that depends only on $\alpha_{k},\beta_{k},c_{l}$ such
that if
\begin{align}
\forall j\geq 2,\quad x_{j}(0)\geq  x_{j-1}(0)+D,\label{1eq:eloin}
\end{align}
then the following holds. We set $\theta:=\frac{\beta\tau}{32}$,
with $\beta$ and $\tau$ given by (\ref{1eq:beta-tau}) and $p(t)$
the multi-breather associated to $P$ by Proposition \ref{1thm:main}.
There exists $A_{s}\geq 1$ for any $s\geq 2$ that depends only on $\alpha_{k},\beta_{k},c_{l}$
and $D$ such that
\begin{align}
\forall t\geq 0,\quad\lVert p(t)-P(t)\rVert_{H^{s}}\leq  A_{s}e^{-\theta t}.\label{1eq:-107}
\end{align}
\end{thm}

Firstly, we will prove that for any $D>0$, if (\ref{1eq:eloin}) is satisfied,
then the constants $A_{s}$ and $T^{*}$ do only depend on $\alpha_{k},\beta_{k},c_{l}$
and $D$. Finally, we will prove that if $D>0$ is large enough
with respect to the given parameters, then we can take $T^{*}=0$.

To establish the validity of this theorem, it is enough to read again
the whole article and to make sure that on any step of the proof,
there is no dependence on initial positions of our objects when our
objects are initially far from each other for the constant $C$. This will allow to claim the same
for the constants $A$ and $T^{*}$ (but, these constants may depend on $D$).
This works, but we should change a bit the way we write our
results.

For Proposition \ref{1prop:decay}, we should write:
\begin{align}
\lvert\partial_{x}^{n}\partial_{t}^{m}P_{j}(t,x)\rvert\leq  Ce^{-\beta\lvert x-v_{j}t-x_{j}(0)\rvert}.\label{1eq:-108}
\end{align}
Therefore, in Proposition \ref{1prop:cross-product}, we have nothing to change, but the constant
$C$ do depend on $D$. This will also be the case in the following
propositions and lemmas of this proof.

We should replace $\sigma_{j}t$ for the definition of $\varphi_{j}$ in \eqref{1eq:phij} and \eqref{1eq:phij-1}
by $\sigma_{j}t+\frac{x_{j-1}(0)+x_{j}(0)}{2}$ to take into account of
initial positions. More precisely, we will have for any $j=2,...,J-1$,
\begin{align}
\varphi_{j}(t,x):=\psi\Bigg(\frac{x-\sigma_{j}t-\frac{x_{j-1}(0)+x_{j}(0)}{2}}{\delta t}\Bigg)-\psi\Bigg(\frac{x-\sigma_{j+1}t-\frac{x_{j}(0)+x_{j+1}(0)}{2}}{\delta t}\Bigg),\label{1eq:-109}
\end{align}
and similarly for other definitions.

After having done the modulation with $C$ and $T^{*}$ depending
on $D$, for Proposition \ref{1prop:decay-mod}, we should write:
\begin{align}
\lvert\partial_{x}^{n}\widetilde{P_{j}}(t,x)\rvert\leq  Ce^{-\frac{\beta}{2}\lvert x-v_{j}t-x_{j}(0)\rvert}e^{\frac{\beta\tau}{32}t}.\label{1eq:-110}
\end{align}

Therefore, with these adaptations, the same proof works to prove that
$A_{s}$ and $T^{*}$ do depend only on $\alpha_{k},\beta_{k},c_{l}$
and $D$.

Now, given $\alpha_{k},\beta_{k},c_{l}$, we choose $D_{0}>0$ in
an arbitrary maner. Therefore, we get $A_{s}(D_{0})$ and $T^{*}(D_{0})$
associated to $D_{0}$. Let $\Lambda:=v_{J}-v_{1}$ the maximal difference
between two velocities. We set $D:=D_{0}+\Lambda\cdot T^{*}(D_{0})$.
Therefore, if we suppose (\ref{1eq:eloin}) in $t=0$ for $D$, then we have
(\ref{1eq:eloin}) in $t=-T^{*}(D_{0})$ for $D_{0}$. Therefore, by appliying
the intermediate result for $D_{0}$, we obtain the desired conclusion
with $D$ and $A_{s}$ that depend on $D_0$.

\section{Uniqueness}
\label{1sec:uniq}

$p$ is the multi-breather constructed in the existence part. The goal here is to prove that if a solution $u$ converges to $p$ when $t \rightarrow +\infty$ (in some sense), then $u=p$ (under well chosen assumptions).

We prove here two propositions. For both of them, we assume that the velocities of all our objects are distinct (this was also an assumption for the existence). The first proposition does not make more assumptions on velocities of our objects, but it is a partial uniqueness result as we restrict ourselves to the class of super polynomial convergence to the multi-breather. The second proposition assumes that the velocities of all our objects are positive (this is a new assumption and it is needed because this proof uses monotonicity arguments). 

\subsection{A solution converging super polynomialy to a multi-breather is this multi-breather}
\label{1sec:super-pol}
The goal of this subsection is to prove Proposition \ref{1lem:polyn}.
\begin{rem}
Note that in Proposition \ref{1lem:polyn}, we don't make any assumptions on the
sign of $v_{1}$ or $v_{2}$. This uniqueness proposition has the
same degree of generality as Theorem \ref{1thm:MAIN}.
\end{rem}

\begin{proof}[Proof of Proposition \ref{1lem:polyn}]
Let $p(t)$ be the multi-breather associated to $P$ by Theorem \ref{1thm:MAIN}. Recall that for any $s$,
\begin{align}
\lVert p(t)-P(t)\rVert _{H^{s}}=O(e^{-\theta t}),\label{1eq:00}
\end{align}
for a suitable $\theta>0$.

Let $N>2$ to be chosen later. We take $u(t)$ an $H^{2}$ solution of \eqref{1mKdV} such
that there exists $C_{0}>0$ such that for $t$ large enough,
\begin{align}
\lVert u(t)-P(t)\rVert _{H^{2}}\leq\frac{C_{0}}{t^{N}}.
\end{align}
From that, we may deduce that for $t$ large enough (namely, $t\geq2C_{0}$
along with the previous condition),
\begin{align}
\lVert u(t)-P(t)\rVert _{H^{2}}\leq\frac{1}{2}\frac{1}{t^{N-1}}.\label{1eq:01}
\end{align}

Our goal is to find a condition on $N$ that do not depend on $u$,
such that the condition (\ref{1eq:01}) on $u$ for $t$ large enough
implies that $u\equiv p$.

Because of (\ref{1eq:00}), the condition (\ref{1eq:01}) for $t$ large
enough is equivalent to: for $t$ large enough,
\begin{align}
\lVert u(t)-p(t)\rVert _{H^{2}}\leq\frac{1}{t^{N-1}}.
\end{align}
We denote $z(t):=u(t)-p(t)$. Our goal is to find $N$ large enough
that do not depend on $z$, for which we will be able to prove that
$z\equiv0$, given
\begin{align}
\lVert z(t)\rVert _{H^{2}}\leq\frac{1}{t^{N-1}},\label{1eq:hyp}
\end{align}
for $t$ large enough.
Because $z$ is a difference of two solutions of \eqref{1mKdV}, we may derive
the following equation for $z$:
\begin{align}
z_{t}+\big(z_{xx}+(z+p)^{3}-p^{3}\big)_{x}=0.\label{1eq:z}
\end{align}

We divide our proof in several steps.

\emph{Step 1.} Modulation on $z$.

For $j=1,...,J$, if $P_{j}=B_{k}$ is a breather, we denote 
\begin{align}
K_{j}:=
\begin{pmatrix}
\partial_{x_{1}}B_{k}\\
\partial_{x_{2}}B_{k}
\end{pmatrix}
,
\end{align}
and if $P_{j}=R_{l}$ is a soliton, we denote:
\begin{align}
K_{j}=\partial_{x}R_{l}.
\end{align}
We may derive the following equation for $K_{j}$:
\begin{align}
(K_{j})_{t}+\big((K_{j})_{xx}+3P_{j}^{2}K_{j}\big)_{x}=0.
\end{align}

For $j=1,...,J$, if $P_{j}=B_{k}$ is a breather, let $c_{j}(t)\in\mathbb{R}^{2}$
defined for $t$ large enough and if $P_{j}=R_{l}$ is a soliton, let $c_{j}(t)\in\mathbb{R}$
defined for $t$ large enough such that for 
\begin{align}
\widetilde{z}(t):=z(t)+\sum_{j=1}^{J}c_{j}(t)K_{j}(t),\label{1eq:def}
\end{align}
where $c_{j}K_{j}$ is either a product of two numbers of $\mathbb{R}$
or a scalar product of two vectors of $\mathbb{R}^{2}$, the following
condition is satisfied: for any $j=1,...,J$, for $t$ large enough,
\begin{align}
\int\widetilde{z}(t)K_{j}(t)\sqrt{\varphi_{j}(t)}=0,\label{1eq:02}
\end{align}
where $\varphi_{j}$ is defined in Section \ref{1sec:2.2} (in this proof,
it is OK to take $\delta=1$).
It is possible to do so in a unique way, because the Gram matrix associated
to $K_{j}(t)\sqrt{\varphi_{j}(t)}$, $1\leq j\leq J$, is invertible;
which is the case because $K_{j}(t)\sqrt{\varphi_{j}(t)}$, $1\leq j\leq J$,
are linearly independent. This is why $c_{j}(t)$, $1\leq j\leq J$,
are defined in a unique way. For the same reason, $c_{j}(t)$ is obtained
linearly from $\int K_{k}(t)z(t)\sqrt{\varphi_{k}(t)}$, $1\leq k\leq J$,
with coefficients that depend only on $K_{k}$, $1\leq k\leq J$.
This is why, from Cauchy-Schwarz, we may deduce the following lemma.
\begin{lem}
\label{1lem:3}For any $j=1,...,J$, for $t$ large enough, there exists
$C>0$ that do not depend on $z$, such that
\begin{align}
\lvert c_{j}(t)\rvert\leq C\lVert z(t)\rVert_{L^{2}},
\end{align}
\begin{align}
\lVert\widetilde{z}(t)\rVert_{H^{2}}\leq C\lVert z(t)\rVert_{H^{2}}.
\end{align}
\end{lem}
The Gram matrix is $C^{1}$ in time and invertible. This
is why, its inverse is $C^{1}$ in time. Because $\int K_{j}z\sqrt{\varphi_{j}}$
are $C^{1}$ in time, we deduce by multiplication that $c_{j}(t)$
are $C^{1}$ in time.

By differentiating in time the linear relation that defines $c_{j}(t)$,
we see that $c_{j}'(t)$ is obtained linearly from $\int K_{k}(t)z(t)\sqrt{\varphi_{k}(t)}$,
$1\leq k\leq J$,  and from $\frac{d}{dt}\int K_{k}(t)z(t)\sqrt{\varphi_{k}(t)}$,
$1\leq k\leq J$, with coefficients that depend on $K_{k}$, $1\leq k\leq J$ (and their derivatives).
Because it is easy to see that $\frac{d}{dt}\int K_{k}(t)z(t)\sqrt{\varphi_{k}(t)}$
may still be bounded by $C\lVert z(t)\rVert_{L^{2}}$, we deduce that
for any $j=1,...,J$, for $t$ large enough, there exists $C>0$ that
do not depend on $z$, such that
\begin{align}
\lvert c_{j}'(t)\rvert\leq C\lVert z(t)\rVert_{L^{2}}.\label{1eq:03}
\end{align}

We may derive the following equation for $\widetilde{z}$:
\begin{align}
\widetilde{z}_{t}+\big(\widetilde{z}_{xx}+3\widetilde{z}p^{2}\big)_{x}=-\big(3z^{2}p+z^{3}\big)_{x}+\sum_{k=1}^{J}c_{k}'(t)K_{k}-3\sum_{k=1}^{J}c_{k}(t)\big((P_{k}^{2}-p^{2})K_{k}\big)_{x}.\label{1eq:04}
\end{align}

\emph{Step 2.} A bound for $\lvert c_{j}'(t)\rvert$. 

The goal here
is to improve (\ref{1eq:03}).
\begin{lem}
\label{1lem:4}For any $j=1,...,J$, for $t$ large enough, there exists
$C>0$ and $\theta>0$ that do not depend on $z$, such that
\begin{align}
\lvert c_{j}'(t)\rvert\leq C\lVert\widetilde{z}(t)\rVert_{H^{2}}+Ce^{-\theta t}\lVert z(t)\rVert_{H^{2}}+C\lVert z(t)\rVert_{H^{2}}^{2}.
\end{align}
\end{lem}

\begin{proof}
We may differentiate (\ref{1eq:02}):
\begin{align}
0 & =\frac{d}{dt}\int\widetilde{z}K_{j}\sqrt{\varphi_{j}}\\
 & =\int\widetilde{z}_{t}K_{j}\sqrt{\varphi_{j}}+\int\widetilde{z}(K_{j})_{t}\sqrt{\varphi_{j}}+\int\widetilde{z}K_{j}(\sqrt{\varphi_{j}})_{t}\\
 & =-\int\big(\widetilde{z}_{xx}+3\widetilde{z}p^{2}\big)_{x}K_{j}\sqrt{\varphi_{j}}-\int\big(3z^{2}p+z^{3}\big)_{x}K_{j}\sqrt{\varphi_{j}}\\
 & +\sum_{k=1}^{J}\int\big(c_{k}'(t)\cdot K_{k}\big)K_{j}\sqrt{\varphi_{j}}-3\sum_{k=1}^{J}c_{k}(t)\int\Big(c_{k}(t)\cdot\big((P_{k}^{2}-p^{2})K_{k}\big)_{x}\Big)K_{j}\sqrt{\varphi_{j}}\\
 & -\int\widetilde{z}\big((K_{j})_{xx}+3K_{j}P_{j}^{2}\big)_{x}\sqrt{\varphi_{j}}+\int\widetilde{z}K_{j}(\sqrt{\varphi_{j}})_{t}.
\end{align}

We know that $(\sqrt{\varphi_{j}})_{x}$ and $(\sqrt{\varphi_{j}})_{t}$
are bounded (from inequalities established in Section \ref{1sec:2.2}). This is why, for any $t$ large enough,
\begin{align}
 \bigg\lvert\int\widetilde{z}K_{j}(\sqrt{\varphi_{j}})_{t}\bigg\rvert \leq C\lVert\widetilde{z}(t)\rVert_{H^{2}}. 
\end{align}
For the same
reason, after eventually doing an integration by parts, for any $t$ large enough,
\begin{align}
\bigg\lvert\int\big(\widetilde{z}_{xx}+3\widetilde{z}p^{2}\big)_{x}K_{j}\sqrt{\varphi_{j}}\bigg\rvert +  \bigg\lvert\int\widetilde{z}\big((K_{j})_{xx}+3K_{j}P_{j}^{2}\big)_{x}\sqrt{\varphi_{j}}\bigg\rvert \leq C\lVert\widetilde{z}(t)\rVert_{H^{2}}.
\end{align}
$\int(3z^{2}p+z^{3})_{x}K_{j}\sqrt{\varphi_{j}}$
is clearly bounded by $C\lVert z(t)\rVert_{H^{2}}^{2}$. Finally, we
see that $(P_{k}^{2}-p^{2})K_{k}$ is exponentially bounded in time
(in Sobolev or $L^{\infty}$ norm), and using Lemma \ref{1lem:3},
we deduce that 
\begin{align}
\int\Big(c_{k}(t)\cdot\big((P_{k}^{2}-p^{2})K_{k}\big)_{x}\Big)K_{j}\sqrt{\varphi_{j}}\leq Ce^{-\theta t}\lVert z(t)\rVert_{H^{2}},
\end{align}
for a suitable
$\theta>0$ that do not depend on $z$.
This is why, we deduce that for any $j=1,...,J$, for $t$ large enough,
there exists $C>0$ and $\theta>0$ that do not depend on $z$, such
that
\begin{align}
\Bigg\lvert\sum_{k=1}^{J}\int\big(c_{k}'(t)\cdot K_{k}\big)K_{j}\sqrt{\varphi_{j}}\Bigg\rvert\leq C\lVert\widetilde{z}(t)\rVert_{H^{2}}+Ce^{-\theta t}\lVert z(t)\rVert_{H^{2}}+C\lVert z(t)\rVert_{H^{2}}^{2}.
\end{align}

We recall that for any $(e_{1},e_{2})\in(\mathbb{R})^{2}\:\text{or}\:(\mathbb{R}^{2})^{2},e_{3}\in\mathbb{R}\:\text{or}\:\mathbb{R}^{2}$,
we have the following equality between two elements of $\mathbb{R}$
or $\mathbb{R}^{2}$ (where vectors are denoted as a colon)
\begin{align}
\big(e_{1}\cdot e_{2}\big)e_{3}=\Big(e_{1}^{T}\big(e_{2}e_{3}^{T}\big)\Big)^{T},
\end{align}
where $^{T}$ denotes the transpose.

First of all, because $\int K_{k}K_{j}^{T}\sqrt{\varphi_{j}}$ converges
exponentially to $\int K_{k}K_{j}^{T}$, for $k\neq j$, $\int K_{k}K_{j}^{T}$
is exponentially decreasing, and from (\ref{1eq:03}), we may write
that for any $j=1,...,J$, for $t$ large enough, there exists $C>0$
and $\theta>0$ that do not depend on $z$, such that
\begin{align}
\Bigg\lvert\bigg(c_{j}'(t)^{T}\int K_{j}K_{j}^{T}\bigg)^{T}\Bigg\rvert\leq C\lVert\widetilde{z}(t)\rVert_{H^{2}}+Ce^{-\theta t}\lVert z(t)\rVert_{H^{2}}+C\lVert z(t)\rVert_{H^{2}}^{2}.
\end{align}
Now, in the case when $K_{j}\in\mathbb{R}^{2}$, using the fact that its components
are linearly independent and Cauchy-Schwarz inequality, we deduce
the desired lemma.
\end{proof}
\emph{Step 3.} Coercivity.

We define the following functional quadratic in $\widetilde{z}$:
\begin{align}
H(t) & =\frac{1}{2}\int\widetilde{z}_{xx}^{2}-\frac{5}{2}\int p^{2}\widetilde{z}_{x}^{2}+\frac{5}{2}\int p_{x}^{2}\widetilde{z}^{2}+5\int pp_{xx}\widetilde{z}^{2}+\frac{15}{4}\int p^{4}\widetilde{z}^{2}\\
 & \quad+\sum_{j=1}^{J}\big(b_{j}^{2}-a_{j}^{2}\big)\bigg(\int\widetilde{z}_{x}^{2}\varphi_{j}-3\int p^{2}\widetilde{z}^{2}\varphi_{j}\bigg)+\sum_{j=1}^{J}\big(a_{j}^{2}+b_{j}^{2}\big)^{2}\frac{1}{2}\int\widetilde{z}^{2}\varphi_{j}.
\end{align}
We will prove the following lemma:

\begin{lem}
There exists $C>0$ that do not depend on $z$, such that for $t$
large enough,
\begin{align}
\lVert\widetilde{z}(t)\rVert_{H^{2}}^{2}\leq CH(t)+C\sum_{j=1}^{J}\bigg(\int\widetilde{z}P_{j}\bigg)^{2}.
\end{align}
\end{lem}

\begin{proof}
We denote $\mathcal{Q}_{j}$ the quadratic form associated to $P_{j}$.
We remind that
\begin{align}
\mathcal{Q}_{j}[\varepsilon] & :=\frac{1}{2}\int\varepsilon_{xx}^{2}-\frac{5}{2}\int P_{j}^{2}\varepsilon_{x}^{2}+\frac{5}{2}\int(P_{j})_{x}^{2}\varepsilon^{2}+5\int P_{j}(P_{j})_{xx}\varepsilon^{2}\\
 & \quad+\frac{15}{4}\int P_{j}^{4}\varepsilon^{2}+\big(b_{j}^{2}-a_{j}^{2}\big)\bigg(\int\varepsilon_{x}^{2}-3\int P_{j}^{2}\varepsilon^{2}\bigg)+\big(a_{j}^{2}+b_{j}^{2}\big)^{2}\frac{1}{2}\int\varepsilon^{2}.
\end{align}

In any case, we have that for any $j=1,...,J$, there exists $\mu_{j}>0$,
such that if $\varepsilon\in H^{2}$ satisfies $\int K_{j}\varepsilon=0$,
then we have
\begin{align}
\mathcal{Q}_{j}[\varepsilon]\geq\mu_{j}\lVert\varepsilon\rVert_{H^{2}}^{2}-\frac{1}{\mu_{j}}\bigg(\int\varepsilon P_{j}\bigg)^{2}.
\end{align}
Here, we apply this coercivity result with $\varepsilon=\widetilde{z}\sqrt{\varphi_{j}}$
for which the orthogonality conditions \eqref{1eq:02} are satisfied. Thus,
\begin{align}
\lVert\widetilde{z}\sqrt{\varphi_{j}}\rVert_{H^{2}}^{2}\leq C\mathcal{Q}_{j}[\widetilde{z}\sqrt{\varphi_{j}}]+C\bigg(\int\widetilde{z}P_{j}\sqrt{\varphi_{j}}\bigg)^{2}.
\end{align}

We denote:
\begin{align}
\mathcal{Q}_{j}'[\varepsilon] & :=\frac{1}{2}\int\varepsilon_{xx}^{2}\varphi_{j}-\frac{5}{2}\int p^{2}\varepsilon_{x}^{2}\varphi_{j}+\frac{5}{2}\int p_{x}^{2}\varepsilon^{2}\varphi_{j}+5\int pp_{xx}\varepsilon^{2}\varphi_{j}\\
 & \quad+\frac{15}{4}\int p^{4}\varepsilon^{2}\varphi_{j}+\big(b_{j}^{2}-a_{j}^{2}\big)\bigg(\int\varepsilon_{x}^{2}\varphi_{j}-3\int p^{2}\varepsilon^{2}\varphi_{j}\bigg)+\big(a_{j}^{2}+b_{j}^{2}\big)^{2}\frac{1}{2}\int\varepsilon^{2}\varphi_{j},
\end{align}
and we observe that 
\begin{align}
H(t)=\sum_{j=1}^{J}\mathcal{Q}_{j}'[\widetilde{z}(t)].
\end{align}
In $\mathcal{Q}_{j}'[\widetilde{z}(t)]$, we may replace $p$ by $P_{j}$
with an error bounded by $Ce^{-\theta t}\lVert\widetilde{z}(t)\rVert_{H^{2}}^{2}$,
because of (\ref{1eq:00}) mainly. After that, the expression obtained
may be replaced by $\mathcal{Q}_{j}[\widetilde{z}(t)\sqrt{\varphi_{j}(t)}]$
with an error bounded by $\frac{C}{t}\lVert\widetilde{z}(t)\rVert_{H^{2}}^{2}$
(cf. calculations done in the proof of Lemma \ref{1lem:coerc}). For the same reason,
$\lVert\widetilde{z}\sqrt{\varphi_{j}}\rVert_{H^{2}}^{2}$ may be replaced
by $\int(\widetilde{z}^{2}+\widetilde{z}_{x}^{2}+\widetilde{z}_{xx}^{2})\varphi_{j}$
with an error bounded by $\frac{C}{t}\lVert\widetilde{z}(t)\rVert_{H^{2}}^{2}$.
Therefore, because of
\begin{align}
\lVert \widetilde{z}\rVert _{H^{2}}^{2}=\sum_{j=1}^{J}\int\big(\widetilde{z}^{2}+\widetilde{z}_{x}^{2}+\widetilde{z}_{xx}^{2}\big)\varphi_{j},
\end{align}
the fact that $P_{j}\sqrt{\varphi_{j}}$ converges exponentially
to $P_{j}$, and the fact that $\frac{C}{t}$ may be as small as we
want if we take $t$ large enough, we deduce the desired lemma.
\end{proof}

\emph{Step 4.} Modification of $H$ for the sake of simplification.

We define:
\begin{align}
\widetilde{H}(t) & :=\int\bigg[\frac{1}{2}\widetilde{z}_{xx}^{2}-\frac{5}{2}\big((\widetilde{z}+p)^{2}(\widetilde{z}+p)_{x}^{2}-p^{2}p_{x}^{2}-2\widetilde{z}pp_{x}^{2}-2\widetilde{z}_{x}p^{2}p_{x}\big)\\
& \qquad+\frac{1}{4}\big((\widetilde{z}+p)^{6}-p^{6}-6\widetilde{z}p^{5}\big)\bigg]+\frac{1}{2}\sum_{j=1}^{J}\big(a_{j}^{2}+b_{j}^{2}\big)^{2}\int\widetilde{z}^{2}\varphi_{j}\\
 & \quad+2\sum_{j=1}^{J}\big(b_{j}^{2}-a_{j}^{2}\big)\int\bigg[\frac{1}{2}\widetilde{z}_{x}^{2}-\frac{1}{4}\big((\widetilde{z}+p)^{4}-p^{4}-4\widetilde{z}p^{3}\big)\bigg]\varphi_{j}.
\end{align}
We observe that the difference between $H$ and $\widetilde{H}$ is
bounded by $O\big(\lVert\widetilde{z}(t)\rVert_{H^{2}}^{3}\big)$.
We can thus claim:
\begin{lem}
\label{1lem:6}There exists $C>0$ that do not depend on $z$, such
that for $t$ large enough,
\begin{align}
\lVert\widetilde{z}(t)\rVert_{H^{2}}^{2}\leq C\widetilde{H}(t)+C\sum_{j=1}^{J}\bigg(\int\widetilde{z}P_{j}\bigg)^{2}.
\end{align}
\end{lem}
\emph{Step 5.} A bound for $\frac{d\widetilde{H}}{dt}$.
\begin{lem}
\label{1lem:7}There exists $C>0$ and $\theta>0$ that do not depend
on $z$, such that for $t$ large enough,
\begin{align}
\bigg\lvert\frac{d\widetilde{H}}{dt}\bigg\rvert\leq\frac{C}{t}\lVert\widetilde{z}(t)\rVert_{H^{2}}^{2}+Ce^{-\theta t}\lVert\widetilde{z}(t)\rVert_{H^{2}}\lVert z(t)\rVert_{H^{2}}+C\lVert\widetilde{z}(t)\rVert_{H^{2}}\lVert z(t)\rVert_{H^{2}}^{2}.
\end{align}
\end{lem}
\begin{proof}
We develop the expression of $\widetilde{H}(t)$, we differentiate
each term obtained and we use (\ref{1eq:04}), the fact that $p$ is
a solution of \eqref{1mKdV} and the fact that $(\varphi_{j})_{t}=-\frac{x}{t}(\varphi_{j})_{x}$,
where $\frac{x}{t}$ is bounded independently from $z$ because of
the compact support of $\varphi_{j}$. We obtain several sorts of
terms after doing several integrations by parts and several obvious simplifications.

Several terms are clearly bounded by one of the bounds of the lemma,
because in these terms, the cumulated degree of $z$ and $\widetilde{z}$ is larger than
$2$. As an example, we show how to deal with $\int z_{xxx}z\widetilde{z}_{xx}p$.
We use the fact that $z=\widetilde{z}-\sum_{j=1}^{J}c_{j}K_{j}$, and we obtain the following:
\begin{align}
\MoveEqLeft\int z_{xxx}z\widetilde{z}_{xx}p  =\int\widetilde{z}_{xxx}\widetilde{z}\widetilde{z}_{xx}p-\int\widetilde{z}_{xxx}\bigg(\sum_{j=1}^{J}c_{j}K_{j}\bigg)\widetilde{z}_{xx}p\\
 & -\int\bigg(\sum_{j=1}^{J}c_{j}(K_{j})_{xxx}\bigg)\widetilde{z}\widetilde{z}_{xx}p+\int\bigg(\sum_{j=1}^{J}c_{j}(K_{j})_{xxx}\bigg)\bigg(\sum_{j=1}^{J}c_{j}K_{j}\bigg)\widetilde{z}_{xx}p.
\end{align}
It is easy to see that any of these terms is bounded as we want in
the lemma (several of them are bounded by $\frac{C}{t}\lVert\widetilde{z}(t)\rVert_{H^{2}}^{2}$,
the last one is bounded by $C\lVert\widetilde{z}(t)\rVert_{H^{2}}\lVert z(t)\rVert_{H^{2}}^{2}$),
because of Lemma \ref{1lem:3} and of (\ref{1eq:hyp}).

Other terms contain $\widetilde{z}$ quadratically and contain $(\varphi_{j})_{x}$.
And, $(\varphi_{j})_{x}$ is bounded by $\frac{C}{t}$.
This is why, such terms are bounded by $\frac{C}{t}\lVert\widetilde{z}(t)\rVert_{H^{2}}^{2}$.

Several other terms can be, by doing suitable integrations by parts transformed
in one of the two following expressions:
\begin{align}
6\sum_{j=1}^{J}\int\widetilde{z}\widetilde{z}_{x}p\Big[p_{xxxx}-2(b_{j}^{2}-a_{j}^{2})(p_{xx}+p^{3})+(a_{j}^{2}+b_{j}^{2})^{2}p\\
+5pp_{x}^{2}+5p^{2}p_{xx}+\frac{3}{2}p^{5}\Big]\varphi_{j},
\end{align}
\begin{align}
3\sum_{j=1}^{J}\int\widetilde{z}^{2}p_{x}\Big[p_{xxxx}-2(b_{j}^{2}-a_{j}^{2})(p_{xx}+p^{3})+(a_{j}^{2}+b_{j}^{2})^{2}p\\
+5pp_{x}^{2}+5p^{2}p_{xx}+\frac{3}{2}p^{5}\Big]\varphi_{j}.
\end{align}
To deal with these two expressions, we use the elliptic equation satisfied
by $P_{j}$:
\begin{align}
(P_{j})_{xxxx}-2(b_{j}^{2}-a_{j}^{2})\big((P_{j})_{xx}+P_{j}^{3}\big)+(a_{j}^{2}+b_{j}^{2})^{2}P_{j}\\
+5P_{j}(P_{j})_{x}^{2}+5P_{j}^{2}(P_{j})_{xx}+\frac{3}{2}P_{j}^{5} & =0,\label{1eq:ell}
\end{align}
 and the fact that 
\begin{align}
\big[p_{xxxx}-2(b_{j}^{2}-a_{j}^{2})(p_{xx}+p^{3})+(a_{j}^{2}+b_{j}^{2})^{2}p+5pp_{x}^{2}+5p^{2}p_{xx}+\frac{3}{2}p^{5}\big]\varphi_{j}
\end{align} 
converges exponentially to 
\begin{align}
(P_{j})_{xxxx}-2(b_{j}^{2}-a_{j}^{2})\big((P_{j})_{xx}+P_{j}^{3}\big)+(a_{j}^{2}+b_{j}^{2})^{2}P_{j}\\
+5P_{j}(P_{j})_{x}^{2}+5P_{j}^{2}(P_{j})_{xx}+\frac{3}{2}P_{j}^{5},
\end{align}
which is a direct consequence of (\ref{1eq:00}). This is why, such terms are bounded by $\frac{C}{t}\lVert\widetilde{z}(t)\rVert_{H^{2}}^{2}$.

Other terms contain $(P_{j}^{2}-p^{2})K_{j}$,
which is bounded exponentially, with $c_{j}$ bounded by $\lVert z\rVert_{H^{2}}$.
Those terms are obviously bounded by $Ce^{-\theta t}\lVert\widetilde{z}(t)\rVert_{H^{2}}\lVert z(t)\rVert_{H^{2}}$.

Other terms contain $K_{k}$ (or a derivative) and $\varphi_{j}$
with $j\neq k$. In this case, this product gives an exponential decreasing,
and such a term is bounded by $Ce^{-\theta t}\lVert\widetilde{z}(t)\rVert_{H^{2}}\lVert z(t)\rVert_{H^{2}}$,
using (\ref{1eq:03}).

Therefore, we are left with the following terms:
\begin{align}
\sum_{j=1}^{J}c_{j}'(t)\int\Big[(K_{j})_{xx}\widetilde{z}_{xx}-10K_{j}\widetilde{z}_{x}pp_{x}-5K_{j}\widetilde{z}p_{x}^{2}\\
-10(K_{j})_{x}\widetilde{z}pp_{x}-5(K_{j})_{x}\widetilde{z}_{x}p^{2}+\frac{15}{4}K_{j}\widetilde{z}p^{4}\\+
2(b_{j}^{2}-a_{j}^{2})(K_{j})_{x}\widetilde{z}_{x}-6(b_{j}^{2}-a_{j}^{2})K_{j}\widetilde{z}p^{2}+(a_{j}^{2}+b_{j}^{2})^{2}K_{j}\widetilde{z}\Big]\varphi_{j}.
\end{align}
We may replace $p$ by $P_{j}$ in the preceeding expression with
an error bounded by 
\begin{align}
Ce^{-\theta t}\lVert\widetilde{z}(t)\rVert_{H^{2}}\lVert z(t)\rVert_{H^{2}},
\end{align}
because of (\ref{1eq:03}) and (\ref{1eq:00}). This is acceptable,
knowing the result we want to prove.
By integration by parts, we obtain several terms of the form $c_{j}'(t)\int(K_{j})_{xx}\widetilde{z}_{x}(\varphi_{j})_{x}$,
which are bounded by $\frac{C}{t}\lvert c_{j}'(t)\rvert\lVert\widetilde{z}(t)\rVert_{H^{2}}$.
Now, from Lemma \ref{1lem:4}, we deduce that they are bounded by \begin{align}
\frac{C}{t}\lVert\widetilde{z}(t)\rVert_{H^{2}}^{2}+Ce^{-\theta t}\lVert\widetilde{z}(t)\rVert_{H^{2}}\lVert z(t)\rVert_{H^{2}}+C\lVert\widetilde{z}(t)\rVert_{H^{2}}\lVert z(t)\rVert_{H^{2}}^{2},
\end{align} 
which is exactly the bound that we want. And, we are left with the following
terms:
\begin{align}
\sum_{j=1}^{J}c_{j}'(t)\int\Big[(K_{j})_{xxxx}+10(K_{j})_{x}P_{j}(P_{j})_{x}+5K_{j}(P_{j})_{x}^{2}\\
+10K_{j}P_{j}(P_{j})_{xx}+5(K_{j})_{xx}P_{j}^{2}+\frac{15}{2}K_{j}P_{j}^{4}\\
-2(b_{j}^{2}-a_{j}^{2})(K_{j})_{xx}-6(b_{j}^{2}-a_{j}^{2})K_{j}P_{j}^{2}+(a_{j}^{2}+b_{j}^{2})^{2}K_{j}\Big]\widetilde{z}\varphi_{j}.
\end{align}

The last expression equals zero, because of the elliptic equation
satisfied by $K_{j}$, which we may derive by differentiating (\ref{1eq:ell}).
\end{proof}
\emph{Step 6.} A bound for $\frac{d}{dt}\int\widetilde{z}P_{j}$.
\begin{lem}
There exists $C>0$ and $\theta>0$ that do not depend on $z$, such
that for $t$ large enough, for any $j=1,...,J$,
\begin{align}
\bigg\lvert\frac{d}{dt}\int\widetilde{z}P_{j}\bigg\rvert\leq Ce^{-\theta t}\lVert z(t)\rVert_{H^{2}}+C\lVert z(t)\rVert_{H^{2}}^{2}.
\end{align}
\end{lem}
\begin{proof}
We observe that
\begin{align}
\int\widetilde{z}P_{j}=\int zP_{j}+\sum_{k=1}^{J}c_{k}(t)\int K_{k}P_{j}.
\end{align}

Firstly, for $k=j$,
\begin{align}
\int K_{j}P_{j}=0,
\end{align}
and for $k\neq j$,
\begin{align}
\frac{d}{dt}\bigg[c_{k}(t)\int K_{k}P_{j}\bigg]=c_{k}'(t)\int K_{k}P_{j}+c_{k}(t)\int(K_{k})_{t}P_{j}+c_{k}(t)\int K_{k}(P_{j})_{t},
\end{align}
and it is obvious, from Lemma \ref{1lem:3} and (\ref{1eq:03}), that
the latter is bounded by $Ce^{-\theta t}\lVert z(t)\rVert_{H^{2}}$.

It is left to bound $\frac{d}{dt}\int zP_{j}$. We use (\ref{1eq:z})
and we obtain:
\begin{align}
\frac{d}{dt}\int zP_{j}=-\int\big(z_{xx}+(z+p)^{3}-p^{3}\big)_{x}P_{j}-\int z\big((P_{j})_{xx}+P_{j}^{3}\big)_{x}.
\end{align}

Several terms are immediately boundable by $C\lVert z(t)\rVert_{H^{2}}^{2}$,
we kill several others by integration by parts and we are left with
\begin{align}
\int z(p^{2}-P_{j}^{2})(P_{j})_{x},
\end{align}
which is obviously bounded by $Ce^{-\theta t}\lVert z(t)\rVert_{H^{2}}$,
because of (\ref{1eq:00}).
\end{proof}
By differentiation of a square, we obtain that
\begin{lem}
\label{1lem:9}There exists $C>0$ and $\theta>0$ that do not depend
on $z$, such that for $t$ large enough, for any $j=1,...,J$,
\begin{align}
\Bigg\lvert\frac{d}{dt}\bigg(\int\widetilde{z}P_{j}\bigg)^{2}\Bigg\rvert\leq Ce^{-\theta t}\lVert\widetilde{z}(t)\rVert_{H^{2}}\lVert z(t)\rVert_{H^{2}}+C\lVert\widetilde{z}(t)\rVert_{H^{2}}\lVert z(t)\rVert_{H^{2}}^{2}.
\end{align}
\end{lem}
\emph{Step 7.} A bound for $\lVert z(t)\rVert_{H^{2}}$ in function
of $\widetilde{z}(t)$.

Because we have chosen $N>2$ and because of (\ref{1eq:hyp}), we may
claim that for $t$ large enough, the integral
\begin{equation*}
\int_{t}^{+\infty}\lVert z(s)\rVert_{H^{2}}\,ds
\end{equation*}
is finite.

Because of Lemma \ref{1lem:3} and (\ref{1eq:hyp}), we deduce that
\begin{align}
c_{j}(t)\rightarrow_{t\rightarrow+\infty}0.
\end{align}

Knowing this, from Lemma \ref{1lem:4}, we deduce by integration that
\begin{align}
\lvert c_{j}(t)\rvert & \leq\int_{t}^{+\infty}\lvert c_{j}'(s)\rvert \,ds\\
 & \leq C\int_{t}^{+\infty}\lVert\widetilde{z}(s)\rVert_{H^{2}}\,ds+C\int_{t}^{+\infty}e^{-\theta s}\lVert z(s)\rVert_{H^{2}}\,ds+\int_{t}^{+\infty}\lVert z(s)\rVert_{H^{2}}^{2}\,ds.
\end{align}

Knowing this and using (\ref{1eq:def}), we may deduce that
\begin{align}
\lVert z(t)\rVert_{H^{2}} & \leq C\lVert\widetilde{z}(t)\rVert_{H^{2}}+C\int_{t}^{+\infty}\lVert\widetilde{z}(s)\rVert_{H^{2}}\,ds+C\int_{t}^{+\infty}e^{-\theta s}\lVert z(s)\rVert_{H^{2}}\,ds\\
&\quad+\int_{t}^{+\infty}\lVert z(s)\rVert_{H^{2}}^{2}\,ds\\
 & \leq C\lVert\widetilde{z}(t)\rVert_{H^{2}}+C\int_{t}^{+\infty}\lVert\widetilde{z}(s)\rVert_{H^{2}}\,ds+C\sup_{s\geq t}\lVert z(s)\rVert_{H^{2}}e^{-\theta t}\\
 &\quad+C\sup_{s\geq t}\lVert z(s)\rVert_{H^{2}}\int_{t}^{+\infty}\lVert z(s)\rVert_{H^{2}}\,ds,
\end{align}
which implies, because 
\begin{align}
\int_{t}^{+\infty}\lVert\widetilde{z}(s)\rVert_{H^{2}}\,ds,\quad\sup_{s\geq t}\lVert z(s)\rVert_{H^{2}}e^{-\theta t},\quad\sup_{s\geq t}\lVert z(s)\rVert_{H^{2}}\int_{t}^{+\infty}\lVert z(s)\rVert_{H^{2}}\,ds
\end{align}
are decreasing in time, that
\begin{align}
\sup_{s\geq t}\lVert z(s)\rVert_{H^{2}} & \leq C\sup_{s\geq t}\lVert\widetilde{z}(s)\rVert_{H^{2}}+C\int_{t}^{+\infty}\lVert\widetilde{z}(s)\rVert_{H^{2}}\,ds+C\sup_{s\geq t}\lVert z(s)\rVert_{H^{2}}e^{-\theta t}\\
& +C\sup_{s\geq t}\lVert z(s)\rVert_{H^{2}}\int_{t}^{+\infty}\lVert z(s)\rVert_{H^{2}}\,ds,
\end{align}
and because $e^{-\theta t}$ and $\int_{t}^{+\infty}\lVert z(s)\rVert_{H^{2}}\,ds$
may be as small as we want for $t$ large enough (dependent on $z$),
we may deduce that
\begin{lem}
\label{1lem:10}There exists $C>0$ that do not depend on $z$, such
that for $t$ large enough,
\begin{align}
\lVert z(t)\rVert_{H^{2}}\leq\sup_{s\geq t}\lVert z(s)\rVert_{H^{2}}\leq C\sup_{s\geq t}\lVert\widetilde{z}(s)\rVert_{H^{2}}+C\int_{t}^{+\infty}\lVert\widetilde{z}(s)\rVert_{H^{2}}\,ds.
\end{align}
\end{lem}
\emph{Step 8.} Conclusion.

By integration, from Lemmas \ref{1lem:6}, \ref{1lem:7} and \ref{1lem:9},
for $t$ large enough (depending on $z$), with constants $C$ and
$\theta$ that do not depend on $z$,
\begin{align}
\lVert\widetilde{z}(t)\rVert_{H^{2}}^{2} & \leq C\int_{t}^{+\infty}\frac{1}{s}\lVert\widetilde{z}(s)\rVert_{H^{2}}^{2}\,ds+C\int_{t}^{+\infty}e^{-\theta s}\lVert\widetilde{z}(s)\rVert_{H^{2}}\lVert z(s)\rVert_{H^{2}}\,ds\\
&\quad+C\int_{t}^{+\infty}\lVert\widetilde{z}(s)\rVert_{H^{2}}\lVert z(s)\rVert_{H^{2}}^{2}\,ds  \\
 & \leq C\sup_{s\geq t}\lVert\widetilde{z}(s)\rVert_{H^{2}}\int_{t}^{+\infty}\Big(\frac{1}{s}\lVert\widetilde{z}(s)\rVert_{H^{2}}+e^{-\theta s}\lVert z(s)\rVert_{H^{2}}+\lVert z(s)\rVert_{H^{2}}^{2}\Big)\,ds.
\end{align}
Because the right-hand side of the inequality above
is decreasing in time, we deduce after taking the supremum of the
previous inequality and after simplification, that for $t$ large
enough,
\begin{align}
\sup_{s\geq t}\lVert\widetilde{z}(s)\rVert_{H^{2}} & \leq C\int_{t}^{+\infty}\frac{1}{s}\lVert\widetilde{z}(s)\rVert_{H^{2}}\,ds+C\int_{t}^{+\infty}e^{-\theta s}\lVert z(s)\rVert_{H^{2}}\,ds\\
&\quad+C\int_{t}^{+\infty}\lVert z(s)\rVert_{H^{2}}^{2}\,ds\\
 & \leq C\int_{t}^{+\infty}\frac{1}{s}\lVert\widetilde{z}(s)\rVert_{H^{2}}\,ds+C\sup_{s\geq t}\lVert z(s)\rVert_{H^{2}}e^{-\theta t}\\
 &\quad+C\sup_{s\geq t}\lVert z(s)\rVert_{H^{2}}\int_{t}^{+\infty}\lVert z(s)\rVert_{H^{2}}\,ds.
\end{align}
And using (\ref{1eq:hyp}), the fact that $N-1>1$ and the fact that
$e^{-\theta t}$ is decreasing faster than $\frac{1}{t^{N-2}}$, we
deduce that for $t$ large enough,
\begin{align}
\sup_{s\geq t}\lVert\widetilde{z}(s)\rVert_{H^{2}}\leq C\int_{t}^{+\infty}\frac{1}{s}\lVert\widetilde{z}(s)\rVert_{H^{2}}\,ds+C\frac{1}{t^{N-2}}\sup_{s\geq t}\lVert z(s)\rVert_{H^{2}}.
\end{align}
And using Lemma \ref{1lem:10}, we deduce that
\begin{align}
\sup_{s\geq t}\lVert\widetilde{z}(s)\rVert_{H^{2}} &\leq C\int_{t}^{+\infty}\frac{1}{s}\lVert\widetilde{z}(s)\rVert_{H^{2}}\,ds+C\frac{1}{t^{N-2}}\sup_{s\geq t}\lVert\widetilde{z}(s)\rVert_{H^{2}}\\
&\quad+C\frac{1}{t^{N-2}}\int_{t}^{+\infty}\lVert\widetilde{z}(s)\rVert_{H^{2}}\,ds.
\end{align}
And because $\frac{1}{t^{N-2}}$ can be as small as we want for $t$
large enough, we deduce that for $t$ large enough and for a constant
$C>0$ that do not depend on $z$ or on $N$,
\begin{align}
\lVert\widetilde{z}(t)\rVert_{H^{2}}\leq\sup_{s\geq t}\lVert\widetilde{z}(s)\rVert_{H^{2}}\leq C\int_{t}^{+\infty}\frac{1}{s}\lVert\widetilde{z}(s)\rVert_{H^{2}}\,ds+C\frac{1}{t^{N-2}}\int_{t}^{+\infty}\lVert\widetilde{z}(s)\rVert_{H^{2}}\,ds.\label{1eq:fin}
\end{align}

Let us pick $T>0$ large enough such that for $t\geq T$, the inequality
(\ref{1eq:fin}) works (i.e. $T$ is large enough so that every part of the preceeding proof works).
From (\ref{1eq:def}) and Lemma \ref{1lem:3}, we know that for $t\geq T$
(by taking $T$ larger if needed),
\begin{align}
\lVert\widetilde{z}(t)\rVert_{H^{2}}\leq\frac{C}{t^{N-1}}.\label{1eq:hyp-tild}
\end{align}
This is why, the following quantity is well defined:
\begin{align}
A:=\sup_{t\geq T}\{t^{N-1}\lVert\widetilde{z}(t)\rVert_{H^{2}}\},\label{1eq:sup}
\end{align}
which means that for $t\geq T$,
\begin{align}
\lVert\widetilde{z}(t)\rVert_{H^{2}}\leq\frac{A}{t^{N-1}}.\label{1eq:conc}
\end{align}

Now, using (\ref{1eq:hyp-tild}) and (\ref{1eq:conc}), we deduce from
(\ref{1eq:fin}) that for $t\geq T$, with $C>0$ that do not depend
on $z$, on $N$ or on $A$,
\begin{align}
\lVert\widetilde{z}(t)\rVert_{H^{2}}\leq\frac{CA}{N-1}\frac{1}{t^{N-1}}+\frac{CA}{N-2}\frac{1}{t^{2N-4}}\leq\frac{CA}{N-2}\frac{1}{t^{N-1}},\label{1eq:concl}
\end{align}
if we assume that $N>3$.
Now, from (\ref{1eq:sup}), we deduce that there exists $T^{*}>T$
such that
\begin{align}
(T^{*})^{N-1}\lVert\widetilde{z}(T^{*})\rVert_{H^{2}}\geq\frac{A}{2}.
\end{align}
This is why, by evaluating (\ref{1eq:concl}) in $t=T^{*}$, we find
that
\begin{align}
\frac{A}{2(T^{*})^{N-1}}\leq\frac{CA}{N-2}\frac{1}{(T^{*})^{N-1}},
\end{align}
which, if we assume that $A>0$, after simplification yields:
\begin{align}
N-2\leq2C.
\end{align}

This means that if we assume that $N>2C+2$ and $N>3$, the assumption
$A>0$ leads to a contradiction. Therefore, $A=0$ under that assumption
on $N$, which implies $\lVert\widetilde{z}(t)\rVert_{H^{2}}=0$, and
from Lemma \ref{1lem:10}, this implies that $z\equiv0$. This means
that the condition that we have established for $N$, namely
\begin{align}
N>\max(2C+2,3),
\end{align}
do not depend on $z$ and allows us to deduce that under (\ref{1eq:hyp}),
we may establish that $z\equiv0$. The Proposition \ref{1lem:polyn} is now proved.
\end{proof}

\subsection{A solution converging to a multi-breather converges exponentially
to this multi-breather, if the velocities are positive}
\label{1sec:uniq_mon}

\begin{prop} \label{1prop:conv_exp}
Let $u(t)$ be an $H^{2}$ solution of \eqref{1mKdV} on $[T,+\infty)$,
for $T\in\mathbb{R}$. We assume that 
\begin{align}
\lVert u(t)-p(t)\rVert_{H^{2}}\rightarrow_{t\rightarrow+\infty}0,\label{1eq:-114}
\end{align}
where $p$ is the multi-breather constructed in Section \ref{1sec:constr}.
If 
\begin{align}
v_{1}>0, 
\end{align}
then there exists $\varpi>0$, $T_{0}\geq  T$ and $C>0$
such that for any $t\geq  T_{0}$,
\begin{align}
\lVert u(t)-p(t)\rVert_{H^{2}}\leq  Ce^{-\varpi t}.\label{1eq:-115}
\end{align}
\end{prop}

Note that in the formulation of the Proposition above, we may replace $p$ by $P$ without changing its content (it is a consequence from \eqref{1eq:-12-1-1}).

\begin{proof}
We set $v(t):=u(t)-P(t)$, such that $\lVert v(t)\rVert_{H^{2}}\rightarrow_{t\rightarrow+\infty}0$.

We denote: 
\begin{align}
\Psi(x):=\frac{2}{\pi}\arctan\big(\exp(-\sqrt{\sigma}x/2)\big),\label{1eq:-117}
\end{align}
where $\sigma>0$ is small enough (with precise conditions that will
be mentioned throughout the proof). By direct calculations, 
\begin{align}
\Psi'(x)=\frac{-\sqrt{\sigma}}{2\pi\cosh(\sqrt{\sigma}x/2)}.\label{1eq:-118}
\end{align}
Thus, 
\begin{align}
\lvert\Psi'(x)\rvert\leq  C\exp(-\sqrt{\sigma}\lvert x\rvert/2).\label{1eq:-119}
\end{align}
We have the following properties: $\lim_{+\infty}\Psi=0$, $\lim_{-\infty}\Psi=1$,
for all $x\in\mathbb{R}$ $\Psi(-x)=1-\Psi(x)$, $\Psi'(x)<0$, $\lvert\Psi''(x)\rvert\leq \frac{\sqrt{\sigma}}{2}\lvert\Psi'(x)\rvert$,
$\lvert\Psi'''(x)\rvert\leq \frac{\sqrt{\sigma}}{2}\lvert\Psi''(x)\rvert$,
$\lvert\Psi'(x)\rvert\leq \frac{\sqrt{\sigma}}{2}\Psi$ and $\lvert\Psi'(x)\rvert\leq \frac{\sqrt{\sigma}}{2}(1-\Psi)$.

For $j=2,...,J$, let $m_{j}$ be such that
\begin{align}
m_j=\frac{v_{j-1}+v_j}{2}.\label{1eq:-120}
\end{align}
Let us denote $\tau_{0}>0$ the minimal distance between a $v_{j}$
and a $m_{j}$.

From this, we define for $j=2,...,J$, 
\begin{align}
\Phi_{j}(t,x):=\Psi(x-m_{j}t).\label{1eq:-121}
\end{align}
We may extend this definition to $j=1$ and $j=J+1$ in the following
way: $\Phi_{1}:=0$ and $\Phi_{J+1}:=1$.
Thus, the function that allows us to study properties around each
object $P_j$ (for $j=1,...,J$) is $\chi_{j}:=\Phi_{j+1}-\Phi_{j}$.

The goal is to prove that, for $t$ large enough,
\begin{align}
\lVert v(t)\rVert_{H^{2}}\leq  Ce^{-\varpi t},\label{1eq:-122}
\end{align}
where $\varpi>0$ is a constant to be deduced from the constants of
the problem. Proposition \ref{1prop:conv_exp} follows from this, because of Theorem \ref{1thm:MAIN}.

Let $\varpi>0$ to be deduced from the constants of the problem with
respect to the needs of the following proof.

We will prove \eqref{1eq:-122} by induction. We will prove, for $j=2,...,J+1$,
that $\int(v^{2}+v_{x}^{2}+v_{xx}^{2})\Phi_{j}\leq  Ce^{-2\varpi t}$
for $t$ large enough, knowing that $\int(v^{2}+v_{x}^{2}+v_{xx}^{2})\Phi_{j-1}\leq  Ce^{-2\varpi t}$
for $t$ large enough (note that this assumption is empty when $j=2$).
This implies the desired inequality.
(Note that it is OK if $\varpi$ becomes smaller after a step of this induction, as long as it stays positive.)

Let us write the $j$-th step of our reasoning by induction (where
$j\in\{2,...,J+1\}$). Thus, $j$ is fixed in the rest of the proof.
We assume that
\begin{align}
\int\big(v^{2}+v_{x}^{2}+v_{xx}^{2}\big)\Phi_{j-1}\leq  Ce^{-2\varpi t}.\label{1eq:-123}
\end{align}
We divide our proof in several steps.

\emph{Step 1.} Almost-conservation of localized conservation laws.

We define quantities that are similar to quantities defined in Section \ref{1sec:2.2}. We note that we localize around the first $j-1$ objects, not only around the $(j-1)$-th object. Notations defined in Section \ref{1sec:2.2} should not
be considered in the following proof and should be replaced by notations
we define here:
\begin{align}
M_{j}(t):=\frac{1}{2}\int u^{2}(t)\Phi_{j}(t),\label{1eq:locm}
\end{align}
\begin{align}
E_{j}(t):=\int\Big[\frac{1}{2}u_{x}^{2}-\frac{1}{4}u^{4}\Big]\Phi_{j}(t),\label{1eq:loce}
\end{align}
\begin{align}
F_{j}(t):=\int\Big[\frac{1}{2}u_{xx}^{2}-\frac{5}{2}u^{2}u_{x}^{2}+\frac{1}{4}u^{6}\Big]\Phi_{j}(t).\label{1eq:locf}
\end{align}
\begin{lem}
\label{1lem:monoto}Let $\omega_{2},\omega_{6}>0$, as small as desired. There exists $T_{1}\geq  T$ and $C>0$ such that for $t\geq  T_{1}$,
\begin{align}
\sum_{i=1}^{j-1}M[P_{i}]-M_{j}(t)\geq -Ce^{-2\varpi t},\label{1eq:mm}
\end{align}
\begin{align}
\sum_{i=1}^{j-1}\big(E[P_{i}]+\omega_{2}M[P_{i}]\big)-\big(E_{j}(t)+\omega_{2}M_{j}(t)\big)\geq -Ce^{-2\varpi t},\label{1eq:me-1}
\end{align}
\begin{align}
\sum_{i=1}^{j-1}\big(F[P_{i}]+\omega_{6}M[P_{i}]\big)-\big(F_{j}(t)+\omega_{6}M_{j}(t)\big)\geq -Ce^{-2\varpi t}.\label{1eq:mf}
\end{align}
\end{lem}

\begin{proof}
We will use the results of the computations made at the bottom of
 page 1115 and at the bottom of page 1116 of \cite{key-2},
as well as in Section \ref{1sec:55} (Appendix) to claim the three following facts:
\begin{align}
\frac{d}{dt}\frac{1}{2}\int u^{2}f & =\int\Big(-\frac{3}{2}u_{x}^{2}+\frac{3}{4}u^{4}\Big)f'+\frac{1}{2}\int u^{2}f''', \\
\frac{d}{dt}\int\Big[\frac{1}{2}u_{x}^{2}-\frac{1}{4}u^{4}\Big]f & =\int\Big[-\frac{1}{2}\big(u_{xx}+u^{3}\big)^{2}-u_{xx}^{2}+3u_{x}^{2}u^{2}\Big]f'+\frac{1}{2}\int u_{x}^{2}f''',\label{1eq:-124}
\end{align}
\begin{align}
\MoveEqLeft\frac{d}{dt}\int\Big(\frac{1}{2}u_{xx}^{2}-\frac{5}{2}u^{2}u_{x}^{2}+\frac{1}{4}u^{6}\Big)f \\
 & =\int\Big(-\frac{3}{2}u_{xxx}^{2}+9u_{xx}^{2}u^{2}+15u_{x}^{2}uu_{xx}+\frac{9}{16}u^{8}+\frac{1}{4}u_{x}^{4}+\frac{3}{2}u_{xx}u^{5}\\
 & \qquad-\frac{45}{4}u^{4}u_{x}^{2}\Big)f' +5\int u^{2}u_{x}u_{xx}f''+\frac{1}{2}\int u_{xx}^{2}f'''.\label{1eq:-125}
\end{align}
where $f$ is a $C^{3}$ function that does not depend on time.

\emph{For the mass:}

If $j\leq  J$,

\begin{align}
2\frac{d}{dt}M_{j}(t)=-\int\Big(3u_{x}^{2}+m_{j}u^{2}-\frac{3}{2}u^{4}\Big)\Phi_{jx}(t)+\int u^{2}\Phi_{jxxx}(t).\label{1eq:-126}
\end{align}
We recall that 
\begin{align}
\lvert\Phi_{jxx}\rvert\leq \frac{\sqrt{\sigma}}{2}\lvert\Phi_{jx}\rvert,\quad\lvert\Phi_{jxxx}\rvert\leq \frac{\sigma}{4}\lvert\Phi_{jx}\rvert,\quad\Phi_{jx}\leq 0,\label{1eq:-127}
\end{align}
where we can choose $\sigma$ as small as desired. For this proof,
we would like to ask for $\sigma$:
\begin{align}
0<\sigma\leq  m_{2}\leq  m_{j}.\label{1eq:-128}
\end{align}
Thus,
\begin{align}
2\frac{d}{dt}M_{j}(t)\geq \int\Big(3u_{x}^{2}+\frac{3\sigma}{4}u^{2}-\frac{3}{2}u^{4}\Big)\lvert\Phi_{jx}(t)\rvert .\label{1eq:-129}
\end{align}

By Corollary \ref{1cor:unifdecay}, for $r>0$, if $t,x$ satisfy $v_{j-1}t+r<x<v_{j}t-r$,
then
\begin{align}
\lvert u(t,x)\rvert & \leq \lvert P(t,x)\rvert+\lVert v(t)\rVert_{L^{\infty}} \\
 & \leq  Ce^{-\beta r}+C\lVert v(t)\rVert_{H^{2}},\label{1eq:-130}
\end{align}
the same could be said for $u_{x}$.

We can thus deduce that for $r$ large enough and for $T_{1}$
large enough, for $x\in(v_{j-1}t+r,v_{j}t-r)$, we can obtain that
$\lvert u\rvert$ is bounded by any fixed constant, that can be taken
as small as desired. Here, we will use the latter to bound $\frac{3}{2}u^{2}$
by $\frac{\sigma}{4}$.

For $t\geq  T_{1}$ and $x\leq  v_{j-1}t+r$ or $x\geq  v_{j}t-r$, we
have $\lvert x-m_{j}t\rvert\leq \tau_{0}t-r$, and therefore for such $t,x$:
\begin{align}
\lvert\Phi_{jx}(t,x)\rvert & \leq  C\exp\big(-\sqrt{\sigma}\lvert x-m_{j}t\rvert/2\big) \\
 & \leq  C\exp\big(-\sqrt{\sigma}\tau_{0}t/2\big)\exp\big(\sqrt{\sigma}r/2\big).\label{1eq:-131}
\end{align}

Because $\int u^{4}$ is bounded by a constant for any time and $\exp(\sqrt{\sigma}r/2)$
is a fixed constant ($r$ is already chosen), we have, for $t\geq  T_{1}$,
\begin{align}
\frac{d}{dt}M_{j}(t)\geq \int\Big(\frac{3}{2}u_{x}^{2}+\frac{\sigma}{4}u^{2}\Big)\lvert\Phi_{jx}(t)\rvert-Ce^{-2\varpi t}\geq -Ce^{-2\varpi t},\label{1eq:-132}
\end{align}
where $\varpi$ is chosen as a suitable function of $\sigma$ and
$\tau_{0}$.

By integration, we deduce that for any $t_{1}\geq  t$, with a constant
$C>0$ that does not depend on $t_{1}$, we have:
\begin{align}
M_{j}(t_{1})-M_{j}(t)\geq -Ce^{-2\varpi t}.\label{1eq:diff_finie}
\end{align}
We note that this conclusion is immediate when $j=J+1$, because we have
exactly the conserved quantity.

We have that
\begin{align}
\MoveEqLeft\bigg\lvert\sum_{i=1}^{j-1}M[P_{i}]-M_{j}(t_{1})\bigg\rvert\\
 & \leq \bigg\lvert\sum_{i=1}^{j-1}\frac{1}{2}\int P_{i}^{2}-\frac{1}{2}\int P^{2}\Phi_{j}(t_{1})\bigg\rvert +\frac{1}{2}\bigg\lvert\int P^{2}\Phi_{j}(t_{1})-\int u^{2}\Phi_{j}(t_{1})\bigg\rvert \\
 & \leq  Ce^{-\kappa(\beta,\sigma,\tau_{0})t_{1}}+\frac{1}{2}\int\lvert P^{2}-u^{2}\rvert\Phi_{j}(t_{1}) \\
 & \leq  Ce^{-\kappa(\beta,\sigma,\tau_{0})t_{1}}+C\int\lvert P^{2}-u^{2}\rvert\rightarrow_{t_{1}\rightarrow+\infty}0.\label{1eq:-133}
\end{align}

This means that when we take the limit of (\ref{1eq:diff_finie}) when
$t_{1}\rightarrow+\infty$, we obtain, for $t\geq  T_{1}$,
\begin{align}
\sum_{i=1}^{j-1}M[P_{i}]-M_{j}(t)\geq -Ce^{-2\varpi t},\label{1eq:-134}
\end{align}
which is exactly what we wished to prove.

\emph{For the energy:}

If $j\leq  J$, 
\begin{align}
2\frac{d}{dt}E_{j}(t) & =\int\Big[-\big(u_{xx}+u^{3}\big)^{2}-2u_{xx}^{2}+6u_{x}^{2}u^{2}\Big]\Phi_{jx}(t)\\
&\quad-m_{j}\int\Big(u_{x}^{2}-\frac{1}{2}u^{4}\Big)\Phi_{jx}(t)+\frac{1}{2}\int u_{x}^{2}\Phi_{jxxx}(t) \\
 & \geq \int\Big[\big(u_{xx}+u^{3}\big)^{2}+2u_{xx}^{2}-6u_{x}^{2}u^{2}+\frac{3\sigma}{4}u_{x}^{2}-\frac{m_{j}}{2}u^{4}\Big]\lvert\Phi_{jx}(t)\rvert.\label{1eq:-135}
\end{align}

We can do the same reasoning as for the mass to bound above $\frac{m_{j}}{2}u^{2}$
by $\omega_{1}$, a constant that we can choose as small as desired,
and to bound above $6u^{2}$ by $\frac{\sigma}{4}$. We obtain that
if $T_{1}$ is large enough (dependently on the chosen constant $\omega_{1}$),
\begin{align}
2\frac{d}{dt}E_{j}(t)\geq \int\Big[\big(u_{xx}+u^{3}\big)^{2}+2u_{xx}^{2}+\frac{\sigma}{2}u_{x}^{2}-\omega_{1}u^{2}\Big]\lvert\Phi_{jx}(t)\rvert-Ce^{-2\varpi t}.\label{1eq:-136}
\end{align}

By using what we have performed for the mass, we have that if we take
$\omega_{1}$ small enough with respect to $\frac{\omega_{2}\sigma}{2}$,
\begin{align}
\frac{d}{dt}\big(E_{j}+\omega_{2}M_{j}\big)(t)\geq -Ce^{-2\varpi t}.\label{1eq:-137}
\end{align}
Then, by integration and similarly as for the mass, we obtain
the desired conclusion that is true for any $j$.

\emph{For \mbox{$F$}:}

If $j\leq  J$,
\begin{align}
2\frac{d}{dt}F_{j}(t) & =\int\Big(-3u_{xxx}^{2}+18u_{xx}^{2}u^{2}+30u_{x}^{2}uu_{xx}+\frac{9}{8}u^{8}+\frac{1}{2}u_{x}^{4}+3u_{xx}u^{5}\\
&\qquad-\frac{45}{2}u^{4}u_{x}^{2}\Big)\Phi_{jx}(t) 
  -m_{j}\int\Big(u_{xx}^{2}-5u^{2}u_{x}^{2}+\frac{1}{2}u^{6}\Big)\Phi_{jx}(t)\\
  &\quad+10\int u^{2}u_{xx}u_{x}\Phi_{jxx}(t)+\int u_{xx}^{2}\Phi_{jxxx}(t) \\
 & \geq \int\Big(3u_{xxx}^{2}+\frac{45}{2}u^{4}u_{x}^{2}-18u_{xx}^{2}u^{2}-15u_{x}^{2}u^{2}-15u_{x}^{2}u_{xx}^{2}-\frac{9}{8}u^{8}\\
 &\qquad-\frac{1}{2}u_{x}^{4}-\frac{3}{2}u_{xx}^{2}u^{4}-\frac{3}{2}u^{6}\Big)\lvert\Phi_{jx}(t)\rvert \\
 & +\int\Big(\sigma u_{xx}^{2}+\frac{\sigma}{2}u^{6}-5m_{j}u^{2}u_{x}^{2}\Big)\lvert\Phi_{jx}(t)\rvert-5\int u^{2}u_{x}^{2}\lvert\Phi_{jxx}(t)\rvert \\ & -5\int u^{2}u_{xx}^{2}\lvert\Phi_{jxx}(t)\rvert-\int u_{xx}^{2}\lvert\Phi_{jxxx}(t)\rvert.\label{1eq:-138}
\end{align}

By the same reasoning as for the energy and the mass,
if we set $\omega_{3},\omega_{4},\omega_{5}>0$ constants that we
can take as small as desired, and if $T_{1}$ is large enough dependently
on these constants, for $t\geq  T_{1}$, we have that
\begin{align}
2\frac{d}{dt}F_{j}(t) &\geq \int\Big(3u_{xxx}^{2}+\frac{45}{2}u^{4}u_{x}^{2}+\frac{3\sigma}{4}u_{xx}^{2}+\frac{\sigma}{2}u^{6}-\omega_{3}u_{xx}^{2}-\omega_{4}u_{x}^{2} \\
&\qquad-\omega_{5}u^{2}\Big)\lvert\Phi_{jx}(t)\rvert-Ce^{-2\varpi t}.\label{1eq:-139}
\end{align}

By using what we have carried out for the mass, we have that if we
take $\omega_{3},\omega_{4},\omega_{5}$ small enough (with respect to $\omega_6$),
\begin{align}
\frac{d}{dt}\big(F_{j}+\omega_{6}M_{j}\big)(t)\geq -Ce^{-2\varpi t}.\label{1eq:-140}
\end{align}
Then, by integration and similarly as before, we obtain that the desired
conclusion true for any $j$.
\end{proof}
\begin{rem}
\label{1rem:j=00003DJ+1}If $j=J+1$, we have that
\begin{align}
\sum_{i=1}^{J}M[P_{i}]-M_{J+1}(t)=0,\label{1eq:-141}
\end{align}
\begin{align}
\sum_{i=1}^{J}E[P_{i}]-E_{J+1}(t)=0,\label{1eq:-142}
\end{align}
\begin{align}
\sum_{i=1}^{J}F[P_{i}]-F_{J+1}(t)=0.\label{1eq:-143}
\end{align}
\end{rem}

\emph{Step 2.} Modulation. 

Notations that were defined in Section
\ref{1sec:2.3} should not be taken into consideration in the following proof and should be replaced by
notations we define here.
\begin{lem}
\label{2lem:mod_uniq}
There exists $C>0$, $T_{2}\geq  T$, such that there exist unique
$C^{1}$ functions $y_{1},y_{2}:[T_{2},+\infty)\rightarrow\mathbb{R}$
such that if we set:
\begin{align}
w(t,x):=u-\widetilde{P},\label{1eq:-144}
\end{align}
where
\begin{align}
\widetilde{P}(t,x):=\sum_{i=1}^{J}\widetilde{P_{i}}(t,x),\label{1eq:-145}
\end{align}
for $i\neq j-1$,
\begin{align}
\widetilde{P_{i}}(t,x):=P_{i}(t,x),\label{1eq:-146}
\end{align}
and either,
\begin{align}
\widetilde{P_{j-1}}(t,x):=\kappa_lQ_{c_l+y_1(t)}(x-x_{0,l}^0+y_2(t)-c_lt),\quad\text{ if $P_{j-1}=R_l$ is a soliton},\label{1eq:-147}
\end{align}
or,
\begin{align}
\widetilde{P_{j-1}}(t,x):=B_{\alpha_k,\beta_k}(t,x;x_{1,k}+y_{1}(t),x_{2,k}+y_{2}(t)),\quad\text{ if $P_{j-1}=B_k$ is a breather},\label{1eq:-148}
\end{align}

then, $w(t)$ satisfies, for any $t\in[T_{2},+\infty)$, either,
\begin{align}
\int\widetilde{P_{j-1}}_{1}(t)w(t)=\int\widetilde{P_{j-1}}_{2}(t)w(t)=0,\quad\text{if $P_{j-1}$ is a breather},\label{1eq:-150}
\end{align}
or,
\begin{align}
\int\widetilde{P_{j-1}}(t)w(t)=\int\widetilde{P_{j-1}}_{x}(t)w(t)=0,\quad\text{if $P_{j-1}$ is a soliton},\label{1eq:-151}
\end{align}
where in the case when $P_{j-1}$ is a breather we denote:
\begin{align}
\widetilde{P_{j-1}}_{1}(t,x):=\partial_{x_{1}}\widetilde{P_{j-1}},\quad\widetilde{P_{j-1}}_{2}(t,x):=\partial_{x_{2}}\widetilde{P_{j-1}}.\label{1eq:-149}
\end{align}

Moreover, for any $t\in[T_{2},+\infty)$,
\begin{align}
\lVert w(t)\rVert_{H^{2}}+\lvert y_{1}(t)\rvert+\lvert y_{2}(t)\rvert\leq  C\lVert v(t)\rVert_{H^{2}},\label{1eq:-152}
\end{align}
and, if $\varpi$ is small enough, 
\begin{align}
\lvert y_{1}'(t)\rvert +\lvert y_{2}'(t)\rvert \leq  C\bigg(\int w(t)^{2}\Phi_{j}\bigg)^{1/2}+Ce^{-\varpi t}.\label{1eq:last}
\end{align}
\end{lem}

\begin{proof}
The proof that has to be performed is similar to the proof of Lemma
\ref{1lem:mod}, which is a consequence of a quantitative version of the implicit function theorem. See \cite[Section 2.2]{CH} for a precise statement.
The proof of (\ref{1eq:last}) is also similar: as in the proof of
Lemma \ref{1lem:mod}, we take the time derivative of $\int\widetilde{P_{j-1}}_{1}(t)w(t)=\int\widetilde{P_{j-1}}_{2}(t)w(t)=0$. To be complete, let us perform this proof.

For $t\in [T_2,+\infty)$, let
\begin{align}
F_t:L^2(\mathbb{R})\times\mathbb{R}^2\to\mathbb{R}^2
\end{align}
be such that if $P_{j-1}=B_k$ is a breather,
\begin{align}
(U,y_1,y_2)  \longmapsto & \bigg(\int\partial_{x_1}B_{\alpha_k,\beta_k}(t,x;x_{1,k}^0+y_1,x_{2,k}^0+y_2)\epsilon\,dx,   \\
&  \int\partial_{x_2}B_{\alpha_k,\beta_k}(t,x;x_{1,k}^0+y_1,x_{2,k}^0+y_2)\epsilon\,dx \bigg),
\end{align}
where
\begin{align}
\epsilon:=U-P+P_{j-1}-B_{\alpha_k,\beta_k}(t,x;x_{1,k}^0+y_1,x_{2,k}^0+y_2),
\end{align}
and if $P_{j-1}=R_l$ is a soliton,
\begin{align}
(U,y_1,y_2)\longmapsto & \bigg(\int\kappa_lQ_{c_l+y_1}(x-x_{0,l}^0+y_2-c_lt)\epsilon\,dx,\\
&\int\partial_x\kappa_lQ_{c_l+y_1}(x-x_{0,l}^0+y_2-c_lt)\epsilon\,dx\bigg),
\end{align}
where
\begin{align}
\epsilon:=U-P+P_{j-1}-\kappa_lQ_{c_l+y_1}(x-x_{0,l}^0+y_2-c_lt).
\end{align}

We observe that $F_t$ is a $C^1$ function and that $F_t(P(t),0,0)=0$. Now, let us consider the matrix which gives the differential of $F_t$ (with respect to $y_1,y_2$) in $(P(t),0,0)$.

In the case when $P_{j-1}=B_k$ is a breather, this matrix is:
\begin{align}
DF_t=\begin{pmatrix}
-\int (\partial_{x_1}B_k)^2\,dx & -\int\partial_{x_1}B_k\partial_{x_2}B_k\,dx \\
-\int\partial_{x_1}B_k\partial_{x_2}B_k\,dx & -\int (\partial_{x_2}B_k)^2\,dx
\end{pmatrix}
,
\end{align}
whose determinant is:
\begin{align}
\det(DF_t)=\int (\partial_{x_1}B_k)^2\,dx\int (\partial_{x_2}B_k)^2\,dx-\bigg(\int\partial_{x_1}B_k\partial_{x_2}B_k\,dx\bigg)^2.
\end{align}
By Cauchy-Schwarz inequality and the fact that $\partial_{x_1}B_k$ and $\partial_{x_2}B_k$ are linearly independent as functions of the $x$ variable, for any time $t$ fixed, we see that $\det(DF_t)$ is positive. Since each member of its expression is periodic in time, then $\det(DF_t)$ is bounded below by a positive constant independent on time and translation parameters of $B_k$.

In the case when $P_{j-1}=R_l$ is a soliton, let us recall, denoting $y_{0,l}:=x-x_{0,l}^0+y_2-c_lt$, that
\begin{align}
\partial_{y_1}Q_{c_l+y_1}(y_{0,l})=\frac{1}{2c_l}\big(Q_{c_l+y_1}(y_{0,l})+y_{0,l}\partial_xQ_{c_l+y_1}(y_{0,l})\big).
\end{align}
Thus, denoting $Q_{c_l}(x-x_{0,l}^0-c_lt)$ by $Q_{c_l}$ and $x-x_{0,l}^0-c_lt$ by $y_{0,l}^0$,
\begin{align}
DF_t=\begin{pmatrix}
-\frac{1}{2c_l}\int Q_{c_l}\big(Q_{c_l}+y_{0,l}^0\partial_xQ_{c_l}\big)\,dx & -\int Q_{c_l}\partial_xQ_{c_l}\,dx \\
-\frac{1}{2c_l}\int \partial_xQ_{c_l}\big(Q_{c_l}+y_{0,l}^0\partial_xQ_{c_l}\big)\,dx & -\int \big(\partial_xQ_{c_l}\big)^2\,dx
\end{pmatrix},
\end{align}
whose determinant is:
\begin{align}
\det(DF_t) & =\frac{1}{2c_l}\int Q_{c_l}\big(Q_{c_l}+y_{0,l}^0\partial_xQ_{c_l}\big)\,dx\int \big(\partial_xQ_{c_l}\big)^2\,dx,
\end{align}
because $\int Q_{c_l}\partial_xQ_{c_l}\,dx=0$. And, from the computations made to obtain \eqref{1prod}, we have that
\begin{align}
\det(DF_t)=\frac{1}{4}c_l\int q^2\int q_x^2,
\end{align}
where $q$ denotes the soliton with $c=1$, i.e. $q=Q_1$.

This means that $\det(DF_t)$ is bounded below by a positive constant independent on time and translation parameters of $R_l$.
Thus, in any case, $DF_t$ is invertible.

Now, we may use the implicit function theorem. If $U$ is close enough to $P(t)$, then there exists $(y_1,y_2)$ such that $F_t(U,y_1,y_2)=0$, where $(y_1,y_2)$ depends in a regular $C^1$ way on $U$. It is possible to show that the ``close enough'' in the previous sentence does not depend on $t$; for this, it is required to use a uniform implicit function theorem. This means that for $T_2$ large enough, $\lVert v(t)\rVert_{H^2}$ is small enough for $t\in [T_2,+\infty)$, thus for $t\geq T_2$, $u(t)$ is close enough to $P(t)$ in order to apply the implicit function theorem. Therefore, we have for $t\in[T_2,+\infty)$, the existence of $y_1(t)$ and $y_2(t)$. It is possible to show that these functions are $C^{1}$ in time.
Basically, this comes from the fact that they are $C^{1}$ in $u(t)$
and that $u(t)$ has a similar regularity in time (see \cite{key-7}
for more details).

Now, we prove the inequalities \eqref{1eq:-152} and \eqref{1eq:last}. We can take the differential of the implicit functions with respect to $u(t)$ for $t\in[T_2,+\infty)$. For this, we differentiate the following equation with respect to $u(t)$:
\begin{align}
F_t(u(t),y_1(u(t)),y_2(u(t)))=0.
\end{align}

We know that the matrix that gives the differential of $F_t$ (with respect to $y_1,y_2$) in 
\begin{equation*}
(u(t),y_1(u(t)),y_2(u(t)))
\end{equation*}
is invertible and that its inverse is bounded in time. The differential of $F_t$ with respect to the first variable is also bounded (from its expression, $F_t$ being linear in $U$). Thus, by the mean-value theorem (given $(y_1,y_2)(P(t))=(0,0)$):
\begin{align}
\lvert y_1(u(t)) \rvert +\lvert y_2(u(t))\rvert \leq C\lVert u(t)-P(t) \rVert \leq C\lVert v(t)\rVert_{H^2}.
\end{align}


By applying the mean-value theorem (inequality) for $Q_{c_l}$ or $B_{\alpha_k,\beta_k}$ with respect to $y_1$ and $y_2$, we deduce that
\begin{align}
\lVert P_{j-1}(t)-\widetilde{P_{j-1}}(t)\rVert_{H^2}\leq C(\lvert y_1(t)\rvert +\lvert y_2(t)\rvert).
\end{align}

Finally, by triangular inequality,
\begin{align}
\lVert w(t)\rVert_{H^2} & \leq \lVert u(t)-P(t)\rVert_{H^2}+\lVert P(t)-\widetilde{P}(t)\rVert_{H^2}\\
& \leq \lVert u(t)-P(t)\rVert_{H^2}+C\big(\lvert y_1(t)\rvert +\lvert y_2(t)\rvert\big)\\
& \leq C\lVert v(t)\rVert_{H^2}.
\end{align}
This completes the proof of \eqref{1eq:-152}.

For \eqref{1eq:last}, we will take time derivatives of the equations \eqref{1eq:-150} and \eqref{1eq:-151}. Firstly, we may write the PDE verified by $w$:
\begin{align}
\partial_tw &=-w_{xxx}-\Bigg[w\bigg(w^2+3w\sum_{i=1}^J\widetilde{P_i}+3\sum_{i,m=1}^J\widetilde{P_i}\widetilde{P_m}\bigg)\Bigg]_x\\
& \quad-\sum_{h\neq i\ \text{or}\ i\neq m}\big(\widetilde{P_h}\widetilde{P_i}\widetilde{P_m}\big)_x-E,
\end{align}
where, if $P_{j-1}=B_k$ is a breather,
\begin{align}
E:=y_1'(t)\widetilde{B_k}_1+y_2'(t)\widetilde{B_k}_2,
\end{align}
and if $P_{j-1}=R_l$ is a soliton, denoting $y_{0,l}(t):=x-x_{0,l}^0+y_2(t)-c_lt$,
\begin{align}
E:=\frac{y_1'(t)}{2\big(c_l+y_1(t)\big)}\big(\widetilde{R_l}+y_{0,l}(t)\widetilde{R_l}_x\big)+y_2'(t)\widetilde{R_l}_x.
\end{align}

If $P_{j-1}=B_k$, we start by taking the time derivative of $\int\widetilde{B_k}_1w=0$ and perform some integrations by parts to obtain:
\begin{align}
& -\int \big(\widetilde{B_k}^3\big)_{1x}w+y_1'(t)\int \widetilde{B_k}_{11}w+y_2'(t)\int\widetilde{B_k}_{12}w\\
& +\int\widetilde{B_k}_{1x}w\bigg(w^2+3w\sum_{i=1}^J\widetilde{P_i}+3\sum_{h,i=1}^J\widetilde{P_h}\widetilde{P_i}\bigg)-\int\widetilde{B_k}_1\sum_{h\neq i\ \text{or}\ g\neq h}\big(\widetilde{P_h}\widetilde{P_i}\widetilde{P_g}\big)_x\\
& =y_1'(t)\int \widetilde{B_k}_1^2+y_2'(t)\int\widetilde{B_k}_1\widetilde{B_k}_2,
\end{align}
then, we take the time derivative of $\int\widetilde{B_k}_2w=0$:
\begin{align}
& -\int \big(\widetilde{B_k}^3\big)_{2x}w+y_1'(t)\int \widetilde{B_k}_{12}w+y_2'(t)\int\widetilde{B_k}_{22}w\\
& +\int\widetilde{B_k}_{2x}w\bigg(w^2+3w\sum_{i=1}^J\widetilde{P_i}+3\sum_{h,i=1}^J\widetilde{P_h}\widetilde{P_i}\bigg)-\int\widetilde{B_k}_2\sum_{h\neq i\ \text{or}\ g\neq h}\big(\widetilde{P_h}\widetilde{P_i}\widetilde{P_g}\big)_x\\
& =y_1'(t)\int \widetilde{B_k}_1\widetilde{B_k}_2+y_2'(t)\int \widetilde{B_k}_2^2.
\end{align}

If $P_{j-1}=R_l$, we start by taking the time derivative of $\int\widetilde{R_l}w=0$ and perform some integrations by parts to obtain:
\begin{align}
& -\int \big(\widetilde{R_l}^3\big)_xw+\frac{y_1'(t)}{2c_l}\int\big(\widetilde{R_l}+y_{0,l}(t)\widetilde{R_l}_x\big)w+y_2'(t)\int\widetilde{R_l}_xw\\
& +\int\widetilde{R_l}_xw\bigg(w^2+3w\sum_{i=1}^J\widetilde{P_i}+3\sum_{h,i=1}^J\widetilde{P_h}\widetilde{P_i}\bigg)-\int\widetilde{R_l}\sum_{h\neq i\ \text{or}\ g\neq h}\big(\widetilde{P_h}\widetilde{P_i}\widetilde{P_g}\big)_x\\
& =\frac{y_1'(t)}{2\big(c_l+y_1(t)\big)}\int\widetilde{R_l}\big(\widetilde{R_l}+y_{0,l}(t)\widetilde{R_l}_x\big)+y_2'(t)\int\widetilde{R_l}\widetilde{R_l}_x,
\end{align}
then, we take the time derivative of $\int\widetilde{R_l}_xw=0$:
\begin{align}
& -\int \big(\widetilde{R_l}^3\big)_{xx}w+\frac{y_1'(t)}{2c_l}\int\big(\widetilde{R_l}_x+y_{0,l}(t)\widetilde{R_l}_{xx}\big)w+y_2'(t)\int\widetilde{R_l}_{xx}w\\
& +\int\widetilde{R_l}_{xx}w\bigg(w^2+3w\sum_{i=1}^J\widetilde{P_i}+3\sum_{h,i=1}^J\widetilde{P_h}\widetilde{P_i}\bigg)-\int\widetilde{R_l}_x\sum_{h\neq i\ \text{or}\ g\neq h}\big(\widetilde{P_h}\widetilde{P_i}\widetilde{P_g}\big)_x\\
& =\frac{y_1'(t)}{2\big(c_l+y_1(t)\big)}\int\widetilde{R_l}_x\big(\widetilde{R_l}+y_{0,l}(t)\widetilde{R_l}_x\big)+y_2'(t)\int \big(\widetilde{R_l}_x\big)^2.
\end{align}

As a consequence of \eqref{1eq:-152}, we see that $\lvert y_1(t)\rvert+\lvert y_2(t)\rvert$ tends to $0$ when $t\to +\infty$. This is why, we may use Proposition \ref{1prop:decay-mod} and Corollary \ref{1cor:decay-mod} here, if $T_2$ is large enough. So, several terms of the four equalities above are obviously bounded by $( w(t)^2\Phi_j)^{1/2}$ or $e^{-\varpi t}$ for $\varpi>0$, a constant chosen small enough. Using these bounds, and after several linear combinations, we obtain \eqref{1eq:last}.
\end{proof}
\emph{Step 3.} Quadratic approximations of localized conservation
laws.
\begin{lem}
\label{1lem:quad}Let $\omega>0$ as small as we want. There exists
$C>0,T_{3}\geq  T$ such that the following holds for $t\geq  T_{3}$:
\begin{align}
\bigg\lvert M_{j}(t)-\sum_{i=1}^{j-1}M\big[\widetilde{P_{i}}\big]-\sum_{i=1}^{j-1}\int\widetilde{P_{i}}w-\frac{1}{2}\int w^{2}\Phi_{j}\bigg\rvert\leq  Ce^{-2\varpi t},\label{1eq:-155}
\end{align}
\begin{align}
\MoveEqLeft\bigg\lvert E_{j}(t)-\sum_{i=1}^{j-1}E\big[\widetilde{P_{i}}\big]-\sum_{i=1}^{j-1}\int\Big[\widetilde{P_{i}}_{x}w_{x}-\widetilde{P_{i}}^{3}w\Big]\\
& -\int\Big[\frac{1}{2}w_{x}^{2}-\frac{3}{2}\widetilde{P}^{2}w^{2}\Big]\Phi_{j}\bigg\rvert \leq  Ce^{-2\varpi t}+\omega\int w^{2}\Phi_{j},\label{1eq:-156}
\end{align}
\begin{align}
\MoveEqLeft\bigg\lvert F_{j}(t)-\sum_{i=1}^{j-1}F\big[\widetilde{P_{i}}\big]-\sum_{i=1}^{j-1}\int\Big[\widetilde{P_{i}}_{xx}w_{xx}-5\widetilde{P_{i}}\widetilde{P_{i}}_{x}^{2}w-5\widetilde{P_{i}}^{2}\widetilde{P_{i}}_{x}w_{x}+\frac{3}{2}\widetilde{P_{i}}^{5}w\Big] \\
& \quad-\int\Big[\frac{1}{2}w_{xx}^{2}
-\frac{5}{2}w^{2}\widetilde{P}_{x}^{2}-10\widetilde{P}w\widetilde{P}_{x}w_{x}-\frac{5}{2}\widetilde{P}^{2}w_{x}^{2}+\frac{15}{4}\widetilde{P}^{4}w^{2}\Big]\Phi_{j}(t)\bigg\rvert\\
& \leq  Ce^{-2\varpi t}+\omega\int\big(w^{2}+w_{x}^{2}\big)\Phi_{j}.\label{1eq:-157}
\end{align}
\end{lem}

\begin{proof}
\emph{For the mass:}

We compute:
\begin{align}
M_{j}(t) & =\frac{1}{2}\int\big(\widetilde{P}+w\big)^{2}\Phi_{j} \\
 & =\frac{1}{2}\int\widetilde{P}^{2}\Phi_{j}+\int\widetilde{P}w\Phi_{j}+\frac{1}{2}\int w^{2}\Phi_{j}.\label{1eq:-158}
\end{align}
As in Step 1, we can show that $\frac{1}{2}\int\widetilde{P}^{2}\Phi_{j}$
converges exponentially (we choose $\varpi$ with respect to this
exponential convergence) to $\sum_{i=1}^{j-1}M[\widetilde{P_{i}}]$.
Similarly, the difference between $\int\widetilde{P}w\Phi_{j}$ and
$\sum_{i=1}^{j-1}\widetilde{P_{i}}w$ converges exponentially to $0$
(the velocity of a soliton is not modified a lot by modulation, this
is why it works in any cases).

For $E$ and $F$, we perform similar basic computations with the only
difference that there will also be terms of degree 3 or more in $w$.
We know that $\lVert w(t)\rVert_{H^{2}}\rightarrow_{t\rightarrow+\infty}0$,
this is the reason why for $t$ large enough, such terms are boundable
by $\omega\int w^{2}\Phi_{j}$ or $\omega\int w_{x}^{2}\Phi_{j}$.
\end{proof}
\emph{Step 4.} Approximation of the Lyapunov functional. 

By analogy
with the existence part, we introduce the following Lyapunov functional:
\begin{align}
\mathcal{H}_{j}(t):=F_{j}(t)+2\big(b_{j-1}^{2}-a_{j-1}^{2}\big)E_{j}(t)+\big(a_{j-1}^{2}+b_{j-1}^{2}\big)^{2}M_{j}(t).\label{1eq:-159}
\end{align}
We will use the previous steps to approximate $\mathcal{H}_{j}(t)$.
\begin{lem}
\label{1lem:dl}There exists $T_{4}\geq  T$ such that the following
holds for $t\geq  T_{4}$:
\begin{align}
\mathcal{H}_{j}(t) & =\sum_{i=1}^{j-1}F[\widetilde{P_{i}}]+2\big(b_{j-1}^{2}-a_{j-1}^{2}\big)\sum_{i=1}^{j-1}E[\widetilde{P_{i}}]+\big(a_{j-1}^{2}+b_{j-1}^{2}\big)^{2}\sum_{i=1}^{j-1}M[\widetilde{P_{i}}] \\
 & +H_{j}(t)+O(e^{-2\varpi t})+o\bigg(\int\big(w^{2}+w_{x}^{2}\big)\Phi_{j}\bigg),\label{1eq:-160}
\end{align}
where 
\begin{align}
H_{j}(t): & =\int\Big[\frac{1}{2}w_{xx}^{2}-\frac{5}{2}w_{x}^{2}\widetilde{P_{j-1}}^{2}+\frac{5}{2}w^{2}\widetilde{P_{j-1}}_{x}^{2}+5w^{2}\widetilde{P_{j-1}}\widetilde{P_{j-1}}_{xx}\\
& \qquad+\frac{15}{4}w^{2}\widetilde{P_{j-1}}^{4}\Big]\Phi_{j}(t) 
  +\big(b_{j-1}^{2}-a_{j-1}^{2}\big)\int\Big[w_{x}^{2}-3w^{2}\widetilde{P_{j-1}}^{2}\Big]\Phi_{j}(t)\\
  &\quad+\frac{1}{2}\big(a_{j-1}^{2}+b_{j-1}^{2}\big)^{2}\int w^{2}\Phi_{j}(t).\label{1eq:-161}
\end{align}
\end{lem}

\begin{proof}
This lemma is obtained from the summation of the facts established in
the previous lemma. We get rid of the linear terms in the following
way, by integrations by parts:
\begin{align}
\MoveEqLeft\sum_{i=1}^{j-1}\int\Big(\widetilde{P_{i}}_{xx}w_{xx}-5\widetilde{P_{i}}\widetilde{P_{i}}_{x}^{2}w-5\widetilde{P_{i}}^{2}\widetilde{P_{i}}_{x}w_{x}+\frac{3}{2}\widetilde{P_{i}}^{5}w\Big) \\
 & +2\big(b_{j-1}^{2}-a_{j-1}^{2}\big)\sum_{i=1}^{j-1}\int\Big(\widetilde{P_{i}}_{x}w_{x}-\widetilde{P_{i}}^{3}w\Big)+\big(a_{j-1}^{2}+b_{j-1}^{2}\big)^{2}\sum_{i=1}^{j-1}\widetilde{P_{i}}w \\
 & =\sum_{i=1}^{j-1}\int\Big(\widetilde{P_{i}}_{xxxx}+5\widetilde{P_{i}}\widetilde{P_{i}}_{x}^{2}+5\int\widetilde{P_{i}}^{2}\widetilde{P_{i}}_{xx}+\frac{3}{2}\widetilde{P_{i}}^{5}\Big)w \\
 & +2\big(b_{j-1}^{2}-a_{j-1}^{2}\big)\sum_{i=1}^{j-1}\int\Big(-\widetilde{P_{i}}_{xx}-\widetilde{P_{i}}^{3}\Big)w+\big(a_{j-1}^{2}+b_{j-1}^{2}\big)^{2}\sum_{i=1}^{j-1}\int\widetilde{P_{i}}w.\label{1eq:-162}
\end{align}

If we consider that this sum goes from $i=1$ to $j-2$, we see that
for $1\leq  i\leq  j-2$, this sum is exponentially bounded by induction
assumption (we use that for $i\leq  j-2$, a polynomial in $\widetilde{P_{i}}$
and its derivatives is bounded by $C\Phi_{j-1}$ and that $w=v+(P_{j-1}-\widetilde{P_{j-1}})$).
It is left to consider the sum of the terms with $i=j-1$.

For $i=j-1$, we have nearly the elliptic equation satisfied by $\widetilde{P_{j-1}}$.
It is actually exactly this equation in the case when $\widetilde{P_{j-1}}$
is a breather. When $\widetilde{P_{j-1}}$ is a soliton, its shape
parameter is modified by modulation. This is why, in this case, the
sum of the terms with $i=j-1$ is equal to 
\begin{align}
2y_{1}(t)\int\Big(-\widetilde{P_{j-1}}_{xx}-\widetilde{P_{j-1}}^{3}\Big)w+2b_{j-1}^{2}y_{1}(t)\int\widetilde{P_{j-1}}w+y_{1}(t)^{2}\int\widetilde{P_{j-1}}w,\label{1eq:-163}
\end{align}
which vanishes because of the orthogonality condition from the modulation (Lemma \ref{2lem:mod_uniq})
and the elliptic equation satisfied by a soliton \eqref{1eq:-4}.

$H_{j}$ is obtained as the sum of the quadratic parts of the previous
lemma on which we have performed some integrations by parts, and some simplifications based on
the fact that for $i\geq  j$, $\widetilde{P_{i}}\Phi_{j}(t)$ is
exponentially decreasing, and the fact that for $i\leq  j-2$, $\int\widetilde{P_{i}}w^{2}$
is exponentially decreasing by the induction assumption \eqref{1eq:-123}. Therefore, $H_{j}$
corresponds to the sum of the quadratic parts of previous lemma to
which we have to add $5\int w^{2}\widetilde{P}\widetilde{P}_{x}\Phi_{jx}$,
which is bounded exponentially.
\end{proof}
\emph{Step 5.} Bound from above for $H_{j}(t)$. 

Because $v_1>0$, we have
that $b_{j-1}^{2}-a_{j-1}^{2}\geq 0$. By taking $\omega_{2}$ and $\omega_{6}$
small enough (with respect to $(a_{j-1}^{2}+b_{j-1}^{2})^{2}$),
we obtain, by summation of the facts of Lemma \ref{1lem:monoto}, the
following inequality:
\begin{align}
\mathcal{H}_{j}(t)-\sum_{i=1}^{j-1}F[P_{i}]-2\big(b_{j-1}^{2}-a_{j-1}^{2}\big)\sum_{i=1}^{j-1}E[P_{i}]\\
-\big(a_{j-1}^{2}+b_{j-1}^{2}\big)^{2}\sum_{i=1}^{j-1}M[P_{i}] &\leq  Ce^{-2\varpi t}.\label{1eq:-164}
\end{align}

From Lemma \ref{1lem:dl}, for $t\geq  T_{3}$,
\begin{align}
H_{j}(t) & \leq  F[P_{j-1}]-F[\widetilde{P_{j-1}}]+2\big(b_{j-1}^{2}-a_{j-1}^{2}\big)\Big(E[P_{j-1}]-E[\widetilde{P_{j-1}}]\Big) \\
 & +\big(a_{j-1}^{2}+b_{j-1}^{2}\big)^{2}\Big(M[P_{j-1}]-M[\widetilde{P_{j-1}}]\Big)+Ce^{-2\varpi t}+\omega\int\big(w^{2}+w_{x}^{2}\big)\Phi_{j}.\label{1eq:-165}
\end{align}

In the case if $P_{j-1}$ is a breather, we obtain immediately that
\begin{align}
H_{j}(t)\leq  Ce^{-2\varpi t}+\omega\int\big(w^{2}+w_{x}^{2}\big)\Phi_{j}.\label{1eq:-166}
\end{align}

The case when $P_{j-1}$ is a soliton needs more inspection. As in
the existence part, we have the following relations:
\begin{align}
M[\widetilde{P_{j-1}}](t)=\big(b_{j-1}^{2}+y_{1}(t)\big)^{1/2}M[q],\label{1eq:-167}
\end{align}
\begin{align}
E[\widetilde{P_{j-1}}](t)=\big(b_{j-1}^{2}+y_{1}(t)\big)^{3/2}E[q],\label{1eq:-168}
\end{align}
\begin{align}
F[\widetilde{P_{j-1}}](t)=\big(b_{j-1}^{2}+y_{1}(t)\big)^{5/2}F[q].\label{1eq:-169}
\end{align}
We set $\mathcal{R}_{j-1}(t):=F[\widetilde{P_{j-1}}](t)+2b_{j-1}^{2}E[\widetilde{P_{j-1}}](t)+b_{j-1}^{4}M[\widetilde{P_{j-1}}](t)$,
and  we simplify it as follows:
\begin{align}
\mathcal{R}_{j-1}(t)&=b_{j-1}^{5}\bigg(1+\frac{y_{1}(t)}{b_{j-1}^{2}}\bigg)^{5/2}F[q]+2b_{j-1}^{5}\bigg(1+\frac{y_{1}(t)}{b_{j-1}^{2}}\bigg)^{3/2}E[q]\\
&\quad+b_{j-1}^{5}\bigg(1+\frac{y_{1}(t)}{b_{j-1}^{2}}\bigg)^{1/2}M[q].\label{1eq:-170}
\end{align}
After making a Taylor expansion as in Section \ref{1sec:existence},
\begin{align}
\mathcal{R}_{j-1}(t)-F[P_{j-1}]-2b_{j-1}^{2}E[P_{j-1}]-b_{j-1}^{4}M[P_{j-1}]=O(y_{1}(t)^{3}).\label{1eq:-171}
\end{align}
Therefore, if $T_{4}$ is large enough, $\lVert v(t)\rVert_{H^{2}}$ can
be as small as we want, and for $t\geq  T_{4}$, if $P_{j-1}$ a soliton,
we may write:
\begin{align}
H_{j}(t)\leq  Ce^{-2\varpi t}+\omega\int\big(w^{2}+w_{x}^{2}\big)\Phi_{j}+\omega y_{1}(t)^{2}.\label{1eq:-172}
\end{align}

\emph{Step 6.} Coercivity. 

$H_{j}$ can be seen as the
quadratic form associated to $\widetilde{P_{j-1}}$ and evaluated
in $w\sqrt{\Phi_{j}}$, modulo several terms that can be bounded by $C\sqrt{\sigma}\int(w^{2}+w_{x}^{2}+w_{xx}^{2})\Phi_{j}$
(because these terms depend on derivatives of $\Phi_{j}$). Let us
prove that we can apply Section \ref{1sec:54} (Appendix)
for $w\sqrt{\Phi_{j}}$. 

More precisely, we need to prove that for $\nu>0$ small enough (from Section \ref{1sec:54}),
\begin{align}
\bigg\lvert \int  w\sqrt{\Phi_{j}} \widetilde{P_{j-1}}_1 \bigg\rvert + \bigg\lvert \int w\sqrt{\Phi_{j}} \widetilde{P_{j-1}}_2 \bigg\rvert \leq \nu \lVert w\sqrt{\Phi_{j}} \rVert_{H^2},
\end{align}
if $P_{j-1}$ is a breather or that 
\begin{align}
\bigg\lvert \int  w\sqrt{\Phi_{j}} \widetilde{P_{j-1}} \bigg\rvert + \bigg\lvert \int w\sqrt{\Phi_{j}} \widetilde{P_{j-1}}_x \bigg\rvert \leq \nu \lVert w\sqrt{\Phi_{j}} \rVert_{H^2} ,
\end{align}
 if $P_{j-1}$ is a soliton. In any case, the proof is the same and let us write $K$ at the place of $\widetilde{P_{j-1}}_1$, $\widetilde{P_{j-1}}_2$, $\widetilde{P_{j-1}}$ or $\widetilde{P_{j-1}}_x$. This means that we want to bound $\int w\sqrt{\Phi_{j}} K $.

From \eqref{1eq:-150}, \eqref{1eq:-151}, we see that it is enough to bound $\int w(1-\sqrt{\Phi_{j}}) K $ by $\nu \lVert w\sqrt{\Phi_{j}} \rVert_{H^2} $. The reasonning that follows works  for $j \leq J$, for $j=J+1$ the result is immediate because $\Phi_{J+1}=1$. $\Phi_j$ is a translate of $\Psi$, and, using the fact that when $v \rightarrow 0$, $\sqrt{1+v}=1+O(v)$,
\begin{align}
1-\sqrt{\Psi} & = 1 -  \sqrt{1+\Psi -1}  = 1- \sqrt{1-\Psi(-x)} = O(\Psi(-x)),
\end{align}
which means that $1-\sqrt{\Phi_j} \leq C \min(1,\exp(\sqrt{\sigma}(x-m_j t)/2))$. We may deduce now that
\begin{align}
\bigg\lvert \int w(1-\sqrt{\Phi_{j}})K \bigg\rvert & = \bigg\lvert \int w\sqrt{\Phi_{j}} \frac{1-\sqrt{\Phi_j}}{\sqrt{\Phi_j}} K \bigg\rvert  \\
& \leq \bigg\lVert \frac{1-\sqrt{\Phi_j}}{\sqrt{\Phi_j}} K \bigg\rVert_{L^2} \lVert w\sqrt{\Phi_{j}} \rVert_{L^2}  \\
& \leq C e^{\sqrt{\sigma}(m_j - v_{j-1})t} \lVert w\sqrt{\Phi_{j}} \rVert_{L^2},
\end{align}
if $\frac{\sqrt{\sigma}}{4} < \frac{\beta}{2} $. And so, if $t$ is large enough, we get the bound we want. 

Thus, there exists $\mu>0$ such that for
$t\geq  T_{5}$ (where $T_{5}$ is large enough and depends on $\sigma$),
\begin{align}
\mu\lVert w\sqrt{\Phi_{j}}\rVert_{H^{2}}^{2} & \leq  H_{j}(t)+C\sqrt{\sigma}\int\big(w^{2}+w_{x}^{2}+w_{xx}^{2}\big)\Phi_{j}+\frac{1}{\mu}\bigg(\int\widetilde{P_{j-1}}w\sqrt{\Phi_{j}}\bigg)^{2} \\
 & \leq  Ce^{-2\varpi t}+\omega\int\big(w^{2}+w_{x}^{2}\big)\Phi_{j}+C\sqrt{\sigma}\int\big(w^{2}+w_{x}^{2}+w_{xx}^{2}\big)\Phi_{j} \\
 & +\omega y_{1}(t)^{2}+\frac{1}{\mu}\bigg(\int\widetilde{P_{j-1}}w\sqrt{\Phi_{j}}\bigg)^{2},\label{1eq:-173}
\end{align}
where the term $\frac{1}{\mu}\big(\int\widetilde{P_{j-1}}w\sqrt{\Phi_{j}}\big)^{2}$
is present only if $\widetilde{P_{j-1}}$ is a breather and the term
$\omega y_{1}(t)^{2}$ is present only if $\widetilde{P_{j-1}}$ is
a soliton.

For $\sigma$ small enough and $\omega$ small enough, we deduce that
\begin{align}
\int\big(w^{2}+w_{x}^{2}+w_{xx}^{2}\big)\Phi_{j}\leq  Ce^{-2\varpi t}+\omega y_{1}(t)^{2}+C\bigg(\int\widetilde{P_{j-1}}w\sqrt{\Phi_{j}}\bigg)^{2}.\label{1eq:step6}
\end{align}

We set $T_{0}:=\max(T_{1},T_{2},T_{3},T_{4},T_{5})$.

\emph{Step 7.} Bound for $\big\lvert\int\widetilde{P_{j-1}}w\sqrt{\Phi_{j}}\big\rvert$
(to do in the case if $\widetilde{P_{j-1}}$ is a breather). 

We would
like to prove that $\int\widetilde{P_{j-1}}w\sqrt{\Phi_{j}}$ is exponentially
decreasing. To do so, we would like to get rid of $\sqrt{\Phi_{j}}$. 
It is clear that $\int\widetilde{P_{j-1}}w(1-\sqrt{\Phi_{j}})$
is exponentially decreasing. Thus, it is enough to prove
that $\int\widetilde{P_{j-1}}w$ is exponentially decreasing.

If $i\leq  j-2$, we know that $\int\widetilde{P_{i}}w$ is exponentially
decreasing by the induction assumption \eqref{1eq:-123}. Thus, it is enough
to prove that $\sum_{i=1}^{j-1}\int\widetilde{P_{i}}w$ is exponentially
decreasing.

From the mass approximation of Lemma \ref{1lem:quad} and Lemma \ref{1lem:monoto},
we observe that, for $t\geq  T_{0}$,
\begin{align}
\sum_{i=1}^{j-1}\int\widetilde{P_{i}}w & =O(e^{-2\varpi t})+M_{j}(t)-\sum_{i=1}^{j-1}M[P_{i}]-\frac{1}{2}\int w^{2}\Phi_{j} \\
 & \leq  Ce^{-2\varpi t}-\frac{1}{2}\int w^{2}\Phi_{j} \leq  Ce^{-2\varpi t}.\label{1eq:-174}
\end{align}

Now, we use the fact that the sum of the linear parts of our localized
conservation laws is exponentially decreasing, which we have established
in the proof of Lemma \ref{1lem:dl}. Therefore, the linear terms of $F_{j}+2(b_{j-1}^{2}-a_{j-1}^{2})E_{j}$
are equal to $O(e^{-2\varpi t})-(a_{j-1}^{2}+b_{j-1}^{2})^{2}\sum_{i=1}^{j-1}\int\widetilde{P_{i}}w$.

Now, from the energy and $F$ approximation of Lemma \ref{1lem:quad}
and Lemma \ref{1lem:monoto}, and from \eqref{1eq:-171}, we observe that (we recall that $b_{j-1}^{2}-a_{j-1}^{2}\geq 0$),
for $t\geq  T_{0}$,
\begin{align}
\MoveEqLeft-\big(a_{j-1}^{2}+b_{j-1}^{2}\big)^{2}\sum_{i=1}^{j-1}\int\widetilde{P_{i}}w  =O(e^{-2\varpi t})+o\bigg(\int\big(w^{2}+w_{x}^{2}\big)\Phi_{j}\bigg)+F_{j}(t) \\
 & \quad+2\big(b_{j-1}^{2}-a_{j-1}^{2}\big)E_{j}(t)-\sum_{i=1}^{j-1}F[P_{i}]-2\big(b_{j-1}^{2}-a_{j-1}^{2}\big)\sum_{i=1}^{j-1}E[P_{i}] \\
 & \quad-\int\Big[\frac{1}{2}w_{xx}^{2}-\frac{5}{2}w^{2}\widetilde{P}_{x}^{2}-10\widetilde{P}w\widetilde{P}_{x}w_{x}-\frac{5}{2}\widetilde{P}^{2}w_{x}^{2}+\frac{15}{4}\widetilde{P}^{4}w^{2}\Big]\Phi_{j} \\
 & \quad-2\big(b_{j-1}^{2}-a_{j-1}^{2}\big)\int\Big[\frac{1}{2}w_{x}^{2}-\frac{3}{2}\widetilde{P}^{2}w^{2}\Big]\Phi_{j}+o(y_1(t)^2) \\
 & =O(e^{-2\varpi t})+o\bigg(\int\big(w^{2}+w_{x}^{2}\big)\Phi_{j}\bigg) \\
 & \quad+F_{j}(t)+\omega_{6}M_{j}(t)-\sum_{i=1}^{j-1}F[P_{i}]-\omega_{6}\sum_{i=1}^{j-1}M[P_{i}] \\
 & \quad+2\big(b_{j-1}^{2}-a_{j-1}^{2}\big)\Bigg[E_{j}(t)+\omega_{2}M_{j}(t)-\sum_{i=1}^{j-1}E[P_{i}]-\omega_{2}\sum_{i=1}^{j-1}M[P_{i}]\Bigg] \\
 & \quad+\Big(\omega_{6}+2\omega_{2}\big(b_{j-1}^{2}-a_{j-1}^{2}\big)\Big)\Bigg(\sum_{i=1}^{j-1}M[P_{i}]-M_{j}(t)\Bigg) \\
 & \quad-\int\Big[\frac{1}{2}w_{xx}^{2}-\frac{5}{2}w^{2}\widetilde{P}_{x}^{2}-10\widetilde{P}w\widetilde{P}_{x}w_{x}-\frac{5}{2}\widetilde{P}^{2}w_{x}^{2}+\frac{15}{4}\widetilde{P}^{4}w^{2}\Big]\Phi_{j} \\
 & \quad-2\big(b_{j-1}^{2}-a_{j-1}^{2}\big)\int\Big[\frac{1}{2}w_{x}^{2}-\frac{3}{2}\widetilde{P}^{2}w^{2}\Big]\Phi_{j}+o(y_1(t)^2) \\
 & \leq  Ce^{-2\varpi t}+C\int\big(w^{2}+w_{x}^{2}\big)\Phi_{j}+o(y_1(t)^2) \\
 & \quad-\Big(\omega_{6}+2\omega_{2}\big(b_{j-1}^{2}-a_{j-1}^{2}\big)\Big)\Bigg(\sum_{i=1}^{j-1}\int\widetilde{P_{i}}w+\frac{1}{2}\int w^{2}\Phi_{j}\Bigg),\label{1eq:-175}
\end{align}
where the term $o(y_1(t)^2)$ is present only if $P_{j-1}$ is a soliton. And therefore, for $\omega_{2}$ and $\omega_{6}$ small enough,
\begin{align}
-\sum_{i=1}^{j-1}\int\widetilde{P_{i}}w\leq  Ce^{-2\varpi t}+C\int\big(w^{2}+w_{x}^{2}\big)\Phi_{j}+o(y_1(t)^2).\label{1eq:-176}
\end{align}

Thus, we deduce the following bound:
\begin{align}
\bigg\lvert\int\widetilde{P_{j-1}}w\sqrt{\Phi_{j}}\bigg\rvert\leq  Ce^{-2\varpi t}+C\int\big(w^{2}+w_{x}^{2}\big)\Phi_{j}+o(y_1(t)^2).\label{1eq:-177}
\end{align}
Because $\lVert w(t)\rVert_{H^{2}}\rightarrow_{t\rightarrow+\infty}0$,
we deduce that
\begin{align}
\bigg(\int\widetilde{P_{j-1}}w\sqrt{\Phi_{j}}\bigg)^{2}=o(e^{-2\varpi t})+o\bigg(\int\big(w^{2}+w_{x}^{2}\big)\Phi_{j}\bigg)+o(y_1(t)^2).\label{1eq:step7}
\end{align}

\emph{Step 8.} Conclusion. 

From (\ref{1eq:step6}) and (\ref{1eq:step7}),
we deduce for $t\geq  T_{0}$, that
\begin{align}
\int\big(w^{2}+w_{x}^{2}+w_{xx}^{2}\big)\Phi_{j}=O(e^{-2\varpi t})+o\big(y_{1}(t)^{2}\big)+o\bigg(\int\big(w^{2}+w_{x}^{2}\big)\Phi_{j}\bigg).\label{1eq:-178}
\end{align}
This means that if we take $T_{0}$ large enough, we have:
\begin{align}
\int\big(w^{2}+w_{x}^{2}+w_{xx}^{2}\big)\Phi_{j}=o\big(y_{1}(t)^{2}\big)+O(e^{-2\varpi t}),\label{1eq:quad}
\end{align}
where the term $o(y_{1}(t)^{2})$ is present only if $P_{j-1}$
is a soliton.

Before finishing the proof, we need to find a better bound for $y_{1}(t)$
than just a convergence to $0$ given by the modulation (in the case
when $P_{j-1}$ is a soliton). For this, we study $M_{j}(t)$:
\begin{align}
M_{j}(t) & =\frac{1}{2}\int u^{2}(t)\Phi_{j}(t) =\frac{1}{2}\int\big(\widetilde{P}(t)+w(t)\big)^{2}\Phi_{j}(t) \\
 & =\frac{1}{2}\int\widetilde{P}(t)^{2}\Phi_{j}(t)+\int\widetilde{P}(t)w(t)\Phi_{j}(t)+\frac{1}{2}\int w(t)^{2}\Phi_{j}(t) \\
 & =\frac{1}{2}\sum_{i=1}^{j-1}\int\widetilde{P_{i}}(t)^{2}+\sum_{i=1}^{j-1}\int\widetilde{P_{i}}(t)w(t)+O\big(e^{-2\varpi t}\big)+\frac{1}{2}\int w(t)^{2}\Phi_{j}(t) \\
 & =\frac{1}{2}\int\widetilde{P_{j-1}}(t)^{2}+\int\widetilde{P_{j-1}}(t)w(t)+O\big(e^{-2\varpi t}\big)\\
 &\quad+\frac{1}{2}\int w(t)^{2}\Phi_{j}(t)+\frac{1}{2}\sum_{i=1}^{j-2}\int P_{i}(t)^{2},\label{1eq:-179}
\end{align}
by the induction assumption \eqref{1eq:-123}, then
\begin{align}
M_{j}(t)=\frac{1}{2}\int\widetilde{P_{j-1}}(t)^{2}+O\big(e^{-2\varpi t}\big)+\frac{1}{2}\int w(t)^{2}\Phi_{j}(t)+\frac{1}{2}\sum_{i=1}^{j-2}\int P_{i}(t)^{2},\label{1eq:-180}
\end{align}
by the orthogonality condition from the modulation (Lemma \ref{2lem:mod_uniq}). Therefore,
\begin{align}
M_{j}(t)&=\big(b_{j-1}^{2}+y_{1}(t)\big)^{1/2}M[q]+O\big(e^{-2\varpi t}\big)+\frac{1}{2}\int w(t)^{2}\Phi_{j}(t)\\
&\quad+\frac{1}{2}\sum_{i=1}^{j-2}\int P_{i}(t)^{2}.\label{1eq:-181}
\end{align}

Now, if we take $t_{1}\geq  t$, we obtain from (\ref{1eq:quad}) that
\begin{align}
M_{j}(t_{1})-M_{j}(t)&=\Big[\big(b_{j-1}^{2}+y_{1}(t_{1})\big)^{1/2}-\big(b_{j-1}^{2}+y_{1}(t)\big)^{1/2}\Big]M[q]\\
&\quad+O\big(e^{-2\varpi t}\big)+o\big(y_{1}(t)^{2}\big)+o\big(y_{1}(t_{1})^{2}\big).\label{1eq:inter}
\end{align}

By doing a Taylor expansion of order 1, as in the existence part,
we obtain:
\begin{align}
\big(b_{j-1}^{2}+y_{1}(t_{1})\big)^{1/2}=b_{j-1}\bigg(1+\frac{1}{2}\frac{y_{1}(t_{1})}{b_{j-1}^{2}}+O\big(y_{1}(t_{1})^{2}\big)\bigg).\label{1eq:-182}
\end{align}
Therefore,
\begin{align}
\MoveEqLeft\big(b_{j-1}^{2}+y_{1}(t_{1})\big)^{1/2}-\big(b_{j-1}^{2}+y_{1}(t)\big)^{1/2}\\
&=\frac{1}{2b_{j-1}}\big(y_{1}(t_{1})-y_{1}(t)\big)+O\big(y_{1}(t_{1})^{2}\big)+O\big(y_{1}(t)^{2}\big).\label{1eq:-183}
\end{align}
Now, we recall that when $t_{1}\rightarrow+\infty$, we have $y_{1}(t_{1})\rightarrow0$.
Therefore, by taking the limit of the previous formula when $t_{1}\rightarrow+\infty$,
we obtain:
\begin{align}
b_{j-1}-\big(b_{j-1}^{2}+y_{1}(t)\big)^{1/2}=-\frac{y_{1}(t)}{2b_{j-1}}+O\big(y_{1}(t)^{2}\big).\label{1eq:-184}
\end{align}
Therefore, from (\ref{1eq:inter}), with $t_{1}\rightarrow+\infty$,
\begin{align}
\sum_{i=1}^{j-1}M[P_{i}]-M_{j}(t)=-\frac{y_{1}(t)}{2b_{j-1}}M[q]+O\big(e^{-2\varpi t}\big)+O\big(y_{1}(t)^{2}\big).\label{1eq:mass}
\end{align}

The second step is to study $E_{j}(t)$ (we do the same reasonning
as for $M_{j}$):
\begin{align}
E_{j}(t) & =\int\Big[\frac{1}{2}u_{x}^{2}-\frac{1}{4}u^{4}\Big]\Phi_{j}(t) \\
 & =\int\Big[\frac{1}{2}\widetilde{P}_{x}^{2}-\frac{1}{4}\widetilde{P}^{4}\Big]\Phi_{j}(t)+\int\Big[\widetilde{P}_{x}w_{x}-\widetilde{P}^{3}w\Big]\Phi_{j}(t)+O\bigg(\int w^{2}\Phi_{j}(t)\bigg),\label{1eq:-185}
\end{align}
and after simplications by $\Phi_{j}$ due to exponential convergences,
induction assumption \eqref{1eq:-123} and orthogonality conditions (Lemma \ref{2lem:mod_uniq}),
\begin{align}
E_{j}(t) & =E[\widetilde{P_{j-1}}(t)]+\sum_{i=1}^{j-2}E[P_{i}]+O\big(e^{-2\varpi t}\big)+O\bigg(\int w^{2}\Phi_{j}(t)\bigg) \\
 & =\big(b_{j-1}^{2}+y_{1}(t)\big)^{3/2}E[q]+\sum_{i=1}^{j-2}E[P_{i}]+O\big(e^{-2\varpi t}\big)+O\bigg(\int w^{2}\Phi_{j}(t)\bigg) \\
 & =\big(b_{j-1}^{2}+y_{1}(t)\big)^{3/2}E[q]+\sum_{i=1}^{j-2}E[P_{i}]+O\big(e^{-2\varpi t}\big)+o\big(y_{1}(t)^{2}\big),\label{1eq:-186}
\end{align}
by (\ref{1eq:quad}). And then, by taking the difference for $t_{1}\geq  t$,
\begin{align}
E_{j}(t_{1})-E_{j}(t)&=\Big[\big(b_{j-1}^{2}+y_{1}(t_{1})\big)^{3/2}-\big(b_{j-1}^{2}+y_{1}(t)\big)^{3/2}\Big]E[q]\\
&\quad+O\big(e^{-2\varpi t}\big)+o\big(y_{1}(t_{1})^{2}\big)+o\big(y_{1}(t)^{2}\big).\label{1eq:-187}
\end{align}

By taking a Taylor expansion of order 1, we obtain:
\begin{align}
\big(b_{j-1}^{2}+y_{1}(t_{1})\big)^{3/2}=b_{j-1}^{3}\Big(1+\frac{3}{2}\frac{y_{1}(t_{1})}{b_{j-1}^{2}}+O\big(y_{1}(t_{1})^{2}\big)\Big).\label{1eq:-188}
\end{align}
Therefore, after taking $t_{1}\rightarrow+\infty$, we obtain:
\begin{align}
\sum_{i=1}^{j-1}E[P_{i}]-E_{j}(t)=-\frac{3}{2}b_{j-1}y_{1}(t)E[q]+O\big(e^{-2\varpi t}\big)+O\big(y_{1}(t)^{2}\big).\label{1eq:energ}
\end{align}

This is why, from (\ref{1eq:mass}), (\ref{1eq:energ}) and Lemma \ref{1lem:monoto},
we obtain:
\begin{align}
-\frac{y_{1}(t)}{2b_{j-1}}M[q]+O\big(e^{-2\varpi t}\big)+O\big(y_{1}(t)^{2}\big)\geq -Ce^{-2\varpi t},\label{1eq:-189}
\end{align}
and
\begin{align}
-\frac{3}{2}b_{j-1}y_{1}(t)E[q]+O\big(e^{-2\varpi t}\big)+O\big(y_{1}(t)^{2}\big)\geq -Ce^{-2\varpi t}.\label{1eq:-190}
\end{align}

Because $M[q]=2$ and $E[q]=-\frac{2}{3}$, we rewrite both previous
inequalities (\ref{1eq:-189}) and (\ref{1eq:-190}) in the following way (and we pass $O(e^{-2\varpi t})$
on the other side of each inequality):
\begin{align}
-\frac{y_{1}(t)}{b_{j-1}}+O\big(y_{1}(t)^{2}\big)\geq -Ce^{-2\varpi t},\label{1eq:-191}
\end{align}
and
\begin{align}
b_{j-1}y_{1}(t)+O\big(y_{1}(t)^{2}\big)\geq -Ce^{-2\varpi t}.\label{1eq:-192}
\end{align}

Because $y_{1}(t)\rightarrow+\infty$, by taking $T_{0}$ larger if
needed, $O(y_{1}(t)^{2})$ can be bounded above by any positive
constant multiplied by $\lvert y_{1}(t)\rvert$, so by taking this constant
small enough (by taking $T_{0}$ large enough) and combining both
previous inequalities (\ref{1eq:-191}) and (\ref{1eq:-192}), we obtain:
\begin{align}
\lvert y_{1}(t)\rvert\leq  Ce^{-2\varpi t}.\label{1eq:-193}
\end{align}
Therefore, we have obtained a better bound for $y_{1}(t)$ in the case
when $P_{j-1}$ is a soliton. Therefore, we may conclude that in any
case, for $t\geq  T_{0}$, for $T_{0}$ large enough, 
\begin{align}
\int\big(w^{2}+w_{x}^{2}+w_{xx}^{2}\big)\Phi_{j}(t)=O\big(e^{-2\varpi t}\big).\label{1eq:-194}
\end{align}

Then, we deduce from (\ref{1eq:last}) that
\begin{align}
\lvert y_{1}'(t)\rvert+\lvert y_{2}'(t)\rvert=O(e^{-\varpi t}).\label{1eq:-195}
\end{align}
Because $\lvert y_{1}(t)\rvert+\lvert y_{2}(t)\rvert\rightarrow_{t\rightarrow+\infty}0$,
we obtain by integration:
\begin{align}
\lvert y_{1}(t)\rvert+\lvert y_{2}(t)\rvert=O(e^{-\varpi t}).\label{1eq:-197}
\end{align}
And, so, by the mean-value theorem,
\begin{align}
\big\lVert \widetilde{P_{j-1}}-P_{j-1}\big\rVert _{H^{2}}\leq  C\big(\lvert y_{1}(t)\rvert+\lvert y_{2}(t)\rvert\big)\leq  Ce^{-\varpi t}.\label{1eq:-199}
\end{align}

From $v=w+\widetilde{P_{j-1}}-P_{j-1}$, we deduce:
\begin{align}
\MoveEqLeft\int\big(v^{2}+v_{x}^{2}+v_{xx}^{2}\big)\Phi_{j}  \leq  C\int\big(w^{2}+w_{x}^{2}+w_{xx}^{2}\big)\Phi_{j} \\
 & \quad+C\int\Big[\big(\widetilde{P_{j-1}}-P_{j-1}\big)^{2}+\big(\widetilde{P_{j-1}}-P_{j-1}\big)_{x}^{2}+\big(\widetilde{P_{j-1}}-P_{j-1}\big)_{xx}^{2}\Big]\Phi_{j} \\
 & \leq  Ce^{-2\varpi t},\label{1eq:-201}
\end{align}
and this finishes the induction.
\end{proof}

\subsection{Proof of Theorem \ref{1thm:uniqueness}}
\label{1sec:uniq_ccl}
\begin{proof}[Proof of Theorem \ref{1thm:uniqueness}]
We suppose that $v_1 >0$. Let $p$ be the associated multi-breather given by Theorem \ref{1thm:MAIN}. Let $u$ be a solution of \eqref{1mKdV} such that
\begin{align}
\lVert u(t)-p(t) \rVert_{H^2} \rightarrow_{t \rightarrow +\infty} 0.
\end{align}

From Proposition \ref{1prop:conv_exp}, we deduce that there exists a constant $C>0$ and a constant $\varpi >0$ such that for $t$ large enough
\begin{align}
\lVert u(t)-p(t) \rVert_{H^2} \leq C e^{-\varpi t}.
\end{align}

This implies that $u$ satisfies the assumptions of Proposition \ref{1lem:polyn}. Thus, $u=p$ and Theorem \ref{1thm:uniqueness} is proved.

\end{proof}

\section{Appendix}

The first two subsections of the Appendix show that a soliton has
similar properties as a ``limit breather'' of parameter $\alpha=0$.
Firstly, the corresponding elliptic equation is satisfied by a soliton.
Secondly, the corresponding quadratic form is coercive for a soliton,
and we see that its kernel is spanned by $\partial_{x}Q$ and $\partial_{c}Q$.
In the third subsection, we prove that it is possible for $\epsilon$
to be orthogonal to $Q$ and $\partial_{x}Q$ (instead of $\partial_{x}Q$
and $\partial_{c}Q$) in order to satisfy a coercivity for the quadratic
form. We will use this fact for the proof of the existence, as well
as for the first part of the proof of the uniqueness. In the fourth
subsection, we prove that we can have coercivity for quadratic forms
when the orthogonality condition is not exactly satisfied. We will
use this result for the proof of the uniqueness. The last subsection
is about computations for the third conservation law. It will be useful
for the monotonicity property for localized $F$ that we need in the
proof of the uniqueness.

\subsection{Elliptic equation satisfied by a soliton }
\label{1sec:51}
\begin{lem}
A soliton $Q=R_{c,\kappa}$ satisfies for any time $t\in\mathbb{R}$,
the following nonlinear elliptic equation: 
\begin{align}
Q_{(4x)}-2c(Q_{xx}+Q^{3})+c^{2}Q+5QQ_{x}^{2}+5Q^{2}Q_{xx}+\frac{3}{2}Q^{5}=0.\label{1eq:ellip}
\end{align}
\end{lem}

\begin{proof}
In order to derive this equation, we will use the equation that defines
a soliton (and that is satisfied by $Q$ at any time): 
\begin{align}
Q_{xx}=cQ-Q^{3}.\label{1eq:-271}
\end{align}
We will also need the following equation: 
\begin{align}
Q_{x}^{2}=cQ^{2}-\frac{1}{2}Q^{4},\label{1eq:-272}
\end{align}
that can be derived by taking the space derivative of $Q_{x}^{2}-cQ^{2}+\frac{1}{2}Q^{4}$,
and by showing that this derivative is zero. From this, we deduce
that $Q_{x}^{2}-cQ^{2}+\frac{1}{2}Q^{4}$ is constant, and by taking
its limit when $x\rightarrow\pm\infty$, we see that this constant
is zero. More precisely, the derivative of $Q_{x}^{2}-cQ^{2}+\frac{1}{2}Q^{4}$
is: 
\begin{align}
2Q_{x}Q_{xx}-2cQQ_{x}+2Q^{3}Q_{x}=2Q_{x}(Q_{xx}-cQ+Q^{3})=0.\label{1eq:-273}
\end{align}

From now on, the derivation of (\ref{1eq:ellip}) is straightforward.
It is sufficient to take space derivatives of $Q_{xx}=cQ-Q^{3}$ and
to inject them into the right hand side of the equation (\ref{1eq:ellip}),
which we want to prove to be equal to zero. By doing this, we make
the maximal order of a derivative of $Q$ present in the right hand
side equation lower. At the end, we have only, zero and first order
derivatives. To have only a polynomial in $Q$, we have to use $Q_{x}^{2}=cQ^{2}-\frac{1}{2}Q^{4}$,
and the calculations show that this polynomial is zero. 
\end{proof}

\subsection{Study of coercivity of the quadratic form associated to a soliton}
\label{1coer_sol}

In this article, we adapt the argument for the breathers in \cite{key-1}
to the soliton case. We consider: 
\begin{align}
\mathcal{Q}_{c}^{s}[\epsilon] & :=\frac{1}{2}\int\epsilon_{xx}^{2}-\frac{5}{2}\int Q^{2}\epsilon_{x}^{2}+\frac{5}{2}\int Q_{x}^{2}\epsilon^{2}+5\int QQ_{xx}\epsilon^{2}+\frac{15}{4}\int Q^{4}\epsilon^{2} \\
 & +c\bigg(\int\epsilon_{x}^{2}-3\int Q^{2}\epsilon^{2}\bigg)+c^{2}\frac{1}{2}\int\epsilon^{2}=:\mathcal{Q}_{0,\sqrt{c}}[\epsilon].\label{1eq:-274}
\end{align}

Firstly, we prove, by simple calculations, as in the previous section,
that $Q_{x}$ and $Q+xQ_{x}$ are in the kernel of this quadratic
form. It is easy to see, by asymptotic study that these two functions
are linearly independent.

The self-adjoint linear operator associated to this quadratic form
is: 
\begin{align}
\mathcal{L}_{c}^{s}[\epsilon] & :=\epsilon_{(4x)}-2c\epsilon_{xx}+c^{2}\epsilon+5Q^{2}\epsilon_{xx}+10QQ_{x}\epsilon_{x} \\
 & +\Big(5Q_{x}^{2}+10QQ_{xx}+\frac{15}{2}Q^{4}-6cQ^{2}\Big)\epsilon,\label{1eq:-275}
\end{align}
so that $\mathcal{Q}_{c}^{s}[\epsilon]=\int\epsilon\mathcal{L}_{c}^{s}[\epsilon]$.
$\mathcal{L}_{c}^{s}$ is a compact perturbation of the constant coefficients
operator: 
\begin{align}
\mathcal{M}[\epsilon]:=\epsilon_{(4x)}-2c\epsilon_{xx}+c^{2}\epsilon.\label{1eq:-276}
\end{align}
A direct analysis involving ODE shows that the null space of $\mathcal{M}$
is spawned by four linearly independent functions: 
\begin{align}
e^{\pm\sqrt{c}x},\quad xe^{\pm\sqrt{c}x}.\label{1eq:-277}
\end{align}
Among these four functions, there are only two $L^{2}$-integrable
ones in the semi-infinite line $[0,+\infty)$. Therefore, the null
space of $\mathcal{L}_{c}^{s}\vert_{H^{4}(\mathbb{R})}$ is spanned
by at most two $L^{2}$-functions.
Therefore, 
\begin{align}
\ker(\mathcal{L}_{c}^{s})=\Span(\partial_{x}Q,Q+x\partial_{x}Q).\label{1eq:-278}
\end{align}

\begin{lem}
The operator $\mathcal{L}_{c}^{s}$ does not have any negative eigenvalue. 
\end{lem}

\begin{proof}
$\mathcal{L}_{c}^{s}$ has 
\begin{align}
\sum_{x\in\mathbb{R}}\dim\ker W[Q_{x},Q+xQ_{x}](t,x)\label{1eq:-279}
\end{align}
negative eigenvalues, counting multiplicity, where $W$ is the Wronskian
matrix: 
\begin{align}
W[Q_{x},Q+xQ_{x}](t,x):=
\begin{bmatrix}
Q_{x} & Q+xQ_{x}\\
Q_{xx} & (Q+xQ_{x})_{x}
\end{bmatrix}
.\label{1eq:-280}
\end{align}
For this result, see \cite{key-32}, where the finite interval case
was considered. As shown in several articles \cite{key-33,key-24},
the extension to the real line is direct.

Thus, it is sufficient to see that $\det W[Q_{x},Q+xQ_{x}](t,x)$
is never zero. For this, let us simply calculate this determinant:
\begin{align}
Q_{x}(2Q_{x}+xQ_{xx})-(Q+xQ_{x})Q_{xx} & =2Q_{x}^{2}-QQ_{xx} \\
 & =2cQ^{2}-Q^{4}-Q(cQ-Q^{3}) \\
 & =cQ^{2}>0.\label{1eq:-281}
\end{align}
\end{proof}

\subsection{Coercivity of the quadratic form associated to a soliton}
\label{1sec:53}

For $Q=R_{c,\kappa}$, let
\begin{align}
\mathcal{Q}_{c}^{s}[\epsilon]: & =\frac{1}{2}\int\epsilon_{xx}^{2}-\frac{5}{2}\int Q^{2}\epsilon_{x}^{2}+\frac{5}{2}\int Q_{x}^{2}\epsilon^{2}+5\int QQ_{xx}\epsilon^{2}+\frac{15}{4}\int Q^{4}\epsilon^{2} \\
 & +c\bigg(\int\epsilon_{x}^{2}-3\int Q^{2}\epsilon^{2}\bigg)+c^{2}\frac{1}{2}\int\epsilon^{2}.\label{1eq:-282}
\end{align}

\begin{lem}
There exists $\mu_{c}>0$ such that for any $\epsilon\in H^{2}$ such
that $\int\epsilon Q=\int\epsilon Q_{x}=0$, we have that
\begin{align}
\mathcal{Q}_{c}^{s}[\epsilon]\geq \mu_{c}\lVert\epsilon\rVert_{H^{2}}^{2}.\label{1eq:-283}
\end{align}
\end{lem}

\begin{proof}
From Section \ref{1coer_sol}, we know that if $\int\epsilon\partial_{x}Q=\int\epsilon\partial_{c}Q=0$,
then, for a constant $\nu_{c}>0$, we have that
\begin{align}
\mathcal{Q}_{c}^{s}[\epsilon]\geq \nu_{c}\lVert\epsilon\rVert_{H^{2}}^{2}.\label{1eq:-284}
\end{align}

Let $\epsilon\in H^{2}$ be such that $\int\epsilon Q=\int\epsilon\partial_{x}Q=0$.
There exists $a\in\mathbb{R}$ and $\epsilon_{\bot}\in\Span(\partial_{x}Q,\partial_{c}Q)^{\bot}$
such that 
\begin{align}
\epsilon=a\partial_{c}Q+\epsilon_{\bot}.\label{1eq:-285}
\end{align}

From $\int\epsilon Q=0$, we have that 
\begin{align}
a\int\partial_{c}Q\cdot Q+\int\epsilon_{\bot}Q=0,\label{1eq:-286}
\end{align}
thus,
\begin{align}
\frac{a}{2}\int Q^{2}+\int\epsilon_{\bot}Q=0,\label{1eq:-287}
\end{align}
which allows us to derive:
\begin{align}
a=-2\frac{\int\epsilon_{\bot}Q}{\int Q^{2}}.\label{1eq:-288}
\end{align}

Because $\partial_{c}Q$ is in the kernel of $\mathcal{Q}_{c}^{s}$,
we have that
\begin{align}
\mathcal{Q}_{c}^{s}[\epsilon]=\mathcal{Q}_{c}^{s}[\epsilon_{\bot}]\geq \nu_{c}\lVert\epsilon_{\bot}\rVert_{H^{2}}^{2}.\label{1eq:-289}
\end{align}

Now, from 
\begin{align}
\epsilon=-2\frac{\int\epsilon_{\bot}Q}{\int Q^{2}}\partial_{c}Q+\epsilon_{\bot},\label{1eq:-290}
\end{align}
we have by triangular and Cauchy-Schwarz inequalities that 
\begin{align}
\lVert\epsilon\rVert_{H^{2}} & \leq \lVert\epsilon_{\bot}\rVert_{H^{2}}+2\frac{\big\lvert\int\epsilon_{\bot}Q\big\rvert}{\lVert Q\rVert_{L^{2}}^{2}}\lVert\partial_{c}Q\rVert_{H^{2}} \\
 & \leq \lVert\epsilon_{\bot}\rVert_{H^{2}}+2\frac{\lVert\partial_{c}Q\rVert_{H^{2}}}{\lVert Q\rVert_{L^{2}}}\lVert\epsilon_{\bot}\rVert_{L^{2}} \\
 & \leq \Big(1+2\frac{\lVert\partial_{c}Q\rVert_{H^{2}}}{\lVert Q\rVert_{L^{2}}}\Big)\lVert\epsilon_{\bot}\rVert_{H^{2}}.\label{1eq:-291}
\end{align}

Therefore, we may derive a constant $\mu_{c}$ (independent on $\epsilon$)
such that 
\begin{align}
\mathcal{Q}_{c}^{s}[\epsilon]\geq \mu_{c}\lVert\epsilon\rVert_{H^{2}}^{2}.\label{1eq:-292}
\end{align}
\end{proof}

\subsection{Coercivity with almost orthogonality conditions (to be used for the
uniqueness)}
\label{1sec:54}

For $B:=B_{\alpha,\beta}$ or any of its translations, we define the
canonical quadratic form associated to $B$: 
\begin{align}
\mathcal{Q}_{\alpha,\beta}^{b}[\epsilon]: & =\frac{1}{2}\int\epsilon_{xx}^{2}-\frac{5}{2}\int B^{2}\epsilon_{x}^{2}+\frac{5}{2}\int B_{x}^{2}\epsilon^{2}+5\int BB_{xx}\epsilon^{2}+\frac{15}{4}\int B^{4}\epsilon^{2} \\
 & +\big(\beta^{2}-\alpha^{2}\big)\bigg(\int\epsilon_{x}^{2}-3\int B^{2}\epsilon^{2}\bigg)+\big(\alpha^{2}+\beta^{2}\big)^{2}\frac{1}{2}\int\epsilon^{2},\label{1eq:-293}
\end{align}
and we know that $\partial_{x_{1}}B$ and $\partial_{x_{2}}B$ span
the kernel of $\mathcal{Q}_{\alpha,\beta}^{b}$. More precisely, there
exists $\mu_{\alpha,\beta}^{b}>0$ such that if $\epsilon$ is orthogonal
to $\partial_{x_{1}}B$ and $\partial_{x_{2}}B$, we have that
\begin{align}
\mathcal{Q}_{\alpha,\beta}^{b}[\epsilon]\geq \mu_{\alpha,\beta}^{b}\lVert\epsilon\rVert_{H^{2}}^{2}-\frac{1}{\mu_{\alpha,\beta}^{b}}\bigg(\int\epsilon B\bigg)^{2}.\label{1eq:-294}
\end{align}

We would like to prove the following lemma (adapted from the Appendix
A of \cite{key-4}): 
\begin{lem}
\label{1lem:prep}There exists $\nu:=\nu_{\alpha,\beta}^{b}>0$ such
that, for $\epsilon\in H^{2}(\mathbb{R})$, if 
\begin{align}
\bigg\lvert\int(\partial_{x_{1}}B_{\alpha,\beta})\epsilon\bigg\rvert +\bigg\lvert\int(\partial_{x_{2}}B_{\alpha,\beta})\epsilon\bigg\rvert <\nu\lVert\epsilon\rVert_{H^{2}},\label{1eq:-295}
\end{align}
then 
\begin{align}
\mathcal{Q}_{\alpha,\beta}^{b}[\epsilon]\geq \frac{\mu_{\alpha,\beta}^{b}}{4}\lVert\epsilon\rVert_{H^{2}}^{2}-\frac{4}{\mu_{\alpha,\beta}^{b}}\bigg(\int\epsilon B_{\alpha,\beta}\bigg)^{2},\label{1eq:-296}
\end{align}
where $B_{\alpha,\beta}$ denotes the breather of parameters $\alpha$ and $\beta$
or any of its translations (in space or in time). 
\end{lem}

\begin{proof}
Take $\nu>0$ (we will find a condition on $\nu$ later in the proof)
and take $\epsilon$ satisfying the assumption of the lemma. Then
(denoting $B=B_{\alpha,\beta}$) ,
\begin{align}
\epsilon=\epsilon_{1}+aB_{1}+bB_{2}=\epsilon_{1}+\epsilon_{2},\label{1eq:A1A2}
\end{align}
where $\int\epsilon_{1}B_{1}=\int\epsilon_{1}B_{2}=\int\epsilon_{1}\epsilon_{2}=0$.

By performing a $L^{2}$-scalar product of\textbf{ }(\ref{1eq:A1A2}) with
$B_{1}$ and $B_{2}$, we obtain, by assumption, that
\begin{align}
\bigg\lvert a\int B_{1}^{2}+b\int B_{1}B_{2}\bigg\rvert\leq \nu\lVert\epsilon\rVert_{H^{2}},\label{1eq:-297}
\end{align}
\begin{align}
\bigg\lvert a\int B_{1}B_{2}+b\int B_{2}^{2}\bigg\rvert\leq \nu\lVert\epsilon\rVert_{H^{2}}.\label{1eq:-298}
\end{align}
Therefore, by making linear combinations of these two inequalities, using
triangular and Cauchy-Schwarz inequalities, we obtain that
\begin{align}
\lvert a\rvert+\lvert b\rvert\leq  C\nu\lVert\epsilon\rVert_{H^{2}}.\label{1eq:-299}
\end{align}

We can take space derivatives of (\ref{1eq:A1A2}). And thus, we obtain,
for $\nu$ small enough, that
\begin{align}
\frac{1}{2}\lVert\epsilon\rVert_{H^{2}}^{2}\leq \lVert\epsilon_{1}\rVert_{H^{2}}^{2}\leq 2\lVert\epsilon\rVert_{H^{2}}^{2}.\label{1eq:-300}
\end{align}

Because of $\int BB_{1}=\int BB_{2}=0$, 
\begin{align}
\int\epsilon B=\int\epsilon_{1}B.\label{1eq:-301}
\end{align}
By bilinearity, 
\begin{align}
\mathcal{Q}_{\alpha,\beta}^{b}[\epsilon] & =\mathcal{Q}_{\alpha,\beta}^{b}[\epsilon_{1}]+\mathcal{Q}_{\alpha,\beta}^{b}[\epsilon_{2}]+\int\epsilon_{1,xx}\epsilon_{2,xx}-5\int B^{2}\epsilon_{1,x}\epsilon_{2,x}\\
&\quad+5\int B_{x}^{2}\epsilon_{1}\epsilon_{2}+10\int BB_{xx}\epsilon_{1}\epsilon_{2} +\frac{15}{2}\int B^{4}\epsilon_{1}\epsilon_{2}\\
 & \quad+\big(\beta^{2}-\alpha^{2}\big)\bigg(2\int\epsilon_{1,x}\epsilon_{2,x}-6\int B^{2}\epsilon_{1}\epsilon_{2}\bigg)+\big(\alpha^{2}+\beta^{2}\big)^{2}\int\epsilon_{1}\epsilon_{2}\label{1eq:-302}
\end{align}

We know from the coercivity of $\mathcal{Q}_{\alpha,\beta}^{b}$ that
\begin{align}
\mathcal{Q}_{\alpha,\beta}^{b}[\epsilon_{1}] & \geq \mu_{\alpha,\beta}^{b}\lVert\epsilon_{1}\rVert_{H^{2}}^{2}-\frac{1}{\mu_{\alpha,\beta}^{b}}\bigg(\int\epsilon_{1}B\bigg)^{2} \\
 & \geq \frac{\mu_{\alpha,\beta}^{b}}{2}\lVert\epsilon\rVert_{H^{2}}^{2}-\frac{2}{\mu_{\alpha,\beta}^{b}}\bigg(\int\epsilon B\bigg)^{2}.\label{1eq:-303}
\end{align}

Also, if we denote by $\mathcal{L}_{\alpha,\beta}^{b}$ the self-adjoint
operator associated to the quadratic form $\mathcal{Q}_{\alpha,\beta}^{b}$,
\begin{align}
\mathcal{Q}_{\alpha,\beta}^{b}[\epsilon_{2}] & =a^{2}\mathcal{Q}_{\alpha,\beta}^{b}[B_{1}]+b^{2}\mathcal{Q}_{\alpha,\beta}^{b}[B_{2}]+2ab\int\mathcal{L}_{\alpha,\beta}^{b}[B_{1}]B_{2} \leq  C\nu^{2}\lVert\epsilon\rVert_{H^{2}}^{2}.\label{1eq:-304}
\end{align}
Actually, in this case, $\mathcal{Q}_{\alpha,\beta}^{b}[\epsilon_{2}]=0$,
because $\epsilon_{2}$ is in the kernel of $\mathcal{Q}_{\alpha,\beta}^{b}$
(but, when we adapt this proof for solitons, we can only write the
bound).

Now, we recall that $\int\epsilon_{1}\epsilon_{2}=0$, and study
the other terms by using Cauchy-Schwarz: 
\begin{align}
\MoveEqLeft\bigg\lvert\int\epsilon_{1,xx}\epsilon_{2,xx}-5\int B^{2}\epsilon_{1,x}\epsilon_{2,x}+5\int B_{x}^{2}\epsilon_{1}\epsilon_{2}+10\int BB_{xx}\epsilon_{1}\epsilon_{2} \\
&\quad+\frac{15}{2}\int B^{4}\epsilon_{1}\epsilon_{2}+\big(\beta^{2}-\alpha^{2}\big)\bigg(2\int\epsilon_{1,x}\epsilon_{2,x}-6\int B^{2}\epsilon_{1}\epsilon_{2}\bigg)\bigg\rvert\\
 & \leq  C\big(\lvert a\rvert+\lvert b\rvert\big)\lVert\epsilon_{1}\rVert_{H^{2}}\leq  C\nu\lVert\epsilon\rVert_{H^{2}(\mathbb{R})}^{2}.\label{1eq:-305}
\end{align}
We observe that if we take $\nu$ small enough, the claim of the
lemma is proved. 
\end{proof}
We prove in the same way that we have a similar lemma for solitons: 
\begin{lem}
\label{1lem:prepsol}There exists $\nu:=\nu_{c}^{s}>0$, such that,
for $\epsilon\in H^{2}(\mathbb{R})$, if 
\begin{align}
\bigg\lvert\int(\partial_{c}R_{c,\kappa})\epsilon\bigg\rvert +\bigg\lvert\int(\partial_{x}R_{c,\kappa})\epsilon\bigg\rvert\leq \nu\lVert\epsilon\rVert_{H^{2}},\label{1eq:-306}
\end{align}
then 
\begin{align}
\mathcal{Q}_{c}^{s}[\epsilon]\geq \frac{\mu_{c}^{s}}{4}\lVert\epsilon\rVert_{H^{2}}^{2},\label{1eq:-307}
\end{align}
where $R_{c,\kappa}$ denotes the soliton of parameter $c$ and sign
$\kappa$ or any of its translations. 
\end{lem}

And even,
\begin{lem}
There exists $\nu:=\nu_{c}^{s}>0$, such that, for $\epsilon\in H^{2}(\mathbb{R})$,
if
\begin{align}
\bigg\lvert\int R_{c,\kappa}\epsilon\bigg\rvert +\bigg\lvert\int(\partial_{x}R_{c,\kappa})\epsilon\bigg\rvert\leq \nu\lVert\epsilon\rVert_{H^{2}},\label{1eq:-308}
\end{align}
then
\begin{align}
\mathcal{Q}_{c}^{s}[\epsilon]\geq \frac{\mu_{c}^{s}}{4}\lVert\epsilon\rVert_{H^{2}}^{2},\label{1eq:-309}
\end{align}
where $R_{c,\kappa}$ denotes the soliton of parameter $c$ and sign
$\kappa$ or any of its translations.
\end{lem}

\subsection{Computations for the third localized integral (to be used for the
uniqueness)}
\label{1sec:55}
\begin{lem}
Let $f:\mathbb{R}\rightarrow\mathbb{R}$ be a $C^{3}$ function that
do not depend on time and $u$ a solution of \eqref{1mKdV}. Then,
\begin{align}
\MoveEqLeft\frac{d}{dt}\int\Big(\frac{1}{2}u_{xx}^{2}-\frac{5}{2}u^{2}u_{x}^{2}+\frac{1}{4}u^{6}\Big)f \\
 & =\int\Big(-\frac{3}{2}u_{xxx}^{2}+9u_{xx}^{2}u^{2}+15u_{x}^{2}uu_{xx}+\frac{9}{16}u^{8}+\frac{1}{4}u_{x}^{4}+\frac{3}{2}u_{xx}u^{5} \\
 & \qquad-\frac{45}{4}u^{4}u_{x}^{2}\Big)f'+5\int u^{2}u_{x}u_{xx}f''+\frac{1}{2}\int u_{xx}^{2}f'''.\label{1eq:-310}
\end{align}
\end{lem}

\begin{proof}
We perform by doing integrations by parts when needed and basic calculations:

\begin{align}
\MoveEqLeft\frac{d}{dt}\int\Big(\frac{1}{2}u_{xx}^{2}-\frac{5}{2}u^{2}u_{x}^{2}+\frac{1}{4}u^{6}\Big)f\\
&=\int u_{txx}u_{xx}f-5\int u_{t}uu_{x}^{2}f-5\int u^{2}u_{tx}u_{x}f+\frac{3}{2}\int u_{t}u^{5}f \\
 & =-\int\big(u_{xx}+u^{3}\big)_{xxx}u_{xx}f+5\int\big(u_{xx}+u^{3}\big)_{x}uu_{x}^{2}f \\ & +5\int u^{2}\big(u_{xx}+u^{3}\big)_{xx}u_{x}f-\frac{3}{2}\int\big(u_{xx}+u^{3}\big)_{x}u^{5}f \\
 & =\int\big(u_{xx}+u^{3}\big)_{xx}u_{xxx}f+\int\big(u_{xx}+u^{3}\big)_{xx}u_{xx}f'+5\int\big(u_{xx}+u^{3}\big)_{x}uu_{x}^{2}f \\
 & \quad+5\int u^{2}\big(u_{xx}+u^{3}\big)_{xx}u_{x}f-\frac{3}{2}\int\big(u_{xx}+u^{3}\big)_{x}u^{5}f \\
 & =-\frac{1}{2}\int u_{xxx}^{2}f'+\int\big(u^{3}\big)_{xx}u_{xxx}f+\int\big(u_{xx}+u^{3}\big)_{xx}u_{xx}f'+5\int u_{xxx}uu_{x}^{2}f+5\int\big(u^{3}\big)_{x}uu_{x}^{2}f \\
 & \quad+5\int u^{2}u_{xxxx}u_{x}f+5\int u^{2}\big(u^{3}\big)_{xx}u_{x}f-\frac{3}{2}\int u_{xxx}u^{5}f-\frac{3}{2}\int\big(u^{3}\big)_{x}u^{5}f \\
 & =-\frac{1}{2}\int u_{xxx}^{2}f'+\int\big(u_{xx}+u^{3}\big)_{xx}u_{xx}f'+\int\big(3u_{xx}u^{2}+6u_{x}^{2}u\big)u_{xxx}f+5\int u_{xxx}uu_{x}^{2}f \\
 & \quad+15\int u_{x}^{3}u^{3}f+5\int u^{2}u_{xxxx}u_{x}f+5\int u^{2}\big(3u_{xx}u^{2}+6u_{x}^{2}u\big)u_{x}f-\frac{3}{2}\int u_{xxx}u^{5}f-\frac{9}{2}\int u_{x}u^{7}f \\
 & =-\frac{1}{2}\int u_{xxx}^{2}f'+\int\big(u_{xx}+u^{3}\big)_{xx}u_{xx}f'+3\int u^{2}u_{xx}u_{xxx}f+5\int u^{2}u_{xxxx}u_{x}f \\
 & \quad+11\int uu_{x}^{2}u_{xxx}f+45\int u^{3}u_{x}^{3}f+15\int u^{4}u_{x}u_{xx}f-\frac{3}{2}\int u_{xxx}u^{5}f+\frac{9}{16}\int u^{8}f' \\
 & =-\frac{1}{2}\int u_{xxx}^{2}f'+\int\big(u_{xx}+u^{3}\big)_{xx}u_{xx}f'+\frac{9}{16}\int u^{8}f'-2\int u^{2}u_{xx}u_{xxx}f \\
 & \quad+\int uu_{x}^{2}u_{xxx}f-5\int u^{2}u_{x}u_{xxx}f'+45\int u^{3}u_{x}^{3}f+15\int u^{4}u_{x}u_{xx}f-\frac{3}{2}\int u^{5}u_{xxx}f \\
 & =-\frac{1}{2}\int u_{xxx}^{2}f'+\int\big(u_{xx}+u^{3}\big)_{xx}u_{xx}f'+\frac{9}{16}\int u^{8}f'-5\int u^{2}u_{x}u_{xxx}f'-\int u^{2}\big(u_{xx}^{2}\big)_{x}f \\
 & \quad+\int uu_{x}^{2}u_{xxx}f+45\int u^{3}u_{x}^{3}f+15\int u^{4}u_{x}u_{xx}f-\frac{3}{2}\int u^{5}u_{xxx}f \\
 & =-\frac{1}{2}\int u_{xxx}^{2}f'+\int\big(u_{xx}+u^{3}\big)_{xx}u_{xx}f'+\frac{9}{16}\int u^{8}f'-5\int u^{2}u_{x}u_{xxx}f' \\ & +\int u^{2}u_{xx}^{2}f'+2\int uu_{x}u_{xx}^{2}f
  -\int u_{x}^{3}u_{xx}f-2\int uu_{x}u_{xx}^{2}f \\ & \quad-\int uu_{x}^{2}u_{xx}f'+45\int u^{3}u_{x}^{3}f+15\int u^{4}u_{x}u_{xx}f-\frac{3}{2}\int u^{5}u_{xxx}f \\
 & =-\frac{1}{2}\int u_{xxx}^{2}f'+\int\big(u_{xx}+u^{3}\big)_{xx}u_{xx}f'+\frac{9}{16}\int u^{8}f'-5\int u^{2}u_{x}u_{xxx}f'+\int u^{2}u_{xx}^{2}f' \\
 & \quad-\int uu_{x}^{2}u_{xx}f'-\frac{1}{4}\int\big(u_{x}^{4}\big)_{x}f+45\int u^{3}u_{x}^{3}f+\frac{45}{4}\int u^{4}\big(u_{x}^{2}\big)_{x}f+\frac{3}{2}\int u^{5}u_{xx}f' \\
 & =-\frac{1}{2}\int u_{xxx}^{2}f'+\int\big(u_{xx}+u^{3}\big)_{xx}u_{xx}f'+\frac{9}{16}\int u^{8}f'-5\int u^{2}u_{x}u_{xxx}f'+\int u^{2}u_{xx}^{2}f' \\
 & \quad-\int uu_{x}^{2}u_{xx}f'+\frac{1}{4}\int u_{x}^{4}f'+\frac{3}{2}\int u^{5}u_{xx}f'+45\int u^{3}u_{x}^{3}f-45\int u^{3}u_{x}^{3}f-\frac{45}{4}\int u^{4}u_{x}^{2}f' \\
 & =-\frac{3}{2}\int u_{xxx}^{2}f'-\int u_{xxx}u_{xx}f''+4\int u_{xx}^{2}u^{2}f'+5\int u_{x}^{2}uu_{xx}f'-5\int u^{2}u_{x}u_{xxx}f' \\
 & \quad+\frac{9}{16}\int u^{8}f'+\frac{1}{4}\int u_{x}^{4}f'+\frac{3}{2}\int u^{5}u_{xx}f'-\frac{45}{4}\int u^{4}u_{x}^{2}f' \\
 & =-\frac{3}{2}\int u_{xxx}^{2}f'+9\int u_{xx}^{2}u^{2}f'+15\int u_{x}^{2}uu_{xx}f'+\frac{9}{16}\int u^{8}f'+\frac{1}{4}\int u_{x}^{4}f' \\
 & \quad+\frac{3}{2}\int u^{5}u_{xx}f'-\frac{45}{4}\int u^{4}u_{x}^{2}f'-\int u_{xxx}u_{xx}f''+5\int u^{2}u_{x}u_{xx}f'' \\
 & =\int\big(-\frac{3}{2}u_{xxx}^{2}+9u_{xx}^{2}u^{2}+15u_{x}^{2}uu_{xx}+\frac{9}{16}u^{8}+\frac{1}{4}u_{x}^{4}+\frac{3}{2}u_{xx}u^{5}-\frac{45}{4}u^{4}u_{x}^{2}\big)f' \\
 & \quad+5\int u^{2}u_{x}u_{xx}f''+\frac{1}{2}\int u_{xx}^{2}f'''.\label{1eq:-311}
\end{align}
which is exactly the desired expression.
\end{proof}

\newpage

\bibliography{bib}
\bibliographystyle{siam}

\begin{center}
IRMA, UMR 7501, Université de Strasbourg, CNRS, F-67000 Strasbourg, France

semenov@math.unistra.fr

\end{center}

\end{document}